\documentclass[oneside,english, 11pt]{amsart}
\usepackage[T1]{fontenc}
\usepackage[utf8]{inputenc}
\usepackage{enumitem}
\usepackage{amsthm}
\usepackage{amstext}
\usepackage{amssymb}
\usepackage{amsmath}
\usepackage{mathtools}
\usepackage{xcolor}%
\usepackage[mathscr]{eucal}

\usepackage{vmargin}

\setpapersize{A4}
\setmargins{3.0cm}       
{1.5cm}                        
{14.5cm}                      
{23.42cm}                    
{10pt}                           
{1cm}                           
{0pt}                             
{2cm}

\makeatletter
\numberwithin{equation}{section}
\numberwithin{figure}{section}
\theoremstyle{plain}
\newtheorem{thm}{\protect\theoremname}
\theoremstyle{plain}
\newtheorem{cor}{\protect\corollaryname}
  \theoremstyle{plain}
  \newtheorem{prop}[thm]{\protect\propositionname}
  \theoremstyle{remark}
  \newtheorem{rem}[thm]{\protect\remarkname}
  \theoremstyle{definition}
  \newtheorem{defn}[thm]{\protect\definitionname}
  \theoremstyle{plain}
  \newtheorem{lem}[thm]{\protect\lemmaname}
    \theoremstyle{definition}
  \newtheorem{example}[thm]{\protect\examplename}

\theoremstyle{remark}

\DeclareMathOperator{\re}{Re}
\DeclareMathOperator{\im}{Im}

\DeclareMathOperator{\id}{id}
\DeclareMathOperator{\diff}{Diff}
\DeclareMathOperator{\sing}{Sing}

\DeclareMathOperator{\gl}{GL}
\DeclareMathOperator{\spn}{Span}

\DeclareMathOperator{\dist}{dist}

\usepackage[english]{babel}
  \providecommand{\definitionname}{Definition}
  \providecommand{\lemmaname}{Lemma}
  \providecommand{\corollaryname}{Corollary}
  \providecommand{\propositionname}{Proposition}
  \providecommand{\remarkname}{Remark}
\providecommand{\theoremname}{Theorem}
\providecommand{\examplename}{Example}
\providecommand{\assertionname}{Assertion}



\author{Maycol Falla Luza}
\author{Rudy Rosas}
\thanks{The first author was partially supported by CAPES-COFECUB Ma932/19. The second author was supported by the Vicerrectorado the Investigaci\'on de la Pontificia Universidad Cat\'olica del Per\'u}

\begin{document}

\title{Distributions, first integrals and Legendrian foliations}

\maketitle
\begin{abstract} We study germs of holomorphic distributions with ``separated variables''. 
 In codimension one,
a well know example of this kind of distribution  is given  by
\begin{equation}\nonumber 
dz=(y_1dx_1-x_1dy_1)+\dots+(y_mdx_m-x_mdy_m),
\end{equation}which defines the canonical contact structure on $\mathbb{CP}^{2m+1}$. Another
example is the Darboux distribution
\begin{equation}\label{contact2}\nonumber
dz=x_1dy_1+\dots+x_mdy_m, 
\end{equation} which gives the normal local form of any contact structure. Given a germ $\mathcal{D}$
of holomorphic distribution with separated variables in $(\mathbb{C}^n,0)$, we show that there exists , for some $\kappa \in \mathbb{Z}_{\geq 0}$ related to the Taylor coefficients of $\mathcal{D}$, a holomorphic submersion $$H_{\mathcal{D}}\colon (\mathbb{C}^n,0)\to (\mathbb{C}^{\kappa},0)$$ such that $\mathcal{D}$ is completely non-integrable on each level of $H_{\mathcal{D}}$. Furthermore, we show that there exists a holomorphic vector field $Z$ tangent to $\mathcal{D}$, such that 
each level of $H_{\mathcal{D}}$ contains a leaf of $Z$ that is somewhere dense in the level. In particular,
the field of meromorphic first integrals of $Z$ and that of $\mathcal{D}$ are the same. Between several other results, we show that the canonical contact structure on $\mathbb{CP}^{2m+1}$ supports a Legendrian holomorphic foliation whose generic leaves are dense in $\mathbb{CP}^{2m+1}$. 
So we obtain examples of injectively immersed Legendrian holomorphic open manifolds  that are everywhere dense. 

\end{abstract}
\section{Introduction} 
Let $M,N\in\mathbb{N}$, $M\ge 2$ and consider coordinates $x=(x_1,\dots,x_M)$ and $z=(z_1,\dots,z_N)$ in  $\mathbb{C}^{M}$  and $\mathbb{C}^{N}$, respectively. 
Let
 $\mathcal{D}$  be the distribution of dimension $M$  defined by the system of equations 
 \begin{equation}\label{mainsystem}
\begin{aligned}
dz_1&=\omega_1\\
dz_2&=\omega_2\\
&\;\;\vdots\\
dz_N&=\omega_N,
\end{aligned}
\end{equation}
where $\omega_1,\dots,\omega_N \in \Omega^1(\mathbb{C}^M,0)$, that is, these forms depend only on the variables $(x_1,\dots,x_M)$. An important example of this situation happens for $M=2m$, $N=1$ and  the distribution 
\begin{equation}\label{contact}
dz=(y_1dx_1-x_1dy_1)+\dots+(y_mdx_m-x_mdy_m),
\end{equation}
where we consider coordinates $(x_1,y_1,\dots,x_m,y_m, z)\in\mathbb{C}^{2m+1}$ --- this distribution defines the canonical contact structure on $\mathbb{CP}^{2m+1}$. Another example is given by the Darboux contact form
\begin{equation}\label{contact2}\nonumber
dz=x_1dy_1+\dots+x_mdy_m, 
\end{equation} which gives the normal local form of any contact structure, see \cite{Alarcon}, \cite{Darboux} or \cite{Godbillon}. 
 The purpose of this work is to study 
the integrability properties of the distribution $\mathcal{D}$  and their relationship with the dynamics of the holomorphic vector fields that are tangent to it.  To begin with, we point out  a situation where we easily find a holomorphic
 first integral for $\mathcal{D}$:  suppose that there exist $a_1,\dots,a_N\in\mathbb{C}$, not all zero, such that 
 $$a_1d\omega_1+\dots+a_Nd\omega_N=0.$$ Then, by Poincaré's Lemma, we find a holomorphic
  function $g$ in  $(\mathbb{C}^{M},0)$ with $g(0)=0$ and such that 
 $$dg=a_1\omega_1+\dots+a_N\omega_N.$$ Thus, if we set $$T(z)=a_1z_1+\dots +a_Nz_N,$$  we have
 $$d\left(T(z)-g(x)\right)=a_1(dz_1-\omega_1)+\dots+a_N(dz_N-\omega_N).$$ It follows from this equation 
 that  $\mathcal{D}$ is a subdistribution of $$d\left(T(z)-g(x)\right)=0,$$ 
 so the function $T(z)-g(x)$ is a holomorphic first integral of $\mathcal{D}$. This kind of first 
 integrals of $\mathcal{D}$ will be called
  \emph{elementary}.  Roughly speaking, our first result asserts that the elementary first integrals generate all the space of meromorphic first integrals of $\mathcal{D}$ in $(\mathbb{C}^{M+N},0)$.
 In order to give a precise statement of our results, we briefly introduce some notions. Firstly we define 
    \begin{align*}\omega &=(\omega_1,\dots,\omega_N),  \\
  d\omega &=(d\omega_1,\dots,d\omega_N), \textrm{ and }\\
  \int\limits_{\gamma}\omega &=\left(\int\limits_{\gamma} \omega_1,\dots,\int_{\gamma}\omega_N\right),
  \end{align*} where $\gamma$ is any piecewise smooth path in the domain of definition of $\omega$.
   We can express $d\omega$ as a series at $0\in \mathbb{C}^M$ in the form
 \begin{equation} \label{ecudo}d\omega=\sum{c_{K}^{ij}x^Kdx_idx_j},\end{equation} where $dx_i dx_j$ stands for $dx_i \wedge dx_j$, $c_{K}^{ij}\in\mathbb{C}^N$  
  and the summation extends  over $$K=(k_1,\dots,k_M)\in\left(\mathbb{Z}_{\ge 0}\right)^M\hspace{-0.2cm}, \; 1\le i< j\le M.$$
 Let $$W_\mathcal{D}\subset \mathbb{C}^{N}$$ be the complex vector space spanned by the coefficients 
 $c_{K}^{ij}$ and put $$ \kappa=\operatorname{codim} (W_\mathcal{D}).$$ If $\kappa\ge 1$,   take a linear (surjective) map 
 $$T=(T_1,\dots,T_\kappa)\colon \mathbb{C}^N\to\mathbb{C}^\kappa$$ such that $$\ker (T)=W_\mathcal{D}.$$ 
 The  arguments explained above allow us to find a holomorphic map 
 $$g=(g_1,\dots, g_\kappa)\colon(\mathbb{C}^M,0)\to (\mathbb{C}^\kappa,0)$$ such that the functions
  $$h_j=T_j(z)-g_j(x),\; j=1,\dots,\kappa$$ are elementary first integrals of  $\mathcal{D}$. Then the map 
 \begin{align}\label{defh}H_{\mathcal{D}}:=(h_1,\dots,h_\kappa)&\colon(\mathbb{C}^{M},0)\times \mathbb{C}^N\to (\mathbb{C}^\kappa,0)
  \end{align} is a holomorphic first integral of $\mathcal{D}$. Observe that $\partial_z H_{\mathcal{D}}=T$, 
  so that $H_{\mathcal{D}}$ is a submersion.
   If $\kappa=0$, that is, if $W_\mathcal{D}= \mathbb{C}^{N}$, we define $H_{\mathcal{D}}$ as the constant map from $\mathbb{C}^{M+N}$ to the trivial vector space $\mathbb{C}^0$. The holomorphic functions on 
   $(\mathbb{C}^0,0)$ are --- by convention ---  the constants. Our first result states that $H_{\mathcal{D}}$ is a primitive first integral of $\mathcal{D}$, in the sense that any other meromorphic first integral of $\mathcal{D}$ is a composition of $H_\mathcal{D}$ with some meromorphic function.
  
   \begin{thm} \label{teorema1} A germ $F$ of meromorphic function in 
 $(\mathbb{C}^{M+N},0)$ is a meromorphic first integral of $\mathcal{D}$ if and only if  $$F=f\circ H_{\mathcal{D}},$$ where $f$ is a germ of meromorphic function in $(\mathbb{C}^\kappa,0)$. In particular, if $W_\mathcal{D}=\mathbb{C}^N$, the distribution $\mathcal{D}$ has only constant meromorphic first integrals in $(\mathbb{C}^{M+N},0)$.
 \end{thm}
 \begin{rem}\label{merosub}As a direct consequence of Theorem \ref{teorema1} we have the following fact: 
  if $\mathcal{D}$ has a non-constant meromorphic first integral, then  $\mathcal{D}$ has a holomorphic first integral that is a submersion.
 \end{rem}
  
Let $F$ be   a meromorphic first integral of $\mathcal{D}$. Essentially, the proof of Theorem \ref{teorema1} is equivalent to showing that
  $F$ is constant along the levels of $H_{\mathcal{D}}$. This fact is a direct consequence of Theorem \ref{teorema2} below, which 
  guarantees the existence of holomorphic vector fields tangent to $\mathcal{D}$ having leaves that  are somewhere dense in the levels of $H_{\mathcal{D}}$. For ease of exposition, it will be convenient to adopt the following convention.
  
\begin{defn}Let $\mathcal{D}$ be a holomorphic distribution  on a complex manifold $\mathcal{V}$. Let $\mathcal{F}$ be a singular holomorphic foliation. We say that $\mathcal{F}$ is a Legendrian foliation for $\mathcal{D}$ if its leaves are tangent to $\mathcal{D}$. If $\dim\mathcal{F}=1$ and $\mathcal{F}$ is generated by a holomorphic vector field $Z$, we also say that $Z$ is Legendrian for $\mathcal{D}$. 
\end{defn}
   
If $\mathcal{D}$ is a contact structure on $\mathcal{V}$, the leaves of a Legendrian foliation $\mathcal{F}$ are injectively immersed  holomorphic isotropic varieties for $\mathcal{D}$. If in addition we have  $2\dim\mathcal{F}=\dim\mathcal{D}$, the leaves of $\mathcal{F}$ are  injectively immersed holomorphic Legendrian manifolds for $\mathcal{D}$. For more details on Legendrian varieties see for instance \cite{Arnold1}, \cite{Arnold2} and \cite{Buczynski}.
From now on and for the sake of simplicity, we  say holomorphic foliation to mean singular holomorphic foliation.

  \begin{thm}\label{teorema2} Let  $\mathcal{D}$ and $H_{\mathcal{D}}$ be defined by equations \ref{mainsystem} and \ref{defh} repectively. Then $\mathcal{D}$ supports
   a Legendrian holomorphic  vector field $Z$, defined on the domain of definition of $\mathcal{D}$, such that the following property holds:  given a neighborhood $\Delta$ of the origin in  $\mathbb{C}^{M+N}$,  
   there is a nonempty open set $\Delta^*\subset\Delta$  such that each leaf $L$
   of $Z|_{\Delta^*}$ is dense in the level of ${H_{\mathcal{D}}}|_{\Delta^*}$ containing $L$. In particular, if  $W_{\mathcal{D}}=\mathbb{C}^N$, the leaves 
   of $Z|_{\Delta^*}$  are dense in $\Delta^*$. Moreover, if $M=2$, there exists a hypersurface 
   $\mathcal{S}$ through $0\in \mathbb{C}^{2+N}$ invariant by $Z$   such that, given $\Delta$,  the set $\Delta^*$ can be chosen 
   in the form $\Delta^*=\mathcal{B}\backslash \mathcal{S}$  
   for some neighborhood $\mathcal{B}$  of the origin.
   
   \end{thm}
  The proof of this theorem is constructive and, as immediate consequence of this construction --- see Proposition \ref{log}, we have the following additional properties about the vector field $Z$.
  \begin{prop}\label{detalles}Under the same assumptions and notation of Theorem \ref{teorema2}, we have the following additional properties:
  \begin{enumerate}
  \item \label{detalle1}There exists a polynomial vector field $$X(x)=A_1(x)\frac{\partial}{\partial x_1}+\dots +A_M(x)\frac{\partial}{\partial x_M}$$ on $\mathbb{C}^M$, depending  only on a finite jet of the series \eqref{ecudo} of $d\omega$,  such that the vector field $Z$ is given by
  $$Z(x,z)=\sum\limits_{i=1}^MA_i(x)\frac{\partial}{\partial x_i}+
\sum\limits_{j=1}^N\big[\omega_j(x)\cdot X(x)\big]\frac{\partial}{\partial z_j}.$$ 
In particular, $Z$ depends only on the variables $(x_1,\dots,x_M)\in\mathbb{C}^M$. Here the dot product stands for the evaluation of a $1$-form on a vector field.

\item \label{detalle3}There exists a hypersurface $S$, which is a union of hyperplanes through the origin of $\mathbb{C}^M$, that is invariant by $X$ and such that the orbits of $X$ through $\mathbb{C}^M\backslash S$ are everywhere dense.  

\end{enumerate}
  \end{prop}

Observe that, if the forms $\omega_i$ are defined on a neighborhood $U$ of the origin in $\mathbb{C}^M$,   the distribution $\mathcal{D}$ is
defined on the set $$U\times \mathbb{C}^N.$$ Thus, by Proposition \ref{detalles}, the vector field $Z$ is also defined on the set 
$U\times \mathbb{C}^N$. Nevertheless, note that Theorem \ref{teorema2} establishes a local property of the vector field $Z$ near the origin, a fact which can be useful,
as we show with a little example.  Let $\eta$ be a holomorphic
1-form on a neighborhood of the origin in $\mathbb{C}^3$ such that $\eta\wedge d\eta (0)\neq 0$. By Darboux Theorem, the distribution $\eta=0$ near the origin
is equivalent to the one defined by $$dz=xdy.$$ Since this distribution is of type \eqref{mainsystem}, from 
Theorem \ref{teorema2} the following corollary follows. 
\begin{cor}\label{cor1}
Let $\eta$ be a holomorphic
1-form on a neighborhood of the origin in $\mathbb{C}^3$ such that $\eta\wedge d\eta (0)\neq 0$. Then the distribution $\eta=0$ supports a Legendrian holomorphic vector field on a neighborhood  of the 
origin in $\mathbb{C}^3$, having an invariant hypersurface $\mathcal{S}$ through the origin, such that the following property holds:  given any neighborhood $\Delta$ of the origin in $\mathbb{C}^3$, there exists a neighborhood $\Delta^*\subset\Delta$ of the origin in $\mathbb{C}^3$ such that any leaf of $Z$ through a point in $\Delta^*\backslash \mathcal{S}$ is dense in $\Delta^*$. 
\end{cor}

Now, we consider distributions $\mathcal{D}$ that are  globally defined. If such a distribution is completely non-integrable, we show that it supports Legendrian holomorphic one-dimensional foliations with dense generic leaves.
  \begin{thm}\label{teorema3} Assume that the holomorphic 1-forms $\omega_i$ in the system  \eqref{mainsystem}
extend as meromorphic 1-forms on $\mathbb{C}^M$.  Let $(\omega)_\infty$ be the union of pole sets of the $\omega_i$, let $Z$  
be as given by Theorem \ref{teorema2}, and $S$ as given by Proposition \ref{detalles}. Then we have the following properties:
\begin{enumerate}
\item \label{item111} $Z$ extends as a meromorphic vector field on
$\mathbb{C}^{M+N}$. If the forms $\omega_i$
  are rational (resp. polynomial), the vector field $Z$ is also rational (resp. polynomial). 
  \item \label{item222} If  $W_{\mathcal{D}}=\mathbb{C}^N$, then  each leaf of $Z$ passing through a point
  $$(x,z)\in\mathbb{C}^{M+N},\; x\notin S\cup (\omega)_\infty$$ is dense in $\mathbb{C}^{M+N}$. 
  \end{enumerate}
  \end{thm}
 As an easy consequence we have the following corollary.
\begin{cor}\label{isotrop} Consider coordinates $(x_1,\dots,x_M,z)$ in $\mathbb{C}^{M+1}$.
Let $\mathcal{D}$ be the distribution defined  by the equation
$$dz=\omega,$$ where $\omega$ is a rational non-closed 1-form in the variables
$(x_1,\dots,x_M)$. Then $\mathcal{D}$ supports a Legendrian polynomial  vector field whose generic leaf is dense in
 $\mathbb{C}^{M+1}$. In particular,
    if $M=2m$ and $\xi$ is the contact distribution on $\mathbb{CP}^{2m+1}$ defined by Equation \eqref{contact}, then
     $\xi$ supports a Legendrian one-dimensional holomorphic foliation whose generic leaf is dense in 
     $\mathbb{CP}^{2m+1}$. 
  \end{cor}
 
In relation with the contact structure $\xi$ defined by Equation \eqref{contact}, the leaves of the foliation given by Corollary \ref{isotrop} are examples of  injectively immersed isotropic open holomorphic curves that are everywhere dense, although Corollary \ref{isotrop} does not specify  the analytic type of these curves. Recently, in \cite[Corollary 6.11]{Alarcon2} the authors construct examples of  everywhere dense injectively immersed Legendrian open curves
of arbitrary analytic type in odd dimensional complex projective spaces.
 A slightly more elaborated application of our results allow us to obtain examples of  Legendrian $m$-dimensional
  holomorphic  foliation for $\xi$  whose generic leaves are dense  in $\mathbb{CP}^{2m+1}$ --- so we obtain
    examples of injectively immersed Legendrian holomorphic open manifolds for $\xi$ that are everywhere dense.

   \begin{thm}\label{legendrian} Let  $\xi $ be the canonical contact distribution on $\mathbb{CP}^{2m+1}$ given by Equation \eqref{contact}.  Then $\xi$ supports a Legendrian  $m$-dimensional holomorphic foliation whose generic leaf is dense in $\mathbb{CP}^{2m+1}$. 
  \end{thm}
 
It is worth pointing out the relationship between our results and some classical connectivity properties of distributions in the real case.   The problem of local connectivity by curves tangent to a distribution was studied initially in the real case by Carathéodory (codimension one) and by Chow (any codimension), see for instance \cite{Chow} and \cite{Gromov}. Precisely, we have the following theorem.

\begin{thm}[Chow Connectivity Theorem, \cite{Chow}, \cite{Gromov}]
Let $X_1, \dots, X_m$ be $C^{\infty}$ vector fields on a connected manifold $V$, such that successive commutators of these fields span each tangent space $T_v V$, $v \in V$. Then every two points on $V$ can be joined by a piecewise smooth curve in $V$ where each piece is a segment of an integral curve of one of the fields $X_i$. 
\end{thm}
  
As a direct consequence we have the total connectivity property by curves tangent to a distribution  $\mathcal{D}$, provided it is completely non-integrable  --- in the sense that vector fields tangent to $\mathcal{D}$ Lie-generate the full tangent space. We will see in Remark \ref{totally-non-integrable} that, for a distribution given by a system of the form \eqref{mainsystem}, this is exactly the case when $H_{\mathcal{D}}$ is a constant function.

With the aid of Chow Connectivity Theorem we can successively connect an arbitrary sequence of points; in this way we can 
construct dense piecewise smooth curves that are tangent to a completely non-integrable distribution $\mathcal{D}$. 
As can be expected, there are smoothing methods which allow us to deform  such  piecewise smooth curves to obtain smooth curves also tangent to the distribution  --- see \cite{Gromov}. Since, in general,  the $C^\infty$ smoothing methods have no parallel in the holomorphic class,
the existence of dense holomorphic curves tangent to completely  non-integrable distributions is
a more complicated problem, which will not be a direct consequence of a holomorphic version 
of the Connectivity Theorem --- 
for a holomorphic version of the Connectivity Theorem in $\mathbb{C}^3$ we refer to \cite{Zorich}. Even more, we can look for the existence of holomorphic vector fields tangent to a totally non-integrable distribution with dense generic leaves. In this direction, Theorem \ref{teorema2} and Corollaries \ref{cor1} and \ref{isotrop} can be regarded as partial
 positive answers.

\section{Lifting vector fields} \label{seccionlifting}
Consider the canonical projection $$\pi_1\colon\mathbb{C}^{M+N}\to\mathbb{C}^{M}.$$  Let $X$ be a holomorphic vector field  on a neighborhood $U$ of the origin in $\mathbb{C}^M$. We assume that $U$ is contained in the domain of definition of the 1-forms $\omega_i$ defining 
the distribution $\mathcal{D}$. The lifting of $X$ to the distribution $\mathcal{D}$ is the vector field $X^{\mathcal{D}}$ on 
$U\times\mathbb{C}^{N}$ that is tangent to $\mathcal{D}$ and 
satisfies the equation 
$$d\pi_1\cdot X^{\mathcal{D}}=X.$$ If $$X=A_1\frac{\partial}{\partial x_1}+\dots+ A_M\frac{\partial}{\partial x_M},$$ the lifting
 $X^{\mathcal{D}}$  is  explicitly given by 
$$X^{\mathcal{D}}=\sum\limits_{i=1}^MA_i(\pi_1)\frac{\partial}{\partial x_i}+
\sum\limits_{j=1}^N\big[\omega_j(\pi_1)\cdot X(\pi_1)\big]\frac{\partial}{\partial z_j}.$$  If $\mathcal{F}$ is the holomorphic foliation defined by $X$, we say that the holomorphic foliation defined by $X^{\mathcal{D}}$ is the lifting of $\mathcal{F}$ to $\mathcal{D}$ and denote it by ${\mathcal{F}}^{\mathcal{D}}$ 
--- this foliation is a Legendrian one-dimensional holomorphic foliation for $\mathcal{D}$

The holomorphic vector $Z$ given by Theorem \ref{teorema2} will be constructed as the lifting of some vector field $X$ as above. Thus, the main task of our work is the study of  foliations that are obtained by  the lifting process. This kind of foliations 
defines holonomy maps which are strongly related with the values obtained by the integration of the 
forms $\omega_i$ along loops in $\mathbb{C}^M$, as we explain in the rest of the section.
\subsection*{Holonomy}
Let $\mathcal{F}$ be as above. 
Fix a neighborhood $\Delta=\Delta_1\times\Delta_2$ of $0\in\mathbb{C}^{M+N}$,  where $\Delta_1$ and $\Delta_2$ are --- respectively --- Euclidean balls centered at the origins of $0\in\mathbb{C}^M$ and
 $0\in\mathbb{C}^{N}$, and assume that $\mathcal{F}$ and $\mathcal{D}$ are defined on $\Delta$ --- in this case 
 $\mathcal{F}$ and $\mathcal{D}$ are actually defined on $\Delta_1\times\mathbb{C}^N$, as we have pointed out in the previous section.
  Let $p\in\Delta_1\backslash \sing(\mathcal{F})$ and consider the fiber 
 $\Delta_p:=\{p\}\times\Delta_2$.  Take any
 $\zeta\in \Delta_p$ and denote by $\mathcal{L}$ the leaf of 
 $\mathcal{F}^{\mathcal{D}}|_\Delta$
 passing through $\zeta$. We are interested in the set of "returns" of $\mathcal{L}$ to the fiber $\Delta_p$, that is, we want to study the set $\mathcal{L}\cap\Delta_p$.  A way to obtaining points in this set is by lifting some loops in $\Delta_1$ based at $p$, as we describe below. 
 Let $\gamma\colon [0,1]\to\Delta_1$ be a piecewise path with $\gamma(0)=p$. Then there exists a unique piecewise 
  path $\tilde{\gamma}$ starting at $\zeta\in\Delta_p$, tangent to $\mathcal{D}$ and such that  $\pi_1\circ \tilde{\gamma}=\gamma$. In fact, if  $\zeta=(p,z)$ and 
 $$\gamma(t)=(x_1(t),\dots,x_M(t))$$ for $t\in[0,1]$, we have 
 $$\tilde{\gamma}(t)=(x_1(t),\dots,x_M(t), z_1(t),\dots,z_N(t)),$$ where  the functions $z_1,\dots,z_N$ are the solutions of the system 
 \begin{align}
\nonumber z_1'&=\omega_1(\gamma)\cdot\gamma'\\
&\;\;\vdots \label{system}\\
z_N'&=\omega_N(\gamma)\cdot\gamma'\nonumber
\end{align}
with the initial condition $$(z_1(0),\dots,z_N(0))=z.$$
 Now, suppose that $\gamma$  is contained in the leaf of $\mathcal{F}$ through $p$. Then $$\gamma'(t)=\theta(t)X(\gamma(t))$$ for some 
 piecewise function $\theta$ and, from System \eqref{system} and the definition of $X^\mathcal{D}$, we deduce that $\tilde{\gamma}$ is contained in 
 the leaf of $\mathcal{F}^\mathcal{D}$ through $\zeta$. In principle, the path $\tilde{\gamma}$ is contained in 
 $\Delta_1\times\mathbb{C}^N$, it is not necessarily contained in $\Delta$. If $\gamma$ were a loop based at $p$ and we had $\tilde{\gamma}\subset \Delta$,  then the ending point of $\tilde{\gamma}$ would be a return of $\mathcal{L}$ to the fiber $\Delta_p$.  Lemma \ref{levantamiento} below gives us
 an elementary  criteria to guarantee
  that $\tilde{\gamma}$ is contained in $\Delta$. Recall that  the forms $\omega_i$ can be considered as an $N$-tuple of forms 
  \begin{equation}\label{multiforma}\omega =(\omega_1,\dots,\omega_N).\end{equation}
 Given any $\xi\in\Delta$, we can see $\omega(\xi)$ as a linear map $\mathbb{C}^M\to\mathbb{C}^N$, so we  consider its canonical norm $|\omega(\xi)|$. Moreover, we denote by $\ell(\gamma)$ the Euclidean length of $\gamma$.
 
    \begin{lem}\label{levantamiento} Suppose that there exists $ \mathfrak{K}>0$ such that $|\omega|\le \mathfrak{K}$ on
     $$\Delta=\Delta_1\times\Delta_2.$$ Let $\gamma\subset \Delta_1$ be a piecewise path starting at $p$. Consider a point $\zeta=(p,z)$ in $\Delta$ and let $r>0$ be the radius of $\Delta_2$. Then, if $$|z|+ \mathfrak{K}\ell(\gamma)<r,$$ the lifting of $\gamma$ to $\mathcal{D}$ starting at $\zeta$ is contained in $\Delta$.
  \end{lem}
  \begin{proof} The lifting $\tilde{\gamma}$  of $\gamma$ to $\mathcal{D}$ starting at $\zeta$ has the form 
  $$\tilde{\gamma}(t)=(\gamma(t),\mathfrak{z}(t)),$$ where $\mathfrak{z}(t)=(z_1(t),\dots,z_M(t))$ satisfies System \eqref{system} with $\mathfrak{z}(0)=z$. Then $$\mathfrak{z}'=\omega(\gamma(t))\cdot \gamma'(t),$$
  so that, for any $s\in [0,1]$  we have  \begin{align*} \mathfrak{z}(s)- z&=\int\limits_{0}^s\omega(\gamma(t))\cdot \gamma'(t) dt.
  \end{align*} Therefore 
  \begin{align*}|\mathfrak{z}(s)| &\le |z|+\left|\int\limits_{0}^s\omega(\gamma(t))\cdot \gamma'(t) dt\right|\\
  &\le |z|+\mathfrak{K}\ell(\gamma)<r,
  \end{align*}  so $\tilde{\gamma}(s)$ belongs to $\Delta$ for all $s\in[0,1]$. 
  \end{proof}

  By Lemma \ref{levantamiento}, we conclude that, in order to obtain returns of the leaf $\mathcal{L}$ to the fiber $\Delta_p$, it suffices to lift suitable small loops $\gamma$  that are contained in the leaf of $\mathcal{F}|_{\Delta_1}$ through $p$.  
 \subsection*{Non-trivial returns and main ideas} 
Let $\gamma$ be a loop based at $p$, contained in the leaf of $\mathcal{F}|_{\Delta_1}$ through $p$ and such that
the lifting of $\gamma$ to $\mathcal{D}$ is contained in $\Delta$ --- so it defines a return of $\mathcal{L}$ to $\Delta_p$. Naturally, we are interested in
foliations $\mathcal{F}$ that are able to generate many returns of $\mathcal{L}$ to $\Delta_p$ that are different from $\zeta$ itself. Thus, 
as a first step, we want to estimate the difference --- associated to the loop $\gamma$ --- between the starting point $\zeta=(p,z)$ an its return
 $\zeta^\gamma:=(p,z^\gamma)$ to $\Delta_p$ associated to
the loop $\gamma$. From 
System \eqref{system} we easily obtain that 
$$z^\gamma-z=\int\limits_{\gamma} \omega.$$
Thus,  our general objective is to study the integral of $\omega$ along loops $\gamma$ that are tangent to some holomorphic foliation $\mathcal{F}$. 
If $d\omega=0$, we clearly obtain $\int_{\gamma}\omega=0$ and we have no returns different from $\zeta$. This allows us to perceive the role played
by the nonzero coefficients of $d\omega$ in its series expansion. On the other hand,
if $\gamma$ were the boundary of a compact Riemann surface --- with boundary --- $S$, by Stokes Theorem we would obtain
$$\int\limits_{\gamma}\omega=\int\limits_{\gamma}\omega|_S=\int\limits_{S}d(\omega|_S)=0,$$ and this will happen even if the surface $S$ is singular. Thus, in order to obtain 
 $\int_{\gamma}\omega\neq 0$ we have to take $\gamma$ and $\mathcal{F}$ such that:
\begin{enumerate}
\item $\gamma$ is not homotopically null in its leaf; otherwise $\gamma$ is the boundary of a complex disc in its leaf.
\item $\gamma$ is not the fundamental loop in a separatrix of $\mathcal{F}$; otherwise $\gamma$ is the boundary of a singular disc.
\end{enumerate}
Given $\omega$ with $d\omega\neq 0$, we illustrate how we can obtain $\int_{\gamma}\omega\neq 0$ by looking for $\mathcal{F}$ 
in the class of linear foliations. For the sake of clarity, we suppose $M=2$.  Then we assume that $\mathcal{F}$ is defined by a linear vector field on
$\mathbb{C}^2$. Since we need many loops tangent to $\mathcal{F}$ that are homotopically non trivial in its corresponding leaves, we take 
$\mathcal{F}$ being generated by the system 
\begin{equation}\label{linearsystem}
\begin{aligned}
 x_1'&=&m_1x_1\\
x_2'&=&-m_2x_2,
\end{aligned}
\end{equation}  
where $m_1$ and $m_2$ are coprime natural numbers. Take $(y_1,y_2)\in\left(\mathbb{C}^*\right)^2$ small and consider the loop 
\begin{align*} \gamma\colon \partial\mathbb{D}&\to \mathbb{C}^2\\
\zeta  &\mapsto \left(y_1\zeta^{m_1},y_2\zeta^{-m_2}\right),
\end{align*}  which is contained in the leaf of $\mathcal{F}$ through $(y_1, y_2)$. We notice that $\gamma$  is the boundary of the parametrized singular disc 
$$D=\Big\{\left(y_1w^{m_1},y_2\bar{w}^{m_2}\right)\colon w\in\overline{\mathbb{D}}\Big\},$$
where $\overline{\mathbb{D}}=\{w\in\mathbb{C}\colon |w|\le 1\}$.
Observe that $D$ is just a $C^{\infty}$ singular disc. 

The 2-form $d\omega$ can be expressed as $$d\omega=\sum\limits c_{k_1k_2}x_1^{k_1}x_2^{k_2}dx_1dx_2,$$ where
 $c_{k_1k_2}\in\mathbb{C}^N$.  
Then we have 
\begin{align}\int\limits_{\gamma}\omega &=\int\limits_{D}d\omega= \sum\limits c_{k_1k_2}\int\limits_{\mathbb{D}}\left(y_1w^{m_1}\right)^{k_1}\left(y_2\bar{w}^{m_2}\right)^{k_2}
d\left(y_1w^{m_1}\right)d\left(y_2\bar{w}^{m_2}\right)\nonumber \\
&=\sum m_1m_2c_{k_1k_2}y_1^{k_1+1}y_2^{k_2+1}\int\limits_{\mathbb{D}}w^{m_1(k_1+1)-1}\bar{w}^{m_2(k_2+1)-1}dwd\bar{w}.\label{suma}
\end{align}
We will use the folowing lemma, whose proof is a straightforward computation.
\begin{lem}\label{dzdz}Let $p,q\in\mathbb{Z}_{\ge 0}$. Then
$$ \int\limits_{\mathbb{D}}w^p\bar{w}^qdwd\bar{w}= \begin{cases} -\frac{2\pi \sqrt{-1}}{p+1} &\mbox{if } p=q,\\ 
0 & \mbox{otherwise. } \end{cases} $$
\end{lem} By this lemma we see that in the summation \eqref{suma} only remain the terms corresponding to indexes in
the set 
$$I=\Big\{(k_1,k_2)\colon m_1(k_1+1)=m_2(k_2+1)\Big\}$$ and we obtain

\begin{align} \label{suma2}
\int\limits_{\gamma}\omega 
=\sum\limits_{(k_1,k_2)\in I} \left(\frac{-2\pi m_2\sqrt{-1}}{k_1+1}\right)c_{k_1k_2}y_1^{k_1+1}y_2^{k_2+1}.
\end{align}
From the equation above we see that $\int_{\gamma}\omega$  is an analytic function of $(y_1,y_2)$, hence 
$\int_{\gamma}\omega$ take a lot of values whenever the series in \eqref{suma2} is not null. To guarantee that this series
is not null we have to choose $(m_1,m_2)$ such that, for some $(k_1,k_2)\in I$, we have $c_{k_1k_2}\neq 0$. This choice is clearly possible if there exists some $c_{k_1k_2}\neq 0$: for this particular $(k_1,k_2)$ we  can choose $(m_1,m_2)$  such that
 $m_1(k_1+1)=m_2(k_2+1)$. In this case we can say that the linear singularity \eqref{linearsystem}
 is adapted to the coefficient $c_{k_1k_2}$.  
The idea for
the construction of the vector field  $Z$ in Theorem \ref{teorema2} is to find a foliation $\mathcal{F}$ that 
``contains'' many linear singularities adapted to each coefficient $c_{k_1k_2}\neq 0$. Actually, in this task we only need to consider a finite set coefficients $c_{k_1k_2}$ which spans the space $W_{\mathcal{D}}$. The term ``contains''
above is used to mean that such linear singularities are generated by $\mathcal{F}$ after  a blow-up of the origin.
The construction of the foliation $\mathcal{F}$  will be achieved in next section  but we will finish this one
with a result --- Proposition \ref{retornos} ---  collecting the ideas sketched above in a more general setting.
\subsection*{Returns associated to resonant singularities }

Given $x,y\in \mathbb{C}^M$, denote   $$x \cdot y=x_1y_1+\dots+x_My_M$$  
and,  for  $1\le i\le M$, let $e_i\in\mathbb{C}^M$ be the standard $i^{\textrm{th}} $ canonical vector. Moreover, given
$L\in  \mathbb{Z}^M$ and $1\le i,j\le M$, we denote \begin{equation}\label{defili}L_{ij}=L+e_i+e_j.\end{equation} 
If $y=(y_1,\dots,y_M)\in\left(\mathbb{C}^*\right)^M$ and $m=(m_1,\dots,m_M)\in \mathbb{Z}^M,$ the loop 
\begin{align*} 
\zeta \in \partial\mathbb{D} \mapsto &\left(y_1\zeta^{m_1},\dots,y_M\zeta^{m_M}\right)\in \mathbb{C}^M
\end{align*} will be called the  $m$-loop based at $y$; this loop is contained in a leaf of the holomorphic foliation defined
by the linear system 

\begin{equation}\label{linearsystemG}
\begin{aligned}
x_1'&=&m_1x_1\\
&\hspace{0.2cm}\vdots&\\
x_M'&=&m_Mx_M.
\end{aligned}
\end{equation} 
Given $\alpha\in\{1,\dots,M\}$, we say that the vector $m=(m_1,\dots,m_M)\in \mathbb{Z}^M$ is {$\alpha$-elementary} if 
$m_{\alpha}>0$ and the other coordinates of $m$ are negative. Let $m\in\mathbb{Z}^M$ be {$\alpha$-elementary} and let $\gamma$
be the $m$-loop based at $y\in\left(\mathbb{C}^*\right)^M$. Observe that $\gamma$ can be expressed as
$$\zeta\in\partial \mathbb{D}\mapsto 
 \left(y_1{\bar{\zeta}}^{-m_1},\dots,y_{{\alpha-1}}\bar{\zeta}^{-m_{\alpha-1}}, y_ {{\alpha}}{\zeta}^{m_{\alpha}}, 
y_{{\alpha+1}}\bar{\zeta}^{-m_{\alpha+1}},\dots,
y_M\bar{\zeta}^{-m_M}\right).$$ Then $\gamma$ is the boundary of the parametrized ($C^{\infty}$) disc 
$$w\in\overline{\mathbb{D}}\mapsto \left(y_1\bar{w}^{-m_1},\dots,y_{{\alpha-1}}\bar{w}^{-m_{\alpha-1}}, 
y_ {{\alpha}}{w}^{m_{\alpha}}, 
y_{{\alpha+1}}\bar{w}^{-m_{\alpha+1}},\dots,
y_M\bar{w}^{-m_M}\right),$$
which will be denoted by $D_\gamma$.

\begin{lem}\label{elementary} Let $m\in \mathbb{Z}^M$ be $\alpha$-elementary and let $\gamma$ be the $m$-loop based at $y\in\left(\mathbb{C}^*\right)^M$.
Let  $K\in\left(\mathbb{Z}_{\ge 0}\right)^M$ and $1\le i<j\le M$. Then  the integral $$\int\limits_{D_\gamma} x^{K}dx_{i}dx_{j}$$ is nonzero only if
 $\alpha\in\{i,j\}$ and $m\cdot {K}_{ i j}=0$. Moreover, in this case we have 
\begin{align}\int\limits_{D_\gamma} x^{K}dx_{i}dx_{j}=\begin{cases}\left(\frac{2\pi \sqrt{-1}}{k_{\alpha}+1}\right)
(m_j)y^{{K}_{ i j}},
 &\textrm{if } i=\alpha,\\
\left(\frac{2\pi \sqrt{-1}}{k_{\alpha}+1}\right)(-m_i)y^{{K}_{ i j}}, &\textrm{if } j=\alpha.
\end{cases}
\end{align}
\end{lem}
\begin{proof}
  If $\alpha\notin\{i,j\}$,
we see that
  $${D_\gamma}^*(dx_idx_j)=d\left(y_ {i}{\bar w}^{-m_{i}}\right)d\left(y_ {j}{\bar w}^{-m_{j}}\right)=0.$$ Then the integration of $x^Kdx_idx_j$ over
   $D_\gamma$ is nonzero 
  only if $i=\alpha$ or $j=\alpha$. Suppose that  $i=\alpha$. Then 
\begin{align*}&\int\limits_{D_\gamma} x^Kdx_{\alpha}dx_j \hspace{1cm}\\ 
&\hspace{0.5cm}=\int\limits_{\mathbb D}\left(y_1\bar{w}^{-m_1} \right)^{k_1}\dots 
\left(y_ {{\alpha}}{w}^{m_{\alpha}}\right)^{k_{\alpha}}\dots \left( y_M\bar{w}^{-m_M}\right)^{k_M}
d\left(y_ {\alpha}{w}^{m_{\alpha}}\right)d\left(y_ {j}{\bar w}^{-m_{j}}\right)\\
&\hspace{1.4cm}=\int\limits_{\mathbb D} y^Kw^{m_{\alpha}k_{\alpha}}  \bar{w}^{-m\cdot K+ m_{\alpha}k_{\alpha}}
\left(m_{\alpha}y_ {\alpha}{w}^{m_{\alpha}-1}dw\right)\left(-m_j y_ {j}{\bar w}^{-m_{j}-1}d\bar{w}\right)\\
&\hspace{1.8cm}=m_{\alpha}(-m_j) y^{K_{\alpha j}}\int\limits_{\mathbb D} w^{m_{\alpha}k_{\alpha}+m_{\alpha}-1} 
 \bar{w}^{-m\cdot K+ m_{\alpha}k_{\alpha}-m_j-1}
dwd\bar w,
\end{align*}
which, by Lemma \ref{dzdz}, is nonzero only if 
$$m\cdot {K}_{\alpha j}=0,$$ and in this case (again by Lemma \ref{dzdz}) we obtain
\begin{align}\nonumber \int\limits_{D_\gamma} x^{K}dx_{\alpha}dx_j
=\left(\frac{2\pi \sqrt{-1}m_j}{k_{\alpha}+1}\right)y^{{K}_{\alpha j}}.\nonumber
\end{align}
If we have $j=\alpha$, a similar computation gives us
\begin{align}\nonumber\int\limits_D x^{K}dx_{i}dx_{\alpha}=\begin{cases}
-\left(\frac{2\pi \sqrt{-1}m_i}{k_{\alpha}+1}\right)y^{{K}_{i\alpha}}, &\textrm{if } m\cdot {K}_{i\alpha }=0\\
0, &\textrm{otherwise}\end{cases}
\end{align}which completes the proof.
\end{proof}
Let  $L\in\left(\mathbb{Z}_{\ge 0}\right)^M$,  $\alpha\in\{1,\dots,M\}$ and,  given $i,j\in\{1,\dots,M\}$, denote $$L_{-ij}=L-e_i-e_j.$$ Then define $\mathcal{M}(L,\alpha)$ as the set composed of the monomial 2-forms
\begin{equation}\label{monomios}\begin{aligned}&x^{L_{-i\alpha}}dx_idx_\alpha \quad\,\textrm{ for } 1\le i<\alpha, \textrm{ and }\\
 &x^{L_{-\alpha j}}dx_\alpha dx_j\quad\textrm{ for } \alpha<j\le M.\end{aligned}
 \end{equation}
 We notice that $\mathcal{M}(L,\alpha)$ has $(M-1)$ monomials, although some of these monomials could have negative exponents. 
Let $\omega$ be the $N$-tuple of forms as we have in \eqref{multiforma} and, as in \eqref{ecudo}, set 
$$d\omega=\sum{c_{K}^{ij}x^Kdx_idx_j}.$$ If $\varsigma$ denotes a monomial $x^Kdx_idx_j$ with 
$K\in\mathbb{Z}^M$, $1\le i<j\le M$, we define $c(\varsigma)$ as the coefficient of $\varsigma$ in $d\omega$.  Naturally, since $\omega$ is holomorphic, if $\varsigma$ has some negative exponent we have 
 $c(\varsigma)=0$.

\begin{prop}\label{retornos}
Let
$m\in\mathbb{Z}^M $  be an $\alpha$-elementary vector for some  $\alpha\in\{1,\dots,M\}$.  Given $y\in\left(\mathbb{C}^*\right)^M$, let 
$\gamma_y$ be the $m$-loop based at $y$. Then the function
 \begin{align*}\rho(m)\colon y\mapsto \int\limits_{\gamma_y}\omega\in\mathbb{C}^N\end{align*}
has a power series expansion $$\rho(m)=\sum\limits a_L(m) y^L,$$ where $a_L(m)\in\mathbb{C}^N$, $L\in\left(\mathbb{Z}_{\ge 0}\right)^M$, such that the following properties hold:
\begin{enumerate}
\item We have \label{item11}$a_L(m)\neq 0$ only if $m\cdot L=0$.
\item \label{item22}Given $L$ with $m\cdot L=0$, let $\varsigma_1,\dots,\varsigma_{M-1}$ be the  monomials 
of $\mathcal{M}(L,\alpha)$ in the order they appear in \eqref{monomios}. Then there exist complex numbers $$\lambda_1(L,m),\dots,\lambda_{M-1}(L,m),$$ not depending on $\omega$, such that
$$a_L(m)=\lambda_1(L,m)c(\varsigma_1)+\dots + \lambda_{M-1}(L,m)c(\varsigma_{M-1}).$$ 
\item \label{item33} Set $L=(l_1,\dots,l_M)$ and assume that $l_\alpha\neq 0$. Suppose that there exists $\mathbb{C}$-linearly independent $\alpha$-elementary vectors 
 $$\mathfrak{m}_1,\dots, \mathfrak{m}_{M-1}\in\mathbb{Z}^M$$  such that $$\mathfrak{m}_1\cdot L=\dots =\mathfrak{m}_{M-1}\cdot L=0.$$ Then the matrix 
$$\hspace{2cm}\Big[\lambda_i(L,\mathfrak{m}_j)\Big]_{1\le i,j \le M-1}$$ is invertible. In particular --- from item \eqref{item22} --- the 
coefficients in $d\omega$ of the monomials in $\mathcal{M}(L,\alpha)$ are linear combinations of the $(M-1)$ coefficients 
$$a_L(\mathfrak{m}_1), \dots,a_L(\mathfrak{m}_{M-1})$$
 of the monomial $y^L$
in the series $$\rho(\mathfrak{m}_1),\dots,\rho(\mathfrak{m}_{M-1}).$$
\item \label{coment} In order to indicate the dependence of $\rho(m)$ on  $\omega$, we will write $$\rho(m)=\rho(\omega,m) \textrm{ and }
a_L(m)=a_L(\omega,m).$$ From item \eqref{item22} we see that $a_L(\omega,m)$ depends only on the coefficients of $d\omega$ corresponding to the monomials in $\mathcal{M}(L,\alpha)$. Since these monomials have  degree $(|L|-2)$, we  conclude the following: if $d\omega_1$ and $d\omega_2$
have series coinciding up to order $k$, then we have $$a_L(\omega_1,m)=a_L(\omega_2,m),\textrm{ whenever }|L|\le k+2.$$

\end{enumerate}
\end{prop}
\begin{rem}\label{remark1} Let $L$ be as in item \eqref{item33}. We will show that, if $L$ has some nonzero coordinate besides the coordinate $l_\alpha$, then  those vectors $\mathfrak{m}_1,\dots, \mathfrak{m}_{M-1}$
always exist --- this will be crucial in the proof of Theorem \ref{teorema2}. Without loss of generality we can assume that $\alpha=M$.  Take vectors $$\mathfrak{m}'_1,\dots,\mathfrak{m}_{M-1}'
\subset \mathbb{N}^{M-1}\times\{0\}\subset \mathbb{C}^M$$ which are $\mathbb{C}$- linearly independent. 
Multiplying these vectors by a suitable natural number if necessary, we can assume that the numbers
$$m_1=\frac{\mathfrak{m}_1'\cdot L}{l_M},\dots , m_{M-1}=\frac{\mathfrak{m}_{M-1}'\cdot L}{l_M}$$
are non-negative integral numbers. Actually, since $L$ has a positive coordinate $l_\beta$ with $\beta\in\{1,\dots, M-1\}$,  the numbers $m_1,\dots,m_{M-1}$ are naturals. Then 
\begin{equation*}
\mathfrak{m}_1 =-\mathfrak{m}'_1+m_1e_M \hspace{0.3cm}, \dots , \hspace{0.3cm}\mathfrak{m}_{M-1}=-\mathfrak{m}'_{M-1}+m_{M-1}e_M
\end{equation*}
are the desired vectors. 
\end{rem}
\begin{proof} By Stokes Theorem,
$$ \int\limits_{\gamma_y}\omega=\int\limits_{D_{\gamma_y}}d\omega=\sum c_K^{ij}\int\limits_{D_{\gamma_y}}x^Kdx_idx_j.$$
By Lemma \ref{elementary}, each integral \begin{equation} \int\limits_{D_{\gamma_y}}x^Kdx_idx_j
\label{integral1}\end{equation} is a monomial of $y=(y_1,\dots,y_{M})$. Moreover, given $L\in\left(\mathbb{Z}_{\ge 0}\right)^M$, again  by Lemma \ref{elementary} we can see that the integral \eqref{integral1} is a monomial in $\mathbb{C}^*y^L$ only if the monomial $x^Kdx_idx_j$ belongs to the set   
\begin{equation}\label{mon22}\mathcal{M}=\{x^Kdx_idx_j\colon\alpha\in\{i,j\},\;m\cdot L=0, \;K_{ij}=L\}.\end{equation}
Then we have 
\begin{equation}a_L(m)y^L=\sum \int\limits_{D_{\gamma_y}}c_K^{ij}x^Kdx_idx_j,
\label{integral2}\end{equation} where the summation extends over the monomials $x^Kdx_idx_j$ in the
 set $\mathcal{M}$. If $m\cdot L\neq 0$, the set $\mathcal{M}$ is empty, so item \eqref{item11} follows. Suppose that 
 $m\cdot L=0$. Then we have $\mathcal{M}=\mathcal{M}(L,\alpha)$, so the summation \eqref{integral2} extends over 
 the monomials $\varsigma_1,\dots,\varsigma_{M-1}$. Thus  --- by Lemma \ref{elementary} --- for each term  of the summation 
 \eqref{integral2} we have
  \begin{align*}\int\limits_{D_{\gamma_y}}c_K^{ij}x^Kdx_idx_j=&
  \begin{cases} \left(\frac{2\pi \sqrt{-1}}{l_\alpha}\right)(-m_i)c_K^{ij}y^{L},
 &\textrm{if } j=\alpha \\
\left(\frac{2\pi \sqrt{-1}}{l_\alpha}\right)(m_j)c_K^{ij}y^{L}, &\textrm{if } i=\alpha .  
\end{cases}
 \end{align*}
 Then, if we set $\eta=\left(\frac{2\pi \sqrt{-1}}{l_\alpha}\right)$ and
 \begin{align}\label{horizonte}\big(\lambda_1(L,m),\dots,\lambda_{M-1}(L,m)\big) :=
 \eta (-m_1,\dots,-m_{\alpha-1},m_{\alpha+1},\dots,m_M),
 \end{align} we obtain
\begin{equation} a_L(m)=\lambda_1(L,m)c(\varsigma_1)+\dots + 
\lambda_{M-1}(L,m)c(\varsigma_{M-1}),\label{horizontal}\end{equation} so Item \eqref{item22}
is proved.  For each $i=1,\dots,M-1$, put $$\mathfrak{m}_i=(\mathfrak{m}^i_1,\dots,\mathfrak{m}^i_M).$$  
Then, from \eqref{horizonte}  we have

 \begin{align*}\nonumber\Big[\lambda_i(L,\mathfrak{m}_j)\Big]^T=\eta\begin{bmatrix}-\mathfrak{m}^1_1 & \dots &-\mathfrak{m}^1_{\alpha-1} &
 \mathfrak{m}^1_{\alpha+1} & \dots &
  \mathfrak{m}_M^1\\
  -\mathfrak{m}^2_1 & \dots &-\mathfrak{m}^2_{\alpha-1} &
 \mathfrak{m}^2_{\alpha+1} & \dots &
  \mathfrak{m}_M^2\\
 \vdots &  & \vdots &\vdots & & \vdots
\\
  -\mathfrak{m}^{M-1}_1 & \dots &-\mathfrak{m}^{M-1}_{\alpha-1} &
 \mathfrak{m}^{M-1}_{\alpha+1} & \dots &
  \mathfrak{m}_M^{M-1}
 \end{bmatrix}.
\end{align*}Thus, if we consider the projection  $$\varpi(x_1,\dots,x_M)=(x_1,\dots,x_{\alpha-1},x_{\alpha+1},\dots,x_M),$$ we have
\begin{equation}\label{matrizfinal}
\det\Big[\lambda_i(L,\mathfrak{m}_j)\Big]= \pm \eta^{M-1}\det \begin{bmatrix} \varpi(\mathfrak{m}_1)\\
 \varpi(\mathfrak{m}_2)\\
 \vdots \\
 \varpi(\mathfrak{m}_{M-1})
 \end{bmatrix}.
 \end{equation}
 Recall that the vectors $\mathfrak{m}_1,\dots,\mathfrak{m}_{M-1}$ belong to the hyperplane  
 $$E=\Big\{(x_1,\dots,x_M)\in\mathbb{C}^M\colon l_1x_1+\dots l_Mx_M=0\Big\}. $$ 
 Since those vectors are linearly independent,  in order to prove that the determinant in \eqref{matrizfinal} is nonzero 
 it is enough to show that the projection $$\varpi\colon \mathbb{C}^M\to\mathbb{C}^{M-1}$$ is injective on $E$, which follows directly
 from the inequality  $l_\alpha\neq 0$.  Item \eqref{item33} is proved.
\end{proof}

\section{Holomorphic maps and linear dependence} \label{seccionlemas}
In this section we show a number of lemmas which will be used in the proof of Theorem \ref{teorema2}. These lemmas are mainly based on the vanishing properties of holomophic functions and polynomials.

 \begin{lem}\label{polinomionulo} Let $d\in\mathbb{Z}_{\ge 0}$ and suppose that $T\subset \mathbb{C}$ has at least $d+1$ elements.
  Let $$f\colon\mathbb{C}\to\mathbb{C}^N$$ be a polynomial map of degree $\le d$ such that $f(T)$ is contained in some complex linear
   subspace $W\subset \mathbb{C}^N$. Then  the image of $f$ is contained in $W$. 
  \end{lem}
  \begin{proof} Let $A\colon\mathbb{C}^N\to\mathbb{C}^n$ be a linear map such that $W=\ker A$. Then 
  $A\circ f\colon\mathbb{C}\to\mathbb{C}^n$ is a polynomial map which vanishes on $T$ and with coordinates of degree $\le d$. Whence each coordinate of $A\circ f$ is null. 
  \end{proof}

  \begin{lem}\label{coeficientes} Consider a germ of holomorphic function  $\rho\colon(\mathbb{C}^M,0)\to(\mathbb{C}^N,0)$ given by a series $$\rho(x)=\sum\theta_{K}x^K,$$ where 
   $K=(k_1,\dots,k_M)$, $x=(x_1,\dots,x_M)$, $\theta_{K}\in \mathbb{C}^N.$  Suppose that the image of $\rho$ is contained in some complex linear 
  subspace $W\subset\mathbb{C}^N$. Then each coefficient $\theta_{K}$ belongs to $W$.
  \end{lem}
  \begin{proof}  Let $A\colon\mathbb{C}^N\to\mathbb{C}^n$ be a linear map such that $W=\ker A$. It follows from the hypothesis that $A\circ \rho\colon(\mathbb{C}^M,0)\to\mathbb{C}^n$ is null. On the other hand, given $x$ in the domain of convergence of the series $\rho$,
   it is easy to prove that 
  $$A\left(\sum\theta_{K}x^K\right)=\sum A(\theta_K) x^K.$$ Then the function  $\sum A(\theta_K) x^K$ vanishes, hence each $A(\theta_K)$ is zero and the lemma follows. 
  \end{proof}
  
  Given $\epsilon>0$,  an $\epsilon$-basis of $\mathbb{R}^k$  will be a basis whose elements have modulus smaller than $\epsilon$. 
   \begin{lem}\label{base real} For each $j=1,\dots,n$, let $U_j\subset\mathbb{C}^M$ be an open connected set accumulate on the origin and let
    $f_j\colon U_j\to\mathbb{C}^N$ be a  holomorphic function such that $f_j(z)\rightarrow 0$ as $z\rightarrow 0$. Suppose that the set 
   $$\Gamma=\{f_j(z)\colon z\in U_j, j=1,\dots,n\}$$ spans $\mathbb{C}^N$ as a complex vector space. Then, for any $\epsilon>0$, the set 
   $\Gamma$ contains an $\epsilon$-basis of  $\mathbb{C}^N$  as a real vector space.
 \end{lem}
 \begin{proof} Let $ \epsilon>0$. For each $j=1,\dots,n$, let $U^*_j\subset U_j$  be a nonempty open set such that
  $|f_j(x)|<\epsilon$ for all $x\in U^*_j$.  We only need to prove that 
 $$\Gamma^*=\{f_j(z)\colon z\in U_j^*, j=1,\dots,n\}$$ spans $\mathbb{C}^N$ as a real vector space. Suppose that this is not true. 
 Then $\Gamma^*$ is contained 
 in some real hyperplane of  $\mathbb{C}^N$. Thus we can find a  nonzero complex linear map $$A\colon\mathbb{C}^N\to\mathbb{C}$$ such that $\Gamma^*\subset \ker (\re (A))$, which means that  $\re(A\circ f_j(x))=0$, $x\in U_j^*$; $j=1,\dots n$. 
 Then, since $A\circ f_j(x)\to 0$ as $x\to 0$, we have
    $$(A\circ f_j)|_{U_j^*}\equiv0,\quad j=1,\dots,n$$ and therefore, since the $U_j$ are connected, $A\circ f_j\equiv 0$,  $j=1,\dots,n$. But this means that $\Gamma$ is contained in the codimension one complex hyperplane $\ker A\subset\mathbb{C}^N$, which is a contradiction.   
 \end{proof}

\begin{lem}\label{densidad} Let $\Gamma\subset\mathbb{R}^k$ be a set such that , for any $\epsilon>0$, there exists an
 $\epsilon$-basis $u_1,\dots,u_k$ of 
$\mathbb{R}^k$ with 
$$\{\pm u_1,\dots,\pm u_k\}\subset \Gamma.$$ Let $\Omega\subset\mathbb{R}^k$ be open, convex and bounded, and 
let   ${\Sigma}\subset\mathbb{R}^k$ be such that ${\Sigma}\cap\Omega\neq\emptyset$.  Suppose that there exists  $\delta>0$ such that:
 \begin{equation}\label{implica} z\in {\Sigma}, \; \dist (z,\Omega)<\delta \implies z+u\in {\Sigma}, \forall u\in \Gamma.\end{equation}  Then  $\Omega$ is contained in the topological closure
  $\overline{{\Sigma}}$ of $\Sigma$.
\end{lem}
\begin{proof}
Take $z_0\in {\Sigma}\cap\Omega$. Let $\varepsilon>0$ be arbitrary. It is sufficient to prove that there exists an $\epsilon$-basis $u_1,\dots,u_k$ of $\mathbb{R}^k$ such that 
$$\Omega \cap \left(z_0+\mathbb{Z}u_1+\dots+\mathbb{Z}u_k\right)\subset {\Sigma}.$$
By hypothesis we can find a basis $u_1,\dots,u_k$ of $\mathbb{R}^k$ with $|u_j|<\min\{\varepsilon, \delta/k\}$
 for all $j=1,\dots,k$ and such that $\{\pm u_1,\dots,\pm u_k\}\subset \Gamma$.
 Let $\mathcal{T}$ be the tessellation of $\mathbb{R}^k$ generated by the basis $u_1,\dots,u_k$ and the point $z_0$. 
Let $\mathfrak{P}$ be the union of  parallelotopes of $\mathcal{T}$ that meet $\Omega$
--- then $\Omega\subset\mathfrak{P}$ and, in particular, $z_0\in\mathfrak{P}$. 
Consider a point $z'$ in the set 
$$\Omega\cap \left(z_0+\mathbb{Z}u_1+\dots+\mathbb{Z}u_k\right);$$ this point is clearly contained in $\mathfrak{P}$. 
 Since $\mathfrak{P}$ is connected and $z_0\in\mathfrak{P}$, we can
 find points $$w_0=z_0,w_1,\dots,w_n=z'$$ that are vertices of  parallelotopes in $\mathfrak{P}$ and such that 
$$w_j-w_{j-1}\in\{\pm u_1,\dots,\pm u_k\}\subset \Gamma, \;\forall j=1,\dots,n.$$
Observe that, since the parallelotopes in $\mathcal{T}$ have Euclidean diameter smaller than 
$$|u_1|+\dots +|u_k|<\delta,$$ any point $w$ belonging to $\mathfrak{P}$ --- which therefore belongs to a parallelotope touching $\Omega$ ---  satisfies $$\dist (w,\Omega)<\delta.$$ In particular we have $\dist (w_j,\Omega)<\delta,\;\forall j=0,1,\dots,n$. Thus, since $(w_{j+1}-w_{j})\in \Gamma$, if we had $w_j\in {\Sigma}$, by the hypothesis \eqref{implica} we would obtain
 $$w_{j+1}=w_j+(w_{j+1}-w_j)\in {\Sigma}.$$  Therefore, since $w_0 =z_0\in {\Sigma}$, we successively obtain $w_1,\dots,w_n\in {\Sigma}$, so that $z'\in {\Sigma}$. 
\end{proof}

  \section{Proof of the results}
  As we have seen in previous sections, our construction of foliations that are tangent to the distribution $\mathcal{D}$ on $\mathbb{C}^{M+N}$
will be achieved 
by lifting to $\mathcal{D}$ some suitable holomorphic foliations in  $\mathbb{C}^M$. The existence of these foliations in $\mathbb{C}^M$ is guaranteed by Proposition \ref{log} below, whose proof is done in Section \ref{seccion3}.
If $k\in\mathbb{N}$, a diffeomorphism $h\in\diff (\mathbb{C}^M,0)$  will be said  $k$-tangent to the identity if  
the difference $(h-\id)$
has vanishing order greater than $k$ at $0\in\mathbb{C}^M$.

\begin{prop}\label{log} Consider coordinates $(t_1,\dots,t_M)$ in $\mathbb{C}^M$. Let $T\subset\mathbb{C}$ be a finite set, let $\lambda\colon T\to\left(\mathbb{C}^*\right)^{M}$ and let $k\in\mathbb{N}$ be given.    Then, there exists a one dimensional singular holomorphic foliation $\mathcal{G}$  on $\mathbb{C}^M$ satisfying the following properties.
\begin{enumerate}
\item\label{con1}For a suitable choice of numbers $P\in\mathbb{N}$, $\nu_{ij}\in\mathbb{C}^*$, $b_j\in\mathbb{C}$, $i=1,\dots,M-1$, $j=1,\dots,P$, such that the $b_j$ are pairwise distinct,  the foliation $\mathcal{G}$ is defined by the non-autonomous system
\begin{equation}\label{linearsystemG2}\nonumber
\begin{aligned}
 t_1'&= \left(\sum\limits_{j=1}^P \frac{\nu_{1j}}{t_M-b_{j}}\right)t_1\\
 t_2'&=\left(\sum\limits_{j=1}^P \frac{\nu_{2j}}{t_M-b_{j}}\right)t_2\\
 &\hspace{0.2cm}\vdots&\\
t_{M-1}'&=\left(\sum\limits_{j=1}^P \frac{\nu_{(M-1)j}}{t_M-b_{j}}\right)t_{M-1}\\
t_M'&=1.
\end{aligned}
\end{equation}Thus the singular set of $\mathcal{G}$ on $\mathbb{C}^N$ is composed of the points 
$$(0,\dots, 0, b_j),\quad  j=1,\dots, P.$$

\item\label{con2} The set $T$ is contained in $\{b_1,\dots,b_P\}$ and, for each $\tau\in T$, there exists a  diffeomorphism $h_\tau\in\diff (\mathbb{C}^M,0)$ such that,
\begin{enumerate}
 \item \label{con2a}$h_\tau$  is $k$-tangent to the identity; and
\item\label{con2c}  if $\lambda(\tau)=(\lambda_1,\dots,\lambda_M)$, the foliation $\mathcal{G}$ at the singularity $$p_\tau=(0,\dots,0,\tau)\in\mathbb{C}^M $$ is generated by the pushforward of the  linear system
\begin{equation}\label{linearsystemG2}\nonumber
\begin{aligned}
 s_1'&=&\lambda_1s_1\\
 s_2'&=&\lambda_2s_2\\
 &\hspace{0.2cm}\vdots&\\
s_M'&=&\lambda_Ms_M
\end{aligned}
\end{equation}  by the diffeomorphism $f_\tau(s)\colon =p_\tau+h_\tau(s)$. 
\end{enumerate}

\item \label{con2.1}Since it is defined by a rational vector field, the foliation  $\mathcal{G}$ can be viewed as the 
 the strict transform of a singular holomorphic foliation $\mathcal{F}$ by the punctual blow-up
\begin{align*} &\hspace{0.84cm}\pi\colon\widehat{\mathbb{C}^M}\to\mathbb{C}^M
\end{align*} at the origin. In other words, there exists  a singular holomorphic foliation 
$\mathcal{F}$ on $\mathbb{C}^M$  --- which  is generated by a 
polynomial vector field --- such that $\mathcal{F}$ is the pushforward of $\mathcal{G}$   by the map
\begin{align*}(t_1,\dots,t_M)\mapsto (x_1,\dots,x_M)= (t_1,t_1t_2,\dots,t_1t_M).\end{align*} 

\item\label{con4} Let $S$ be the hypersurface  defined by the equation $$ x_1\cdots x_{M-1}\prod\limits_{j=1}^P(x_M-b_jx_1)=0.$$
Then
\begin{enumerate}
\item \label{con4a} $S$ is invariant by $\mathcal{F}$;
\item \label{con4b}  the leaves of $\mathcal{F}$ in $\mathbb{C}^M\backslash S$ are everywhere dense; and 
\item \label{con4c} given $\tau\in T$, the germ hypersurface $\pi^{-1}(S)$ at $p_\tau$ is the image by $f_\tau$ of the hypersurface
$\{s_1\cdots s_M=0\}.$
\end{enumerate}
\item \label{con3}   Given $\epsilon>0$ and a neighborhood $\Delta\subset\mathbb{C}^M$ of the origin, there exists an open set 
 $\mathscr{U}\subset \Delta\backslash S$,   such that the following properties hold.  
 \begin{enumerate}
 \item\label{con3a}  For each $p\in\mathscr{U}$, there exists a set $E$ dense in $\mathscr{U}$, contained in the leaf $L_p$ 
of $\mathcal{F}|_\Delta$ through $p$,  such that  any point $q\in E$ can be connected to $p$ by a curve  in $L_p$ whose Euclidean length is smaller than $\epsilon$.
\item\label{con3b} The set $\pi^{-1}\left(\mathscr{U}\cup S\right)$ is a neighborhood of each point in the exceptional divisor  outside  the tangent cone of $$\prod\limits_{\tau\in T}(x_M-\tau x_1)=0.$$
\item\label{con3c} Let  $\tau\in T$ be such that there exists a real line through $0\in\mathbb{C}$ separating $\lambda_M( \tau)$ from the set
\begin{equation*}\{\lambda_1(\tau),\; \lambda_1(\tau)+\lambda_2(\tau), \dots, \lambda_1(\tau)+\lambda_{M-1}(\tau)\}.\end{equation*}
Then $\pi^{-1}\left(\mathscr{U}\cup S\right)$ is a neighborhood of each point in the afin space $\{t_1=0, t_M=\tau\}$.
Furthermore, if the property above takes place for all $\tau\in T$, then  $\mathscr{U}\cup S$ is a neighborhood of the origin. If $M=2$, the hypothesis above means that there exists a real line separating  $\lambda_2(\tau)$ from $\lambda_1(\tau)$, which is equivalent to the condition $\frac{\lambda_2(\tau) }{\lambda_1(\tau)}\in \mathbb{C}\backslash [0,+\infty)$.

\end{enumerate}
\end{enumerate}
\end{prop}

  The rest of this section is devoted to the proof of our main results. 
  \subsection*{Proof of Theorem \ref{teorema2}}
Let
  $\pi\colon\widehat{\mathbb{C}^M}\to\mathbb{C}^M$ the blow-up at the origin an consider coordinates $(t_1,\dots,t_M)$ in $\widehat{\mathbb{C}^M}$ such that $$\pi(t_1,\dots,t_M)=(t_1,t_1t_2,\dots,t_1t_M).$$ Let $\tilde{\omega}=\pi^*(\omega):=(\pi^*(\omega_1),\dots,\pi^*(\omega_N))$. Set 
    \begin{equation} \label{ecat}d\tilde{\omega}=\sum{\theta_{L}^{ij}t^Ldt_idt_j},\end{equation}  where 
  $\theta_L^{ij}\in\mathbb{C}^N$, $L=(l_1,\dots,l_M)\in \left(\mathbb{Z}_{\ge 0}\right)^M$ and the summation extends over  $  |L|\ge 0$,  $1\le i < j\le M$. 
  Recall that 
  \begin{equation} \nonumber d\omega=\sum{c_{K}^{ij}x^Kdx_idx_j}. \end{equation} In order to express the relation between the coefficients $\theta_L^{ij}$ and the $c_K^{ij}$, we analyze the pullback by $\pi$
  of a monomial 2-form  $x^Kdx_idx_j$ with $1\le i<j\le M$. Suppose first that $i>1$. Then
  \begin{align}\nonumber\pi^*(x^Kdx_idx_j)&=(t_1)^{k_1}(t_1t_2)^{k_2}\dots (t_1t_M)^{k_M}d(t_1t_i)d(t_1t_j)\\
 \nonumber &=t_1^{k_1+\dots +k_M}t_2^{k_2}\dots t_M^{k_M}(t_1^2dt_idt_j+t_1t_idt_1dt_j-t_1t_jdt_1dt_i)\\
  \label{ecublow}&=t_1^{k_1+\dots +k_M+2}t_2^{k_2}\dots t_M^{k_M}dt_idt_j\\
  \nonumber&\hspace{1.5cm}+t_1^{k_1+\dots+k_M+1}t_2^{k_2}\dots t_i^{k_i+1} \dots  t_M^{k_M}dt_1dt_j\\
 \nonumber&\hspace{3cm} -t_1^{k_1+\dots+k_M+1}t_2^{k_2}\dots t_j^{k_j+1} \dots  t_M^{k_M}dt_1dt_i.
   \end{align} On the other hand, if $i=1$ we have
   \begin{align}\nonumber\pi^*(x^Kdx_1dx_j)&=(t_1)^{k_1}(t_1t_2)^{k_2}\dots (t_1t_M)^{k_M}d(t_1)d(t_1t_j)\\
  &=t_1^{k_1+\dots +k_M+1}t_2^{k_2}\dots  t_M^{k_M}dt_1dt_j.
   \label{ecublow1}\end{align} 
   From equations \eqref{ecublow} and \eqref{ecublow1} we can deduce the following facts.
   \begin{enumerate}
   \item \label{111}If $i>1$,  the monomial $t^Ldt_idt_j$ only appears in the pullback of the monomial
   $x^Kdx_idx_j$ with $K$ such that  $$L=\phi(K)\colon =(k_1+\dots+k_M+2,k_2,\dots,k_M)$$ and we have $$\theta_L^{ij}=c_{\phi^{-1}(L)}^{ij}.$$ 
   \item \label{222} If we define 
   \begin{align*}\varphi_1(K)\colon &=(k_1+\dots+k_M+1,k_2,\dots,k_M)\\
   \varphi_j(K)\colon & =(k_1+\dots+k_M+1,k_2,\dots, k_j + 1,\dots,k_M), \hspace{0.2cm} j =2, \dots, M,
   \end{align*}
   the monomial $t^Ldt_1dt_j$ appears in the pullbacks of the monomials \begin{align*} 
   &x^{K}dx_idx_j,\; \varphi_i(K)=L,\: i< j, \textrm{ and}\\ 
   &x^{K}dx_jdx_i,\; \varphi_i(K)=L,\; i> j.
   \end{align*}
   Thus, if we put $c_K^{ij}=-c_K^{ji}$ for $i>j$,  we have $$\theta_L^{1j}=\sum_{i\neq j}c_{\varphi_i^{-1}(L)}^{ij}.$$ 
   \item From  items \eqref{111} and \eqref{222} above we can express  the $c_K^{ij}$ in terms of the $\theta_L^{ij}$ and we obtain
    \begin{align*}
   c_K^{1j}&=\theta_{\varphi_1(K)}^{1j}-\sum_{i\neq 1, j}\theta_{\phi\varphi_i^{-1}\varphi_1(K)}^{ij}\\
   c_K^{ij}&=\theta_{\phi(K)}^{ij} \hspace{0.5cm} (i\ne 1).
   \end{align*}
   \item Since all the relations founded above are linear, we conclude that the coefficients $\theta_L^{ij}$ span the same space spanned
   by the $c_K^{ij}$, that is, the space $W_{\mathcal{D}}$.  
   \item \label{condi5} From the equations \eqref{ecublow} and \eqref{ecublow1} we see that every monomial 
   $t_1^{l_1}\dots t_M^{l_M}dt_idt_j$ appearing after the blow-up with a nonzero coefficient is such that
    $$l_1\ge l_2+\dots+l_M.$$ 
    \end{enumerate}
  
   Given $\tau\in\mathbb{C}$,
   let $\tilde{\omega}_{\tau}$ be the expression of $\tilde{\omega}$ in the coordinates
    $$s\colon=(s_1,\dots,s_M)=(t_1,t_2,\dots,t_{M-1},t_M-\tau).$$  By a straightforward computation from Equation \eqref{ecat} we can write
\begin{equation}\label{new} d\tilde{\omega}_\tau=\sum\Theta_L^{ij}(\tau)s^Lds_ids_j,\end{equation}
where

 \begin{equation}\label{teton}\Theta_{L}^{ij}(\tau)=\sum\limits_{k\ge 0}\theta_{(l_1,l_2,\dots,l_M+k)}^{ij}
 \binom{l_M+k}{l_M}\tau^{k}.\end{equation} From fact \eqref{condi5} above, it
 follows that the coefficient
 $\theta_{(l_1,l_2,\dots,l_M+k)}^{ij}$ can be nonzero only if $k\le l_1$.  Then $\Theta_L^{ij}$ is a  polynomial function of $\tau$ such that 
 \begin{equation}\deg \Theta_L^{ij}\le l_1\le |L|.\label{tetita}\end{equation}   
 
 Choose  a finite set  $\mathfrak{M}$ of monomials  in
 $\{\theta_L^{ij}t^Ldt_idt_j\}$   whose coefficients span the space $W_{\mathcal{D}}$ and let $d$ be the maximum degree of the monomials in
 $\mathfrak{M}$. It follows from the discussion above that for each $\theta_{\mathfrak{L}}^{\boldsymbol{\mathfrak{ij}}}t^{\mathfrak{L}}dt_{\boldsymbol{\mathfrak i}}dt_{\boldsymbol{\mathfrak j}}$  in $\mathfrak{M}$  we have 
 \begin{equation} \label{gradod}\deg\Theta_{\mathfrak{L}}^{\boldsymbol{\mathfrak{ij}}}\le |\mathfrak{L}|\le d. \end{equation}
  By Remark \ref{remark1}, we can take a finite set  
   $\Lambda$ in $\mathbb{Z}^{M}$ with the following property:\\

\begin{quote}
   for each monomial  $\theta_{\mathfrak{L}}^{\boldsymbol{\mathfrak{ij}}}t^{\mathfrak{L}}dt_{\boldsymbol{\mathfrak i}}dt_{\boldsymbol{\mathfrak j}}$
in $\mathfrak{M}$, there exist  $\mathbb{C}$-linearly independent $\boldsymbol{\mathfrak{i}}$-elementary vectors $\mathfrak{m}_1,\dots, \mathfrak{m}_{M-1}$ in $\Lambda$  such that 
\begin{align}\label{star} \mathfrak{m}_1\cdot \mathfrak{L}_{\boldsymbol{\mathfrak{ij}}}=\dots=\mathfrak{m}_{M-1}\cdot\mathfrak{L}_{\boldsymbol{\mathfrak{ij}}}=0, \end{align} where, as defined in \eqref{defili}, we have  $$\mathfrak{L}_{\boldsymbol{\mathfrak{ij}}}=\mathfrak{L}+e_{{\boldsymbol{\mathfrak{i}}}}
 +e_{\boldsymbol{\mathfrak{j}}}.$$  \\
\end{quote}  

Take a finite set 
$T\subset \mathbb{C}$ and a function $\lambda \colon T\to \Lambda$  such that
\begin{align} \label{lamb2}
\textrm{ for each $m\in\Lambda$  the set  $\lambda^{-1}(m)$ has at least  $d+1$ points.\hspace{1.8cm}}
\end{align}
Let $\mathcal{G}$ be the foliation given by Proposition \ref{log} associated to  $\lambda:T\to\mathbb{C}^M$ and $k=d+1$.
 The foliation $\mathcal{F}=\pi_*(\mathcal{G})$  is generated by a polynomial vector field $X$  with an isolated singularity at the origin. The vector field $Z$ mentioned in Theorem \ref{teorema2} will be taken being equal to the
  lift $X^\mathcal{D}$ of $X$ to $\mathcal{D}$, so the foliation defined by $Z$ is the lift $\mathcal{F}^\mathcal{D}$
  of $\mathcal{F}$ to $\mathcal{D}$.
  
   Without loss of generality we can suppose that $\Delta =\Delta_1\times\Delta_2$, where $\Delta_1$ and $\Delta_2$ are balls centered at the origins of $\mathbb{C}^M$ and $\mathbb{C}^N$, respectively. We also can  assume that the closure 
of  $\Delta$ is contained in the domains of definition of 
$\mathcal{D}$ and the first integral $H_{\mathcal{D}}$ --- see \eqref{defh}.  Let $\mathfrak{K}>0$ be such that $|\omega|\le \mathfrak{K}$ on $\Delta_1$. Let $r>0$ be the radius of $\Delta_2$ and take $\epsilon>0$ such that 
\begin{equation}\label{repsilon} 6\mathfrak{K}\epsilon<r.\end{equation}
 Associated to the number $\epsilon$,  we have an open set  $ {\mathscr{U}}\subset \Delta_1$ in 
$\mathbb{C}^M$ as given by
\eqref{con3} of Proposition 
\ref{log}. 
Set $$\Delta'={\mathscr{U}}\times\left\{z\in\mathbb{C}^N\colon |z|<\frac{r}{2}\right\}\subset\Delta.$$   Given $p\in\Delta_1$, we set
 $$\Delta_p=\{p\}\times\Delta_2$$ and $$\Delta_p'=\{(p,z)\colon |z|<\frac{r}{2}\}\subset\Delta_p.$$ We clearly have $\Delta'=\bigcup\limits_{p\in {\mathscr{U}}}\Delta'_p.$
 Let $\mathcal{L}$ be a leaf of $\mathcal{F}^{\mathcal{D}}|_{\Delta}$ passing through a point in $\Delta'$. Let $c\in\mathbb{C}^{\kappa}$ be the value taken 
 by $H_{\mathcal{D}}$ on $\mathcal{L}$ and set  $$H_\mathcal{L}=\left({H_\mathcal{D}}|_\Delta\right)^{-1}(c).$$
 That is, $H_\mathcal{L}$ is the level of $H_\mathcal{D}|_\Delta$ containing $\mathcal{L}$.  We will show the following two assertions.
 \begin{enumerate}[label=(\textbf{\Roman*})]
 \item\label{aser1} If $\mathcal{L}$ meets $\Delta'_p$  for some  $p\in {\mathscr{U}}$, then 
 $\mathcal{L}\cap\Delta'_p$ is a dense subset of
  $H_\mathcal{L}\cap\Delta'_p$.  
   \item \label{aser2} If $\mathcal{L}$ contains a point $(p,\mathfrak{z})$ with $p\in {\mathscr{U}}$ and $|\mathfrak{z}|<\frac{r}{3}$, then $\mathcal{L}$ meets $\Delta'_{q}$ for every $q$ in a dense subset of ${\mathscr{U}}$. 
 \end{enumerate}

\noindent\emph{Claim.} Assertions \ref{aser1} and \ref{aser2} imply Theorem \ref{teorema2}. \\

\begin{proof}Consider 
\begin{align}\nonumber 
\Delta''={\mathscr{U}}\times\left\{z\in\mathbb{C}^N\colon |z|<\frac{r}{3}\right\}\subset\Delta'
\end{align}
and define 
$\Delta^*$ as the saturation of $\Delta''$ by the foliation $\mathcal{F}^\mathcal{D}|_{\Delta}$. Let $\mathcal{L}$ be a leaf of $\mathcal{F}^\mathcal{D}|_{\Delta^*}$. It follows from the definition of $\Delta^*$ that $\mathcal{L}$
  is a leaf of $\mathcal{F}^\mathcal{D}|_{\Delta}$ passing through a point in $\Delta''$. Then, by Assertion
  \ref{aser2} we have that  $\mathcal{L}$ meets  $\Delta_q'$ for all $q$ in a set $E$ that is dense in ${\mathscr{U}}$. Therefore,
  by Assertion \ref{aser1}, the closure of $\mathcal{L}$ contains the set
  $$H_\mathcal{L}\cap \Delta'_q$$ for all $q\in E$. Let us prove that  $\overline{\mathcal{L}}$ contains the set $$ H_\mathcal{L}\cap \Delta'_p$$ for all $p\in {\mathscr{U}}$. 
   Take any $p\in {\mathscr{U}}$ and  consider a point $$(p,\mathfrak{z})\in H_\mathcal{L}\cap \Delta'_p.$$ 
  Since $H_\mathcal{D}$ is expressed --- see \ref{defh} --- in the form  $H_\mathcal{D}=A(z)-g(x)$,   the levels of $H_\mathcal{D}$ 
  are transverse to the fibers $x=cst.$ In particular, the manifold 
   $H_\mathcal{L}$ is transverse to the fibers $x=cst.$  Thus, in a neighborhood of $(p,\mathfrak{z})$
   in $H_\mathcal{L}$, the sets $H_\mathcal{L}\cap \Delta'_q$, $q\in {\mathscr{U}}$ defines a locally trivial fibration,  and
    therefore we can find a sequence of points $(q_n,\mathfrak{z}_n)\in H_\mathcal{L}$, $n\in\mathbb{N}$ 
    with $q_n\in E$, such that  $(q_n,\mathfrak{z}_n)\rightarrow (p,\mathfrak{z})$. Since --- as we have 
    seen above --- the closure of $\mathcal{L}$ contains the sets $H_\mathcal{L}\cap \Delta'_{q_n}$,
  we have that $\overline{\mathcal{L}}$ contains the point $(q_n,\mathfrak{z}_n)$ for each $n\in\mathbb{N}$ 
  and therefore  $\overline{\mathcal{L}}$ contains the point $(p,\mathfrak{z})$. Thus,  we have  that
  
  \begin{equation}\label{todope}\overline{\mathcal{L}}\supset \bigcup\limits_{p\in {\mathscr{U}}}H_\mathcal{L}\cap \Delta'_p=  H_\mathcal{L}\cap \Delta'\supset H_\mathcal{L}\cap \Delta''.\end{equation} 
  
  Now, since $\Delta^*$ is the saturation of $\Delta''$ we have that $\overline{\mathcal{L}}$ contains the set $H_\mathcal{L}\cap \Delta^*$. Finally, let us prove the last assertion of Theorem \ref{teorema2}.
  Suppose that $M=2$. Then the foliation $\mathcal{F}$ given by Proposition \ref{log} is a foliation on
   $\Delta_1\subset \mathbb{C}^2$. Given any $\tau\in T$, we have $\lambda(\tau)=(m_1,m_2)\in \Lambda$ and so, from the definition of $\Lambda$, we deduce
  that $m_1/m_2$ is a rational negative number.  It follows from  \eqref{con3c} of Proposition \ref{log}  that 
  $\mathscr{U}\cup S$
  is a neighborhood of $0\in\mathbb{C}^2$. 
  Then $$\mathcal{B}\colon =(\mathscr{U}\cup S) \times\left\{z\in\mathbb{C}^N\colon |z|<\frac{r}{3}\right\}$$
  is a neighborhood of $0\in\mathbb{C}^{2+N}$. Thus, if we set  $\mathcal{S}\colon= S\times \Delta_2$, 
  we can see that   $$\mathcal{B}\backslash \mathcal{S} =\mathscr{U}\times\left\{z\in\mathbb{C}^N\colon |z|<\frac{r}{3}\right\} =\Delta'' \subset \Delta^*.$$ 
 Thus, Theorem \ref{teorema2} is reduced to the proof of assertions \ref{aser1} and \ref{aser2}. 
  \end{proof}
  
 \subsection*{Proof of Assertion \ref{aser1} }
Suppose that $\mathcal{L}$ pass through a point $({p},\mathfrak{z})$ with $p\in {\mathscr{U}}$ and $|\mathfrak{z}|<r$. Recall --- see \eqref{defh} --- that
$$H_{\mathcal{D}}(x,z)=A(z)-g(x).$$  Then we have 
\begin{align*} H_\mathcal{L}\cap\Delta_p&=\{(p,z)\colon |z|<r,  A(z)-g(p)=A(\mathfrak{z})-g(p)\}\\ 
&=\{(p,z)\colon |z|<r,  A(z)=A(\mathfrak{z})\}\\
&=\{(p,\mathfrak{z}+w)\colon  |\mathfrak{z}+w|<r, w\in W_{\mathcal{D}}\}.
\end{align*}
Thus, since $\mathcal{L}\cap\Delta_p\subset H_\mathcal{L}\cap\Delta_p$, there exists a set 
$\Sigma\subset W_{\mathcal{D}}$ such that 
\begin{align}\label{lapiz1}\mathcal{L}\cap\Delta_p =\{(p,\mathfrak{z}+w)\colon w\in \Sigma\}.\end{align} We clearly have $\Sigma\neq \emptyset$, because $0\in \Sigma$. Let $\Omega\subset W_{\mathcal{D}}$ be defined by the equality
\begin{align}\label{lapiz2}H_\mathcal{L}\cap\Delta_p'=\{(p,\mathfrak{z}+w)\colon w\in\Omega\}.\end{align} It is sufficient to prove that $\Omega \subset \overline{\Sigma}$.
This fact will be a consequence of Lemma \ref{densidad},  so we only have to verify the hypotheses of this lemma. 

We begin with the definition of the set $\Gamma$. 
Let $E$ be the set in the leaf $L_p$ of $\mathcal{F}|_{\Delta_1}$ through $p$  as  given by \eqref{con3a} of Proposition 
\ref{log}. Thus, given $q\in E$ there exists a curve $\beta=\beta_q$ in $L_p$,
 connecting $p$ with $q$ and such that $\ell(\beta)<\epsilon$. 
  Fix $\tau$ in the set $$T^*\colon=\left\{\tau\in T\colon \lambda(\tau)\in \Lambda\right\}.$$ Set 
 $$p_\tau =(0,\dots, 0, \tau)\in\mathbb{C}^M$$  and let  $f_\tau=p_\tau+h_\tau$ be as given by \eqref{con2} of Proposition \ref{log}. By \eqref{con3a} of Proposition \ref{log} the set $E$ is dense in $\mathscr{U}$. Moreover, from
 \eqref{con3b} of Proposition \ref{log} we  deduce that  $\pi^{-1}(\mathscr{U})$ has $p_\tau$ as a limit point. Then 
$\tilde{q}\colon =\pi^{-1}(q)$ can be choosen close to $p_\tau$ so that $\tilde{q}$  belongs to the image $U_\tau$ of  $f_\tau$.
 Then
 $\tilde{q}=f_\tau(\tilde{s})$, where 
$\tilde{s}\in\mathbb{C}^M$ is  close to the origin. By \eqref{con3} of Proposition \ref{log} the set $\mathscr{U}$ is disjoint from the
hypersurface $S$, so that  $p\notin S$ and therefore $q\notin S$. Thus,  it follows from \eqref{con4} of Proposition \ref{log} that
  $\tilde{s}\in\left(\mathbb{C}^*\right)^M$. Let $$\lambda (\tau)=m=(m_1,\dots,m_M)\in\Lambda$$ and 
 let  $\alpha=\alpha_q$ be the elementary $m$-loop
  based
  at $\tilde{s}$. As we have seen in Section \ref{seccionlifting}, the loop $\alpha$ is contained in a leaf of the system
  \begin{equation}\nonumber
\begin{aligned}
 s_1'&=&m_1s_1\\
 &\hspace{0.2cm}\vdots&\\
s_M'&=&m_Ms_M.
\end{aligned}
\end{equation}   By  \eqref{con2} of Proposition \ref{log}, the  system above coincides with the pullback
foliation  ${f_\tau}^*(\mathcal{G})$ in a neighborhood of the origin.
Then $f_\tau\circ\alpha$ is a loop  based at $\tilde{q}$ and contained in the leaf of $\mathcal{G}$ through $\tilde{q}$. Therefore, provided $\tilde{s}$ is small enough,  
 $\pi\circ f_\tau\circ \alpha$ is a loop based at $q$ and contained in the leaf of $\mathcal{F}|_{\Delta_1}$ through $q$, which coincides with the leaf $L_p$ of 
$\mathcal{F}|_{\Delta_1}$ through $p$.
 Consider the loop $$\gamma_q=\beta^{-1}*(\pi\circ f_\tau\circ\alpha)*\beta,$$
 where the symbol $*$ stands for the concatenation of paths.
  This loop is based at $p$ and  contained in $L_p$. We can assume that $\tilde{q}$ is close enough to $p_\tau$ such that the length of 
$\pi\circ f_\tau\circ \alpha$ is less than $\epsilon$, so that the length of $\gamma_q$ is less than $3\epsilon$. The inverse path
$\gamma_q^{-1}$ is also a loop  based at $p$ of length smaller that $3\epsilon$. 
Consider the set $$\Gamma=\left\{\int\limits_{\gamma}\omega \colon \gamma\in\{\gamma_q,\gamma_q^{-1}\}, q\in E,
\tilde{q}\in U_\tau, \tau\in T^* \right\}.$$ 

\noindent\emph{Verification of the  hypotheses of Lemma \ref{densidad}.  } 
  By \eqref{repsilon} we can take $$\delta=\frac{r}{2}-3\mathfrak{K}\epsilon>0.$$ We start with the verification of the hypothesis \eqref{implica} of Lemma \ref{densidad}.   Consider
  $w\in {\Sigma}$ with $\dist (w,\Omega)<\delta$ and $u\in\Gamma$ --- recall that $\Sigma$ and $\Omega$ are defined in 
  \eqref{lapiz1} and \eqref{lapiz2}. Then $$u=\int\limits_{\gamma}\omega,$$ where $\gamma\in\{\gamma_q, \gamma_q^{-1}\}$,
 $q\in E$, $\tilde{q}\in U_\tau$ and  $\tau\in {T^*}$.  Take $w'\in\Omega$  such that $|w-w'|<\delta$. Since  $w'\in\Omega$
  we have $(p,\frak{z}+w')\in\Delta'_p$, hence $$|\frak{z}+w'|<\frac{r}{2}$$ and therefore $$|\frak{z}+w|\le |\frak{z}+w'|+|w-w'|<\frac{r}{2}+\delta.$$  Then 
  $$|\frak{z}+w|+\mathfrak{K}\ell(\gamma)<\frac{r}{2}+\delta+\mathfrak{K}(3\epsilon)=r,$$ so it follows from Lemma \ref{levantamiento}  that the lifting $\tilde{\gamma}$ of
 $\gamma$ to $\mathcal{D}$ starting at $(p,\frak{z}+w)$ is well defined, it  is contained in $\Delta$ and its ending point has the form 
 $$\left(p,\frak{z}+w+\int\limits_{\gamma}\omega\right)=(p,\frak{z}+w+u).$$ Then the point $(p,\frak{z}+w+u)$ belongs to the leaf
  of $\mathcal{F}_\mathcal{D}|_\Delta$ through $(p,\frak{z}+w)$. Observe that this leaf through $(p,\frak{z}+w)$ coincides with 
  the  leaf $\mathcal{L}$ through
  $(p,\frak{z})$, because $w\in {\Sigma}$. Then $(p,\frak{z}+w+u)\in \mathcal{L}$
  and therefore  $w+u\in {\Sigma}$. Hypothesis \eqref{implica} is verified.\\
 It remains to prove the condition on $\Gamma$ required by Lemma \ref{densidad}.  
Since $0\in {\Sigma}$,  from Hypothesis \eqref{implica} proved above we have that  $0+u\in {\Sigma}$ for all $u\in\Gamma$. Then
$\Gamma\subset {\Sigma}$ and, in particular, 
\begin{align}\label{lapicero}\Gamma\subset W_\mathcal{D}.\end{align}
Recall that, given $\tau\in T^*$ with $\lambda(\tau)=m$ and $q\in E$ with $\tilde{q} f_\tau (\tilde{s})\in U_\tau$, 
$$\pm\int\limits_{\gamma_q}\omega\in \Gamma.$$   But
 $$\int\limits_{\gamma_q}\omega =\int\limits_{\beta^{-1}*\pi\circ f_\tau\circ\alpha *\beta}\omega =\int\limits_{\pi\circ f_\tau\circ\alpha}\omega 
 =\int\limits_{ f_\tau \circ\alpha}\tilde{\omega}=
 \int\limits_{ h_\tau \circ\alpha}\tilde{\omega}_\tau=\int\limits_{\alpha}h_\tau^*(\tilde{\omega}_\tau)=\rho_\tau(\tilde{s}), $$
 where  \begin{equation}\rho_\tau\colon=\rho(h_\tau^*(\tilde{\omega}_\tau),m) \label{rot}\end{equation} is as given by \eqref{coment} of  Proposition \ref{retornos}.    
   Then  \begin{align}\label{templanza}\tau\in T^*, q\in E, \tilde{q} = f_\tau (\tilde{s}) \implies \pm \rho_\tau(\tilde{s})\in \Gamma.
   \end{align}  Since $p_\tau$ is a limit point of
    $\pi^{-1}(\mathscr{U})$, 
    $$ \mathscr{U}_\tau \colon = f_\tau^{-1}\left(\pi^{-1}(\mathscr{U})\cap U_\tau\right)$$
    is a nonempty open subset of $\mathbb{C}^M$  accumulating at the origin. Recall that 
    $q$ can take any value in the dense subset $E$ of $\mathscr{U}$. Then  
 $\tilde{q}$ can be arbitrarily chosen  in a dense subset of $\pi^{-1}(\mathscr{U})\cap U_\tau$, so that, in view of \eqref{templanza}, we find a  dense
 subset $C_\tau\subset \mathscr{U}_\tau$ such that
 \begin{equation}\label{porta}\pm\rho_\tau(\tilde{s})\in \Gamma,\;\tilde{s}\in C_\tau.\end{equation} Then,
 since $C_\tau$ is  dense in $\mathscr{U}_\tau$ and  $\Gamma\subset W_{\mathcal{D}}$, we have 
 \begin{equation}\label{porta1}\pm\rho_\tau(\tilde{s})\in W_{\mathcal{D}},\;\tilde{s}\in \mathscr{U}_\tau.\end{equation}
 Thus, 
  the set
 \begin{align}\label{set+}\left\{\rho_\tau(\tilde{s})\colon \tau\in {T}^*, \tilde{s}\in\mathscr{U}_\tau\right\}\end{align} is contained in $W_\mathcal{D}$.
  Let $W\subset W_\mathcal{D}$ be the complex vector space spanned by the set in \eqref{set+}.
Given $\epsilon>0$, by  Lemma \ref{base real} 	 we can find a real
   $\epsilon$-basis of $W$ of the form $$\rho_{\tau_1}(c_1),\dots,\rho_{\tau_{n}}(c_{n}),$$
 where $\tau_1,\dots,\tau_{n}\in {T}^*$ and $c_1\in\mathscr{U}_{\tau_1},\dots,c_{n}\in \mathscr{U}_{\tau_n}$. 
 Since each $C_\tau$ is dense in $\mathscr{U}_\tau$, we can find $c_1'\in C_{\tau_1},\dots,c_{n}'\in C_{\tau_n}$
  such that
 $$\rho_{\tau_1}(c'_1),\dots,\rho_{\tau_{n}}(c'_{n})$$ is a real $\epsilon$-basis of $W$. Since each
  $\pm\rho_{\tau_j}(c'_j)$  belongs to $\Gamma$ --- see \eqref{porta},  we conclude that
   $$\{\pm \rho_{\tau_1}(c'_1),\dots,\pm\rho_{\tau_{n}}(c'_{n})\}\subset \Gamma.$$   Therefore, in order to complete
   the  verification of the
   condition on $\Gamma$ required by Lemma \ref{densidad}, 
   it suffices to show that    $W_\mathcal{D}= W$. With this objective in mind, since $W_\mathcal{D}\supset W$, 
it is enough to prove that the coefficients of the monomials in $\mathfrak{M}$ 
 --- they span $W_\mathcal{D}$ --- are all contained in $W$. 
Let $\theta_{\mathfrak{L}}^{\boldsymbol{\mathfrak{ij}}}t^{\mathfrak{L}}dt_{\boldsymbol{\mathfrak i}}dt_{\boldsymbol{\mathfrak j}}$ be a monomial 
in $\mathfrak{M}$. Since --- see \eqref{teton} --- we have $\theta_{\mathfrak{L}}^{\boldsymbol{\mathfrak{ij}}}=\Theta_{\mathfrak{L}}^{\boldsymbol{\mathfrak{ij}}}(0)$,
it suffices to show that  $\Theta_{\mathfrak{L}}^{\boldsymbol{\mathfrak{ij}}}(\tau)$ belongs to $W$ for all $\tau\in\mathbb{C}$. By condition \eqref{star}, there exist
 $\mathbb{C}$-linearly independent $\boldsymbol{\mathfrak{i}}$-elementary vectors  
 $\mathfrak{m}_1,\dots, \mathfrak{m}_{M-1}$ in $\Lambda$  such that 
 $$\mathfrak{m}_1\cdot \mathfrak{L}_{\boldsymbol{\mathfrak{ij}}}=\dots=\mathfrak{m}_{M-1}\cdot\mathfrak{L}_{\boldsymbol{\mathfrak{ij}}}=0.$$
   Observe that the monomial
   $s^{\mathfrak{L}}ds_{\boldsymbol{\mathfrak i}}ds_{\boldsymbol{\mathfrak j}}$ belongs to the set 
  $\mathcal{M} (\mathfrak{L}_{\boldsymbol{\mathfrak{ij}}},{\boldsymbol{\mathfrak i}})$ as defined in \eqref{monomios} --- with $s$ instead of $x$. Then, by \eqref{item33} of  Proposition \ref{retornos}, 
the coefficient $\Theta_{\mathfrak{L}}^{\boldsymbol{\mathfrak{ij}}}(\tau)$ of $s^{\mathfrak{L}}ds_{\boldsymbol{\mathfrak i}}ds_{\boldsymbol{\mathfrak j}}$ is a linear combination of the coefficients 
$$a_{\mathfrak{L}_{\boldsymbol{\mathfrak{ij}}}}(\tilde{\omega}_\tau,\mathfrak{m}_1), 
\dots, a_{\mathfrak{L}_{\boldsymbol{\mathfrak{ij}}}}(\tilde{\omega}_\tau,\mathfrak{m}_{M-1})$$ of
 $y^{\mathfrak{L}_{\boldsymbol{\mathfrak{ij}}}}$ in the series 
 $$\rho(\tilde{\omega}_\tau,\mathfrak{m}_1 ),\dots,\rho(\tilde{\omega}_\tau,\mathfrak{m}_{M-1} ).$$ Then it suffices to show that 
 the coefficients $$a_{\mathfrak{L}_{\boldsymbol{\mathfrak{ij}}}}(\tilde{\omega}_\tau,\mathfrak{m}_1), 
\dots, a_{\mathfrak{L}_{\boldsymbol{\mathfrak{ij}}}}(\tilde{\omega}_\tau,\mathfrak{m}_{M-1})$$
 belong to $W$ for all $\tau\in\mathbb{C}$.
 We continue the argumentation with the  coefficient $a_{\mathfrak{L}_{\boldsymbol{\mathfrak{ij}}}}(\tilde{\omega}_\tau,\mathfrak{m}_1)$ ---
 the proof is exactly the same for the other cases.
 By \eqref{item22} of Proposition \ref{retornos}, we have that $a_{\mathfrak{L}_{\boldsymbol{\mathfrak{ij}}}}(\tilde{\omega}_\tau,\mathfrak{m}_1)$
 is  a linear combination of the coefficients of the monomials in 
 $\mathcal{M} (\mathfrak{L}_{\boldsymbol{\mathfrak{ij}}},{\boldsymbol{\mathfrak i}})$. Observe that each monomial   
 $\Theta_L^{ij}(\tau)s^Lds_ids_j$ with $s^Lds_ids_j$ in $\mathcal{M} (\mathfrak{L}_{\boldsymbol{\mathfrak{ij}}},{\boldsymbol{\mathfrak i}})$
 is such that --- see \eqref{gradod}  --- $$|L|=|\mathfrak{L}_{\boldsymbol{\mathfrak{ij}}}|-2=|\mathfrak{L}|\le d.$$ Then ---
 see \eqref{tetita} ---
 the coefficients of the monomials in $\mathcal{M} (\mathfrak{L}_{\boldsymbol{\mathfrak{ij}}},{\boldsymbol{\mathfrak i}})$ have degree $\le d$ as functions of $\tau$, and therefore
 $a_{\mathfrak{L}_{\boldsymbol{\mathfrak{ij}}}}(\tilde{\omega}_\tau,\mathfrak{m}_1)$ has degree $\le d$ as a polynomial function of  $\tau$. Thus, 
 by Lemma \ref{polinomionulo} and condition \eqref{lamb2},
 it is enough to prove that $a_{\mathfrak{L}_{\boldsymbol{\mathfrak{ij}}}}(\tilde{\omega}_\tau,\mathfrak{m}_1)$ belongs to $W$ for all $\tau$
 in $\lambda^{-1}(\mathfrak{m}_1)$. Consider any $        \tau\in T$ such that  $\lambda (\tau)=\mathfrak{m}_1$. 
By  \eqref{con2} of Proposition \ref{log}, the diffeomorphism $h_\tau$ is $(d+1)$-tangent to the identity, hence the expansion series of $d\tilde{\omega}_\tau$ and $dh_\tau^*(\tilde{\omega}_\tau )=h_\tau^*(d\tilde{\omega}_\tau )$ coincides
up to order $d$. Then, by  \eqref{coment} of Proposition \ref{retornos}, we have that  the series 
$\rho(\tilde{\omega}_\tau,\mathfrak{m}_1 )$ and $\rho(h_\tau^*(\tilde{\omega}_\tau),\mathfrak{m}_1 )$ coincides up to order $d+2$. Thus, since
$|{\mathfrak{L}_{\boldsymbol{\mathfrak{ij}}}}|=|{\mathfrak{L}}| +2\le d+2$, we conclude that 
the coefficient $a_{\mathfrak{L}_{\boldsymbol{\mathfrak{ij}}}}(\tilde{\omega}_\tau,\mathfrak{m}_1)$
coincides with the coefficient  $a_{\mathfrak{L}_{\boldsymbol{\mathfrak{ij}}}}(h_\tau^*(\tilde{\omega}_\tau),\mathfrak{m}_1)$
 of the series
$$\rho(h_\tau^*(\tilde{\omega}_\tau),\mathfrak{m}_1 ),$$ which is equal to the series $\rho_\tau$ defined in \eqref{rot}. Since --- from the definition of $W$ --- we have that $\rho_\tau(s)$ belongs to $W$ for all $s$ in an open set that accumulates in the origin in $\mathbb{C}^N$, 
by Lemma \ref{coeficientes}  the coefficients of $\rho_\tau(s)$ belong to $W$, so that $a_{\mathfrak{L}_{\boldsymbol{\mathfrak{ij}}}}(\tilde{\omega}_\tau,\mathfrak{m}_1)$ belongs 
to $W$. \qed

\subsection*{Proof of Assertion \ref{aser2}}
By hypothesis, $\mathcal{L}$ contains a point $(p,\mathfrak{z})$ with  $p\in {\mathscr{U}}$ and $|\mathfrak{z}|<r/3$. 
Let $E$ be the set in the leaf $L_p$ of $\mathcal{F}|_{\Delta_1}$ through $p$ given by \eqref{con3} of Proposition \ref{log}.
It is enough to prove that $\mathcal{L}$ meets $\Delta'_q$ for all $q\in E$. Fix $q\in E$. By \eqref{con3} 
of Proposition \ref{log}
there exists a path $\gamma$ in $L_p$ connecting $p$ with $q$ such that $\ell(\gamma)<\epsilon$. From \eqref{repsilon} we obtain
$$|\mathfrak{z}|+\mathfrak{K}\ell(\gamma)<\frac{r}{3}+\mathfrak{K}\epsilon<\frac{r}{2}. $$Therefore,
it follows from Lemma 
\ref{levantamiento} that the lifting $\tilde{\gamma}$ of $\gamma$ to $\mathcal{D}$ starting at $(p,\mathfrak{z})$ is contained in 
$$\left\{(x,z)\in \Delta\colon |z|<\frac{r}{2}\right\}.$$ Then the ending point of $\tilde{\gamma}$ belongs to the intersection of $\mathcal{L}$
with $\Delta'_q$. The proof of Theorem \ref{teorema2} is complete. \qed 
\begin{rem}\label{remark2} In the proof of Theorem \ref{teorema2}, given $\Delta=\Delta_1\times\Delta_2$, where $\Delta_1$ and $\Delta_2$ are balls centered at the origins of $\mathbb{C}^M$ and $\mathbb{C}^N$, such that $\overline{\Delta}$ is 
contained in the domains of definition of $\mathcal{D}$ and $H_\mathcal{D}$, we have constructed the set $\Delta^*$ containing a set $\Delta''$ of the form
$$\Delta''={\mathscr{U}}\times\left\{z\in\mathbb{C}^N\colon |z|<\frac{r}{3}\right\},$$ where we have the following facts.
\begin{enumerate}
\item The number $r$ is the radius of the ball $\Delta_2$.
\item The set  ${\mathscr{U}}\subset \Delta_1$ is given by  \eqref{con3}  of Proposition \ref{log}; it depends on the foliation $\mathcal{F}$, the ball $\Delta_1$ and the election --- see \eqref{repsilon} --- of a positive number $\epsilon>0$ such that $6\mathfrak{K}\epsilon<r$ 
--- recall that $\mathfrak{K}$ is a bound for $|\omega|$ on $\Delta_1$.  
\end{enumerate}
Suppose that $\Delta_1$ is fixed. Since $\omega$ depends only on 
$(x_1,\dots,x_M)\in\mathbb{C}^M$, we can take $\Delta_2$ having any radius $r>0$. If we suppose $r>6\mathfrak{K}$, we can choose $\epsilon=1$ and  therefore the set ${\mathscr{U}}$ can be considered fixed and independent of $r$. 
This fact will be used in the proof of Theorem \ref{teorema3}.

\end{rem}
\subsection*{Proof of Theorem \ref{teorema1}} Let $F$ be a germ of meromorphic first integral of $\mathcal{D}$ in $(\mathbb{C}^{M+N},0)$.
 Let $Z$ be as given by Theorem \ref{teorema2}.
Take a neighborhood $\Delta$ of the origin in $\mathbb{C}^{M+N}$ where $F$ and $Z$ are defined and let $\Delta^*\subset\Delta$ be the corresponding set as given by Theorem \ref{teorema2}. Let $p\in \Delta\backslash S$ outside the indeterminacy set of $F$ and let $F(p)=c$. 
Since $F$ is a meromorphic first integral, $F$ takes the value $c$ along the leaf $\mathcal{L}$ of $Z|_\Delta$ through $p$. Then, since Theorem \ref{teorema2}
guarantees that $\left({H_\mathcal{D}}|_{\Delta^*}\right)^{-1}(c)\subset \overline{\mathcal{L}}$,  we have that $F$ is constant on 
$$\left({H_\mathcal{D}}|_{\Delta^*}\right)^{-1}(c).$$ Since the point $p$ can be arbitrarily chosen on an open set accumulating the origin of 
$\mathbb{C}^{M+N}$, we conclude that $F$ --- whenever it is defined --- is constant along the levels of 
${H_\mathcal{D}}|_{\Delta^*}$. 
Since $$H_\mathcal{D}\colon (\mathbb{C}^{M+N},0)\to (\mathbb{C}^\kappa,0)$$ is a submersion,
there exists a neighborhood $U$ of the origin in $\mathbb{C}^\kappa$ and a holomorphic embedding
$$g\colon U\to \Delta^*$$ such that $H_\mathcal{D}\circ g=\id$. Consider the meromorphic function
\begin{align*} f:=F\circ g \colon U\to \mathbb{C}\end{align*} We must show that $F=f\circ H_\mathcal{D}$. 
Take $\zeta\in\Delta^*$ such that $H_\mathcal{D}(\zeta)\in U$. 
Then \begin{equation} f\circ H_\mathcal{D}(\zeta)=F(g\circ H_\mathcal{D}(\zeta)).\label{forno}\end{equation} 
Since $$H_\mathcal{D}(g\circ H_\mathcal{D}(\zeta))=H_\mathcal{D}\circ g (H_\mathcal{D}(\zeta))=H_\mathcal{D}(\zeta),$$
the points $g\circ H_\mathcal{D}(\zeta)$ and $\zeta$ --- both belonging to $\Delta^*$ --- are contained in the same level of $H_\mathcal{D}$,  so 
$F$ takes the same value at these two points. Therefore
$$F(g\circ H_\mathcal{D}(\zeta))=F(\zeta)$$ and by \eqref{forno} we obtain 
$$F(\zeta)=f\circ H_\mathcal{D}(\zeta).$$ The converse assertion of Theorem \ref{teorema1} is trivial, so the proof is complete.\qed

\subsection*{Proof of Theorem \ref{teorema3}}
Item \eqref{item111}  follows directly from  \eqref{detalle1} of Proposition \ref{detalles}, so we only prove Item \eqref{item222}.
 Suppose that
$W_{\mathcal{D}}=\mathbb{C}^N$.  From the proof of Theorem \ref{teorema2}, the holomorphic foliation defined by $Z$ is the lifting
 $\mathcal{F}^{\mathcal{D}}$ of a holomorphic foliation $\mathcal{F}$, which is given by  Proposition \ref{log}. 
Let $\mathcal{L}$ be a leaf of $\mathcal{F}^{\mathcal{D}}$  through a point
  $$(\mathfrak{x},\mathfrak{z})\in\mathbb{C}^{M+N},\; \mathfrak{x}\notin  S\cup(\omega)_\infty$$
  and let $U\subset\mathbb{C}^{M+N}$ be any open set. We must show that $\mathcal{L}$ meets $U$. Then, if we take
    $$(\mathfrak{x}',\mathfrak{z}')\in U,\; \mathfrak{x}'\notin  S\cup (\omega)_\infty,$$
  it is enough to show that $\overline{\mathcal{L}}$ contains the leaf $\mathcal{L}'$ of $\mathcal{F}^{\mathcal{D}}$ through the point
  $(\mathfrak{x}',\mathfrak{z}')$. Let $\Delta_1$ be a ball centered at the origin of $\mathbb{C}^M$ and disjoint of the pole set $(\omega)_\infty$. Let $\mathscr{U}\subset \Delta_1$ be the set as given by  \eqref{con3}  of Proposition \ref{log} for $\Delta=\Delta_1$ and $\epsilon=1$.  By Proposition \ref{log}, the leaf $L$ of $\mathcal{F}$ through the point $\mathfrak{x}$ is dense in
 $\mathbb{C}^M$. Then we can take a path $\gamma$ in $L$ connecting $\mathfrak{x}$ with a point $p\in \mathscr{U}$. Since $L$ meets $\Delta_1$ and $\Delta_1$ is disjoint of $(\omega)_\infty$, the leaf $L$ is  not contained in $(\omega)_\infty$. Then,
 since $\dim L=1$,  the points of $(\omega)_\infty$
 that belong to $L$ conform a discrete set in the intrinsic topology of $L$, hence we can assume that the path $\gamma$
 is disjoint of $(\omega)_\infty$. Thus, the lifting of ${\gamma}$ to $\mathcal{D}$ starting at $(\mathfrak{x},\mathfrak{z})$  is well defined and its ending point
 belongs to the set $$\Sigma_p\colon=\{p\}\times\mathbb{C}^N.$$ We have proved that $\mathcal{L}$ contains some point
  $\zeta\in\Sigma_p$  for some $p\in \mathscr{U}$. Exactly the same arguments show that $\mathcal{L}'$ contains some point
   $\zeta'\in\Sigma_{p'}$
  for some $p'\in \mathscr{U}$. As in the proof of Theorem \ref{teorema2}, let $\mathfrak{K}$ be a bound for $|\omega|$ on $\Delta_1$. Take $r>6\mathfrak{K}$  such that $\zeta$ and $\zeta'$ are contained in the set
  $$\Delta''=\mathscr{U}\times\{z\in\mathbb{C}^N\colon |z|<\frac{r}{3}\}. $$ Set
  $$\Delta=\Delta_1\times\{z\in\mathbb{C}^N\colon |z|<r\} $$ 
   and let $\Delta^*\subset \Delta$ be as given by Theorem \ref{teorema2}. 
  By remark \ref{remark2} we have $\Delta''\subset\Delta^*$. Then,
   since $W_{\mathcal{D}}=\mathbb{C}^N$,  the leaf of $\mathcal{F}^\mathcal{D}|_{\Delta^*}$ passing through $\zeta$ is dense in $\Delta^*$. It follows that  $\mathcal{L}$ is dense in $\Delta''$, so   
   $\overline{\mathcal{L}}$ contains $\zeta'\in\Delta''$ and therefore $\overline{\mathcal{L}}$ contains ${\mathcal{L}'}$.\qed
 \subsection*{Proof of Corollary \ref{isotrop}}
Since 
$\omega$ is non-integrable, then the  series $$d\omega=\sum c_K^{ij}dx_idx_j$$ is not null. Therefore the linear subspace of $\mathbb{C}$ generated by the coefficients of $d\omega$ is the total space $\mathbb{C}$, that is, $W_\mathcal{D}=\mathbb{C}$. Then the result follows directly from Theorem \ref{teorema3}. \qed

 \subsection*{Proof of Theorem \ref{legendrian}}
Let $Z$,  $X$ and $S$ be as given by Theorem \ref{teorema2} and Proposition \ref{detalles} for the distribution
 \begin{equation}
dz=ydx-xdy
\end{equation}
on $\mathbb{C}^3$. By Theorem \ref{teorema3}, the leaves of $Z$ 
outside $$\tilde{S}\colon=\{(x,y,z)\colon (x,y)\in S\}$$ are dense in $\mathbb{C}^3$.  
The vector field $Z$ is the lifting of a polynomial vector field
$$X=A(x,y)\frac{\partial}{\partial x}+B(x,y)\frac{\partial}{\partial y}$$ on $\mathbb{C}^2$. 
Recall that, in the coordinates $(x_1,y_1,\dots,x_m,y_m, z)\in\mathbb{C}^{2m+1}$,  the distribution $\xi$ is defined by the equation
  \begin{equation}
dz=(y_1dx_1-x_1dy_1)+\dots+(y_mdx_m-x_mdy_m).
\end{equation}For each $j=1,\cdots,m$ define  the vector field $X_j$ on
$\mathbb{C}^{2m}$  by the expression
$$X_j=A(x_j,y_j)\frac{\partial}{\partial x_j}+B(x_j,y_j)\frac{\partial}{\partial y_j}.$$
Let $Z_j$ be the lifting to $\xi$ of the vector field $X_j$, that is,
$$Z_j=A(x_j,y_j)\frac{\partial}{\partial x_j}+B(x_j,y_j)\frac{\partial}{\partial y_j}+\Big[y_jA(x_j,y_j)-x_jB(x_j,y_j)\Big]\frac{\partial}{\partial z}.$$
It is easy to see that the vector fields $Z_j$ are pairwise commutative and generically linearly independent. Then the $Z_j$ define an $m$-dimensional singular holomorphic foliation $\mathfrak{F}_m$ on $\mathbb{C}^{2m+1}$, which is  clearly Legendrian for $\xi$. 
From  \eqref{detalle3} of Proposition \ref{detalles},  there is a polynomial $P\in\mathbb{C}[x,y]$ such that 
$$\tilde{S}=\{(x,y,z)\colon P(x,y)=0\}.$$ 
Define the hypersurface $S_m\subset \mathbb{C}^{2m+1}$ by the equation
$$P(x_1,y_1)\cdots P(x_m,y_m)=0. $$ Observe that $S_m$  is invariant by the vector fields $Z_j$, so that $S_m$ is invariant by $\mathfrak{F}_m$. We will prove that the leaves of $\mathfrak{F}_m$ outside $S_m$
are dense in $\mathbb{C}^{2m+1}$. 
We prove this property by induction on $m$. The property is clearly true for $m=1$; in this case $\mathfrak{F}_1$ is equivalent to the 
foliation defined by the vector field $Z$.
Suppose that the property holds for some $m\ge 1$. Let $\mathfrak{F}_m'$ be the foliation on the space $$(x_1,y_1,\dots,x_{m+1},y_{m+1},z)$$
 generated by the vector fields 
$$Z_j=A(x_j,y_j)\frac{\partial}{\partial x_j}+B(x_j,y_j)\frac{\partial}{\partial y_j}+\Big[y_jA(x_j,y_j)-x_jB(x_j,y_j)\Big]\frac{\partial}{\partial z},$$
for $j=1,\dots,m$ --- we are considering the $Z_j$ as vector fields on $\mathbb{C}^{2m+3}$. Since these vector fields does not depend on the variables $(x_{m+1},y_{m+1})$, they leave invariant each codimension $2$ plane 
$$\Sigma(a,b)\colon x_{m+1}=a,\;y_{m+1}=b,$$  where $(a,b)\in\mathbb{C}^2$. Moreover, the foliation generated by the vector fields $Z_1,\dots,Z_m$
on each $\Sigma(a,b)$ is equivalent to the foliation $\mathfrak{F}_m$ via the map 
$$(x_1,y_1,\dots,x_m,y_m,a,b,z)\mapsto (x_1,y_1,\dots,x_m,y_m,z).$$ Let $\mathcal{L}$ be a leaf of $\mathfrak{F}_{m+1}$
outside $S_{m+1}$. Since $\mathfrak{F}_{m+1}$
is generated by the vector fields $Z_1,\dots,Z_m$ --- which generate  $\mathfrak{F}_m'$ --- together with the vector field $Z_{m+1}$, we have that the intersection
of $\mathcal{L}$ with a plane $\Sigma (a,b)$ --- if this intersection is not empty --- is an invariant set of $\mathfrak{F}_m$. The property 
$\mathcal{L}\cap S_{m+1}=\emptyset$ means that the points in $\mathcal{L}$ do not satisfy the equation
\begin{equation}\label{feo}P(x_1,y_1)\cdots P(x_{m+1},y_{m+1})=0.\end{equation}
Then the points in $\mathcal{L}$ do not satisfy the equation $$P(x_1,y_1)\cdots P(x_{m},y_m)=0$$ and, in particular,
the points in $\mathcal{L}\cap \Sigma(a,b)$ do not belong to $S_{m}$. Therefore, we can apply the inductive hypothesis to any
leaf of $\mathfrak{F}_m$ contained in $\mathcal{L}\cap \Sigma(a,b)$ to conclude that 
  $\mathcal{L}\cap \Sigma (a,b)$ is dense in $\Sigma (a,b)$. Then, to guarantee the density of $\mathcal{L}$ it suffices to show that
 $\mathcal{L}$ meets  $\Sigma (a,b)$ for a dense set of points $(a,b)$ in $\mathbb{C}^2$, that is, we must prove that 
 the image of $\mathcal{L}$ by the projection
 $$(x_1,y_1,\dots,x_m,y_m,x_{m+1},y_{m+1},z)\mapsto (x_{m+1},y_{m+1})$$ is dense in $\mathbb{C}^2$. Fix a point
$p\in\mathcal{L}$. This point belongs to some three-dimensional plane 
$$\Sigma\colon x_1=a_1,y_1=b_1,\dots,x_m=a_m,y_m=b_m,$$
for some $a_1,b_1,\dots,a_m,b_m\in\mathbb{C}$.  The vector field
 $Z_{m+1}$ leave the plane $\Sigma$ invariant and the restriction ${Z_{m+1}}|_{\Sigma}$ is equivalent to the vector field $Z$ on 
 $\mathbb{C}^3$
 via the map 
 $$(a_1,b_1,\dots,a_m,b_m,x,y,z)\mapsto (x,y,z).$$ Then the leaf $\mathfrak{L}$ of ${Z_{m+1}}|_{\Sigma}$ through $p$ can be considered as a leaf
 of $Z$. Since $p$ does not satisfy Equation \ref{feo}, we have that $p$ does not satisfy
 $$P(x_{m+1},y_{m+1})=0.$$ This means that, viewed as a point in $\mathbb{C}^3$,   the point $p$ is not contained in the set $\tilde{S}$. Therefore ---
 by the density property of $Z$ --- the leaf $\mathfrak{L}$ is dense in $\mathbb{C}^3$, so that image of $\mathfrak{L}$ by the projection 
 $$(x,y,z)\mapsto (x,y)$$ is dense in $\mathbb{C}^2$. Since $\mathfrak{L}$ is contained in $\mathcal{L}$, we conclude that
 the image of $\mathcal{L}$ by the projection
 $$(x_1,y_1,\dots,x_m,y_m,x_{m+1},y_{m+1},z)\mapsto (x_{m+1},y_{m+1})$$ is dense in $\mathbb{C}^2$. \qed

\section{Final remarks}

Consider  the distribution $\mathcal{D}$ defined by \eqref{mainsystem} and let $$H_{\mathcal{D}}: (\mathbb{C}^{N+M},0)\rightarrow (\mathbb{C}^{\kappa},0)$$ be the associated holomorphic submersion
 defined in \eqref{defh}. Denote by $\mathcal{F}(H_{\mathcal{D}})$ the regular foliation defined by the levels of $H_{\mathcal{D}}$. The following statement is a reinterpretation of Theorem \ref{teorema1}.

\begin{cor}
The field of meromorphic first integrals of $\mathcal{D}$ and that of $\mathcal{F}(H_{\mathcal{D}})$ are the same.
\end{cor}

\begin{rem}\label{totally-non-integrable}
If $H_{\mathcal{D}}$ is constant, that is, if $\mathcal{D}$ has only constant first integrals, then the distribution is completely non-integrable. In fact, consider the vector fields $X_1, \dots, X_M$, obtained by lifting to $\mathcal{D}$  the canonical basis on $\mathbb{C}^M$. Set $Y_{ij}\colon=[X_i, X_j]$, $i<j$. A simple computation shows that
$$
Y_{ij}=\Big(\underbrace{0}_{\in \mathbb{C}^M}, \sum_K{c_{K}^{ij}}x^K\Big). 
$$
In particular, we have 
\begin{enumerate}
\item $[Y_{ij}, Y_{kl}]=0$ for all $i<j$, $k<l$;
\item $[X_l,Y_{ij}]= \left(0, \partial_l(\sum_K{c_{K}^{ij}}x^K)\right)$.
\end{enumerate}
 For $i,j$ and $K_0=(k_1,\dots,k_M)$ fixed, we consider the vector field $X^{ij}_{K_0}$ obtained from $Y_{ij}$ by successive applications of the Lie bracket in the
following
 form:
\begin{align*}
X^{ij}_{K_0}=\underbrace{[X_M, \dots,[X_M }_{k_M},&\underbrace{[X_{M-1},\dots,[X_{M-1}}_{k_{M-1}},\dots  , \underbrace{[X_1,\dots, [X_1}_{k_1},Y_{ij}]\cdots].
\end{align*}
It follows from item $(2)$ that  $$X^{ij}_{K_0}=\Big(0, \partial_{K_0}\Big(\sum_K{c_{K}^{ij}}x^K\Big)\Big).$$
On the other hand we have
$$
\frac{1}{K_0 !}X^{ij}_{K_0}(0) = (0, c_{K_0}^{ij}).
$$
Therefore, since  the $c_{K}^{ij}$ span all $\mathbb{C}^N$ we conclude our claim.
\end{rem}

\begin{rem}
Let $\mathcal{D}$ be a distribution on $\mathbb{P}^{M+N}$ and suppose that in some affine chart $(x_1, \dots, x_M, z_1, \dots, z_N)$, the distribution is defined by a system 
\begin{equation*}
dz_i=\omega_i, \hspace{0.4cm} i=1, \dots, N, 
\end{equation*}
where each  $\omega_i$ is a rational $1$-form in the variables $(x_1, \dots, x_M)$. Then, it follows from the construction of $H_{\mathcal{D}}$ that its coordinates are Liouvillian functions, so the foliation $\mathcal{F}(H_{\mathcal{D}})$ is global  on $\mathbb{P}^{M+N}$ and defined by Liouvillian first integrals --- see \cite{Singer}. 
\end{rem}

The problem of describing the Kernel (algebra of regular first integrals) of a regular distribution has been studied by many autors, see for example \cite{Bonnet}, \cite{Nagata}, \cite{Nowicki} and the references therein. The main result of \cite{Bonnet} shows that for a distribution defined by a family of holomorphic vector fields $\{X_i\}$ there is a holomorphic vector field tangent to the distribution with the same algebra of regular first integrals. Clearly, the field of meromorphic first integral can be extremely large, however, as a consequence of Theorem \ref{teorema2} we can obtain a similar result when the distribution has separated variables.

\begin{cor}
Let $\mathcal{D}$ be the distribution on $(\mathbb{C}^{M+N},0)$ defined by \eqref{mainsystem}. Then there exists a vector field tangent to $\mathcal{D}$ which has the same field of meromorphic first integrals as $\mathcal{D}$.
\end{cor}

As can be expected, the separation of variables --- that is, the existence of coordinates that reduces the distribution to the form \eqref{mainsystem} --- is not always possible, as we see in the following example.

\begin{example}
Let $\mathcal{D}$ be the distribution on $(\mathbb{C}^4,0)$ generated by the vector fields 
$$
X_1=\partial_z, \hspace{0.5cm} X_2 = x\partial_x + y\partial_y + (1+z) \partial_w,
$$
which is a regular rank $2$ non-integrable distribution. It is easy to check that $\mathcal{D}$ 
is tangent to the codimension one foliation $\mathcal{F}: xdy - ydx =0$, that is,   the function $y/x$
 is a rational first integral for $\mathcal{D}$. This means that the field of meromorphic first integrals near every point $p\in \mathbb{C}^4$ is not reduced to the constants. Thus, if the separation of variables 
can be achieved in $(\mathbb{C}^4,p)$, it follows from Remark \ref{merosub} that $\mathcal{D}$
is tangent to a regular codimension one holomorphic foliation $\tilde{\mathcal{F}}$ in 
$(\mathbb{C}^4,p)$. If $\tilde{\mathcal{F}}$ were different from $\mathcal{F}$, the intersection of 
these foliations would contain $\mathcal{D}$ as a sub-distribution, whence $\mathcal{D}$ would be integrable. Then  $\tilde{\mathcal{F}}= \mathcal{F}$ in  $(\mathbb{C}^4,p)$ and therefore
$\mathcal{F}$ is regular at $p$. We conclude that  the separation of variables in  $(\mathbb{C}^4,p)$
is impossible
 if $p \in \sing(\mathcal{F})$.\\
 On the other hand, if $p \notin \sing(\mathcal{F})$, the separation of variables in  $(\mathbb{C}^4,p)$
is actually possible.   For instance, near $p=(1,0,0,0)$ the distribution is given by the system
$$
dy = y \frac{dx}{x}, \quad dw =(1+z)\frac{dx}{x}.
$$ 
Then, after the change $x\leftarrow e^x$ the system is equivalent to 
$$
dy = y dx, \quad dw =(1+z)dx,
$$
which, by the new change $y\leftarrow ye^x$, is taken to 
$$
dy = 0, \quad dw =(1+z)dx.
$$
\end{example}

\noindent\textbf{Question:} Let $\mathcal{D}$ be a (singular) holomorphic  distribution on a complex manifold $M$. Is it always possible to separate variables at a generic point?.

\section{Proof of Proposition \ref{log}}\label{seccion3}

This section is devoted to the proof of Proposition \ref{log}. The proof is quite technical and will be organized with the aid of some lemmas that will be proved  along the way: lemmas \ref{nuca}, \ref{assertion1}, \ref{densote}, \ref{assertion2},  
 \ref{pila1}, \ref{pila2}, \ref{pila3} and \ref{pila4}. Before proceeding with the proof of this proposition, we need three previous lemmas.

 \begin{lem}\label{zar2} Given $m\in\mathbb{N}$, if we set $n=2m+1$ we have that,  for $(\mu_{ij})$ chosen in a dense set of the space of $m\times n$ complex matrices, the vectors
\begin{equation}\nonumber 
\mu_j=\left(e^{2\pi \sqrt{-1}\mu_{1j}}, e^{2\pi \sqrt{-1}\mu_{2j}},\dots,e^{2\pi \sqrt{-1}\mu_{mj}} \right), \quad j=1,\dots,n
\end{equation} generate  a  dense subgroup  in the multiplicative group $(\mathbb{C}^*)^{m}$.
 \end{lem}
 \proof
 If $k_1,\dots,k_n\in\mathbb{Z}$,  we have 
 \begin{align}\nonumber
 \mu_1^{k_1}\mu_2^{k_2}\dots\mu_n^{k_n}=\left(e^{2\pi \sqrt{-1}\sum\limits_{j=1}^n\mu_{1j}k_j}, e^{2\pi \sqrt{-1}\sum\limits_{j=1}^n\mu_{2j}k_j},\dots, e^{2\pi \sqrt{-1}\sum\limits_{j=1}^n\mu_{mj}k_j} \right).
 \end{align}
 Then $\mu_1,\dots,\mu_n$ generate a dense subgroup of $(\mathbb{C}^*)^{m}$ if the set 
 \begin{align}\nonumber
 \Bigg\{\bigg(\sum\limits_{j=1}^n\mu_{1j}k_j, \dots, \sum\limits_{j=1}^n\mu_{mj}k_j \bigg)
 \colon k_1,\dots,k_n\in\mathbb{Z}\Bigg\}
 =\mathbb{Z}\mu_1+\dots +\mathbb{Z}\mu_n
 \end{align} is dense in $\mathbb{C}^m$. Take any $\mu_1,\dots,\mu_{2m}$ generating $\mathbb{C}^m$
 as a real linear space. It is well known that, if we set
 $$\mu_{m+1}\colon =x_1\mu_1+\dots +x_{2m}\mu_{2m}$$
for real numbers $x_1,\dots, x_{2m}$ such that  $1,x_1,\dots, x_{2m}$ are linearly independent over $\mathbb{Z}$,  then

 $$\mathbb{Z}\mu_1+\dots +\mathbb{Z}\mu_{2m+1} $$ is dense in $\mathbb{C}^m$. This proves the lemma. \qed

\begin{lem}\label{zar1.1} 
Let $m\in\mathbb{N}$, let $\Omega\subset\mathbb{C}$ be an open set,  and let $$f_1,\dots,f_m\colon\Omega\to\mathbb{C}$$ be  holomorphic functions. Suppose that the set of functions $\{f_1,\dots,f_m\}$ have no linear relation.  
Then the function  
\begin{equation*}(x_1,\dots,x_m)\mapsto
 \det\begin{bmatrix} f_1(x_1)& f_1(x_2)&\dots &f_1(x_m)\\
  f_2(x_1)& f_2(x_2)&\dots &f_2(x_m)\\
  \vdots& \vdots && \vdots \\
   f_m(x_1)& f_m(x_2)&\dots &f_m(x_m)\\
 
 \end{bmatrix}.
 \end{equation*}
 does not vanish identically.
\end{lem}
\proof
Set $$f\colon =(f_1,\dots,f_m)\colon \Omega\to \mathbb{C}^m.$$ 
It is enough to show that there exist complex numbers $x_1,\dots,x_{m}\in\Omega$ such
 that the vectors $f(x_1),\dots,f(x_{m})$ are linearly independent. If it is not the case, the set
 $$W\colon=\spn\{f(x)\colon x\in\Omega\}$$ is a proper subspace of 
 $\mathbb{C}^{m}$ and we will have a linear relation for $f_1, \dots, f_m$. \qed

 The following Lemma give us a particular set of function satisfying the hypotheses of Lemma \ref{zar1.1}.

\begin{lem}\label{zar1.2} 
If $k,m\in\mathbb{N}$ and $b_1,b_2\dots,b_m\in\mathbb{C}$ are pairwise distinct,    the set  of $km+1$ functions 
\begin{align*}&\hspace{5mm}\Bigg\{1,  \frac{1}{(x-b_1)},\frac{1}{(x-b_1)^2},\dots,\frac{1}{(x-b_1)^k}, \\
&\hspace{3cm}\frac{1}{(x-b_2)}, \frac{1}{(x-b_2)^2 },\dots,\frac{1}{(x-b_2)^k},\\
&\hspace{7.5cm}\vdots \\
&\hspace{5.9cm} \frac{1}{(x-b_m)},\frac{1}{(x-b_m)^2},\dots,\frac{1}{(x-b_m)^k}\Bigg\}
\end{align*} have no linear relation.
\end{lem}
\proof
Suppose that there exist complex numbers $$c, c_{ij}, \quad i=1,\dots, m, j=1,\dots, k$$ 
such that the function  
 $$f(x)=c+\sum\limits_{i,j} \frac{c_{ij}}{(x-b_i)^j}$$ 
 vanishes. If $c_{ij}$ were nonzero for some $i,j$, the function $f$ would have a pole at $b_i$, which is not possible. Then the $c_{ij}$ are all
 zero and therefore  $c$ is also zero.\qed

Let
\begin{equation}\label{taus}
T=\{b_1,b_2,\dots,b_m\},
\end{equation} where $m\in\mathbb{N}$ and the complex numbers  $b_1,\dots,b_m$ are pairwise distinct. 
By Lemma \ref{zar2}, for some $n\in\mathbb{N}$ we can find numbers 
\begin{equation}\label{multiplicative}
\mu_{ij}\in\mathbb{C}, \quad i=1,\dots, M-1, \quad j=1,\dots,n 
\end{equation} with the following properties.
\begin{itemize}
\item For each $j=1,\dots,n$,  we have \begin{align}\begin{aligned} \label{multiplicative0}{\im(\mu_{1j})},
\im (\mu_{1j}+\mu_{2j}),\dots, \im (\mu_{1j}+\mu_{(M-1)j})>0.\end{aligned}\end{align}
\item The vectors
\begin{equation}\label{multiplicative1} 
\mu_j=\left(e^{2\pi \sqrt{-1}\mu_{1j}}, e^{2\pi \sqrt{-1}\mu_{2j}},\dots,e^{2\pi \sqrt{-1}\mu_{(M-1)j}} \right), \quad j=1,\dots,n
\end{equation} generate  a  dense subgroup  in the multiplicative group $(\mathbb{C}^*)^{M-1}$.
\end{itemize}

Consider the natural number $N=(k+1)m+n+1$ and take complex numbers $b_{m+1},b_{m+2},\dots,b_{N}$, such that $b_1,b_2\dots,b_N$ are pairwise distinct. The choice of $b_{m+1},b_{m+2},\dots,b_{N}$ satisfies some generic property  that will be specified later. As mentioned in \eqref{con1} of Proposition \ref{log}, for a suitable choice of  numbers \begin{center} $\nu_{ij}\in\mathbb{C}$, $i=1,\dots,(M-1)$, $j=1,\dots,N$, \end{center} the holomorphic foliation $\mathcal{G}$ will be defined by a rational vector field of the form

\begin{equation}\label{systemsol}
Y\colon=
\sum\limits_{i=1}^{M-1}A_i(t_M)t_i\frac{\partial}{\partial t_i}
+\frac{\partial}{\partial t_M},
\end{equation}
where \begin{equation}\label{systemsol1}
A_i(t_M)=
\sum\limits_{j=1}^N \frac{\nu_{ij}}{t_M-b_{j}}, \quad i=1,\dots M-1.
\end{equation}
This foliation has its singularities
at the points $$p_j\colon=(0,\dots,0,b_j)\in\mathbb{C}^M,\; j=1,\dots, N,$$ which are linearizable, as we explicitly show. 
Let $l\in\{1,\dots,N\}$ and consider coordinates 
\begin{equation}\label{coords}(s_1,s_2,\dots,s_M)=(t_1,t_2,\dots,t_{M-1},t_M-b_l)\end{equation} at the point $p_l$. In these coordinates the vector field $Y$ is explicitly given by

\begin{align*}
Y_l\colon=
\sum\limits_{i=1}^{M-1}\left(\sum\limits_{j=1}^N \frac{\nu_{ij}}{s_M-(b_{j}-b_l)}\right)s_i\frac{\partial}{\partial s_i}+\frac{\partial}{\partial s_M}.
\end{align*} Then, if we set  $$g_{i l}(s_M)\colon= \sum\limits_{{j=1,j\neq l}}^N \frac{\nu_{i j}}{s_M-(b_{j}-b_l)},\quad  i=1,\dots,(M-1),$$
 we have
\begin{align*}
Y_l=\sum\limits_{i=1}^{M-1}
\left(\frac{\nu_{il}}{s_M}+g_{il}(s_M)\right)s_i\frac{\partial}{\partial s_i}+\frac{\partial}{\partial s_M}.
\end{align*} Since the $g_{il}$ are holomorphic near the origin, there exist functions  $h_{i l}$ holomorphic on a neighborhood of the origin,  such that $h_{i l}'=g_{i l}$
and $h_{i l} (0)=0$ for $i=1,\dots,M-1$.  
 Define \begin{equation}\label{defih}h_l(s_1,\dots,s_M)=(s_1e^{h_{1l}(s_M)},\dots, s_{M-1}e^{h_{(M-1)l}(s_M)}, s_M).\end{equation} A straightforward computation shows that 
\begin{equation}\label{linealizacion}h_l^{*}(Y_l)=\frac{1}{s_M}\left(\sum\limits_{i=1}^{M-1}\nu_{il}s_i\frac{\partial}{\partial s_i}+s_M\frac{\partial}{\partial s_M}\right).\end{equation}
That is, the foliation  at $p_l$ is linearizable. Moreover, from the definition of $h_l$, we see that a sufficient condition for $h_l$ to be
 $k$-tangent to the identity is that the functions $g_{i l}$ have orders at least  $k-1$ at the origin. 
The function $g_{i l}$ can be expanded in the form
 \begin{align*}
 g_{i l}(s_M)&=-\sum\limits_{{j=1,j\neq l}}^N \frac{\nu_{i j}}{b_j-b_l}\cdot\frac{1}{1-\frac{s_M}{b_j-b_l}} \\
 &=-\sum\limits_{{j=1,j\neq l}}^N \frac{\nu_{i j}}{b_j-b_l}\left(1+\frac{s_M}{b_j-b_l} +\left(\frac{s_M}{b_j-b_l}\right)^2+\cdots\right),
 \end{align*}
 whence the coefficients up to order $k-1$ of $g_{i l}$ are given by 
 \begin{equation}\label{ces}
 \begin{aligned}
 c_{i l0}&=& -\sum\limits_{{j=1,j\neq l}}^N \frac{\nu_{i j}}{b_j-b_l}\\
 c_{i l1}&=& -\sum\limits_{{j=1,j\neq l}}^N \frac{\nu_{i j}}{(b_j-b_l)^2}\\
 &\hspace{2mm}\vdots&\\
c_{i l(k-1)}&=& -\sum\limits_{{j=1,j\neq l}}^N \frac{\nu_{i j}}{(b_j-b_l)^k}.
 \end{aligned}
 \end{equation}
Therefore, a sufficient condition for $h_l$ to be $k$-tangent to the identity for $l=1,\dots,m$  is to have the equalities 
\begin{align*} c_{il\alpha}=0, \quad \textrm{ for }\;\;i=1,\dots, M-1,\;  l=1,\dots,m, \; \alpha=0,\dots,k-1.
\end{align*} For each 
$i=1,\dots, {M-1}$, it will be convenient to define $F_i \colon \mathbb{C}^N \to \mathbb{C}^{km}$,
\begin{align*} 
F_i &= (c_{i10},c_{i11},\dots,c_{i 1(k-1)},c_{i 20},
\dots,c_{i 2(k-1)}, \dots ,
c_{i m0},\dots,c_{i m(k-1)}),
\end{align*} as a function of the variable $\nu_i\colon =(\nu_{i 1},\dots,\nu_{i N})$. The function $F_i$ is linear and --- in view of \eqref{ces} --- it is represented by the $km\times N$ matrix
 \begin{equation}\label{matriz}
   -\begin{bmatrix}0& \frac{1}{b_2-b_1}&\frac{1}{b_3-b_1}&\cdots & \frac{1}{b_m-b_1}&\cdots & \frac{1}{b_N-b_1}\\
0& \frac{1}{(b_2-b_1)^2}&\frac{1}{(b_3-b_1)^2}&\cdots & \frac{1}{(b_m-b_1)^2}&\cdots & -\frac{1}{(b_N-b_1)^2}\\
\vdots &\vdots & \vdots& &\vdots &&\vdots &\\
0& \frac{1}{(b_2-b_1)^k}&\frac{1}{(b_3-b_1)^k}&\cdots & \frac{1}{(b_m-b_1)^k}&\cdots & \frac{1}{(b_N-b_1)^k}\\
 \frac{1}{b_1-b_2}&0&\frac{1}{b_3-b_2}&\cdots & \frac{1}{b_m-b_2}&\cdots& \frac{1}{b_N-b_2}\\
 \frac{1}{(b_1-b_2)^2}&0&\frac{1}{(b_3-b_2)^2}&\cdots & \frac{1}{(b_m-b_2)^2}&\cdots & \frac{1}{(b_N-b_2)^2}\\
\vdots &\vdots & \vdots& &\vdots &&\vdots\\
\frac{1}{(b_1-b_2)^k}&0&\frac{1}{(b_3-b_2)^k}&\cdots & \frac{1}{(b_m-b_2)^k}&\cdots& \frac{1}{(b_N-b_2)^k}\\
&&\\
\vdots &\vdots & \vdots& &\vdots&&\vdots\\
&&\\
 \frac{1}{b_1-b_m}&\frac{1}{b_2-b_m}&\frac{1}{b_3-b_m}&\cdots&0 &\cdots & \frac{1}{b_N-b_m}\\
 \frac{1}{(b_1-b_m)^2}&\frac{1}{(b_2-b_m)^2}&\frac{1}{(b_3-b_m)^2} & \cdots&0&\cdots&\frac{1}{(b_N-b_m)^2}\\
\vdots &\vdots & \vdots& &\vdots&&\vdots\\
\frac{1}{(b_1-b_m)^k}&\frac{1}{(b_2-b_m)^k}&\frac{1}{(b_3-b_m)^k} & \cdots&0&\cdots&\frac{1}{(b_N-b_m)^k}\\

\end{bmatrix}.
\end{equation}
Since this matrix does not depend on $i=1,\dots, M-1$, we will denote $F_i=F$ for all $i=1,\dots, M-1$. 
We consider any  $km\times km$ submatrix of the matrix in \eqref{matriz} disjoint of the first $m$ columns: we take for example 
the submatrix 
 
  \begin{equation}
J=\begin{bmatrix} \frac{1}{b_{m+1}-b_1}&\frac{1}{b_{m+2}-b_1}&\cdots & \frac{1}{b_{m+mk}-b_1}\\
\frac{1}{(b_{m+1}-b_1)^2}&\frac{1}{(b_{m+2}-b_1)^2}&\cdots & \frac{1}{(b_{m+mk}-b_1)^2}\\
\vdots &\vdots & &\vdots\\
\frac{1}{(b_{m+1}-b_1)^k}&\frac{1}{(b_{m+2}-b_1)^k}&\cdots & \frac{1}{(b_{m+mk}-b_1)^k}\\

 \frac{1}{b_{m+1}-b_2}&\frac{1}{b_{m+2}-b_2}&\cdots & \frac{1}{b_{m+mk}-b_2}\\
 \frac{1}{(b_{m+1}-b_2)^2}&\frac{1}{(b_{m+2}-b_2)^2}&\cdots & \frac{1}{(b_{m+mk}-b_2)^2}\\
\vdots &\vdots & &\vdots\\
\frac{1}{(b_{m+1}-b_2)^k}&\frac{1}{(b_{m+2}-b_2)^k}&\cdots & \frac{1}{(b_{m+mk}-b_2)^k}\\

&&\\
\vdots &\vdots &&\vdots\\
&&\\

 \frac{1}{b_{m+1}-b_m}&\frac{1}{b_{m+2}-b_m} &\cdots & \frac{1}{b_{m+mk}-b_m}\\
 \frac{1}{(b_{m+1}-b_m)^2}&\frac{1}{(b_{m+2}-b_m)^2} & \cdots&\frac{1}{(b_{m+mk}-b_m)^2}\\
\vdots &\vdots && \vdots\\
\frac{1}{(b_{m+1}-b_m)^k}&\frac{1}{(b_{m+2}-b_m)^k}&\cdots&\frac{1}{(b_{m+mk}-b_m)^k}\\

\end{bmatrix}.
\end{equation}
Recall that $b_1,\dots,b_m$ are previously fixed pairwise distinct constants --- see \eqref{taus}. As we have said above, 
  the constants
$b_{m+1},b_{m+2},\dots,b_{N}$ have to be chosen in a particular generic way; we explain now this choice.  By Lemmas \ref{zar1.1} 
and \ref{zar1.2}
we have that,   as a function of $(b_{m+1},b_{m+2},\dots,b_{N})$, the determinant $\det (J)$ does not vanish.
In the same way, this property holds for any $km\times km$ submatrix of \eqref{matriz} disjoint of the first $m$ columns. Thus,
we conclude that the set of $$(b_{m+1},b_{m+2},\dots,b_{N})\in\mathbb{C}^{N-m}$$ such that all the already mentioned $km\times  km$ submatrices of \eqref{matriz} have nonzero determinant is a Zariski open set. Denote by $\mathcal{U}$ this Zariski open set. On the other hand,
we use the last $km+1$ columns of \eqref{matriz} to construct the following $(km+1)\times(km+1)$ matrix:

  \begin{equation}\label{matriz1}
J_1=\begin{bmatrix} 1&1&\cdots&1\\
 \frac{1}{b_{N-km}-b_1}&\frac{1}{b_{N-km+1}-b_1}&\cdots & \frac{1}{b_{N}-b_1}\\
\vdots &\vdots & &\vdots\\
\frac{1}{(b_{N-km}-b_1)^k}&\frac{1}{(b_{N-km+1}-b_1)^k}&\cdots & \frac{1}{(b_{N}-b_1)^k}\\

 \frac{1}{b_{N-km}-b_2}&\frac{1}{b_{N-km+1}-b_2}&\cdots & \frac{1}{b_{N}-b_2}\\

\vdots &\vdots & &\vdots\\
\frac{1}{(b_{N-km}-b_2)^k}&\frac{1}{(b_{N-km+1}-b_2)^k}&\cdots & \frac{1}{(b_{N}-b_2)^k}\\

&&\\
\vdots &\vdots &&\vdots\\
&&\\

 \frac{1}{b_{N-km}-b_m}&\frac{1}{b_{N-km+1}-b_m}&\cdots & \frac{1}{b_{N}-b_m}\\

\vdots &\vdots & &\vdots\\
\frac{1}{(b_{N-km}-b_m)^k}&\frac{1}{(b_{N-km+1}-b_m)^k}&\cdots & \frac{1}{(b_{N}-b_m)^k}\\

\end{bmatrix}.
\end{equation}
 Again by Lemmas \ref{zar1.1} 
and \ref{zar1.2}
we have that,   as a function of $(b_{m+1},b_{m+2},\dots,b_{N}),$ the determinant $\det (J_1)$ does not vanish. Then the
set of $$(b_{m+1},b_{m+2},\dots,b_{N})\in\mathbb{C}^{N-m}$$ such that $\det (J_1)\neq 0$ is a Zariski open set $\mathcal{U}_1$. Finally, from now on
the vector 
$(b_{m+1},b_{m+2},\dots,b_{N})$ will be supposed fixed in the Zariski open set  $\mathcal{U}\cap\mathcal{U}_1$. So that
\begin{itemize}
\item  any  $km\times km$ submatrix of \eqref{matriz} disjoint  of the first $m$ columns has nonzero determinant, and
\item the matrix $J_1$ has nonzero determinant.
 \end{itemize}

\begin{lem} \label{nuca}For each $i=1,\dots,M-1$, there exists  $$\nu_i=(\nu_{i1},\dots,\nu_{iN})\in \mathbb{C}^N$$ such that the following properties hold.
\begin{enumerate}
\item\label{nu1} $\nu_{ij}\neq 0$, for all $i=1,\dots, M-1$, $j=1,\dots,N$.
\item\label{nu2}$F(\nu_1)=\dots=F(\nu_{M-1})=0$.
\item\label{valores} $(\nu_{i1},\dots,\nu_{i(m+n)})=\left(\frac{\lambda_i (b_1)}{\lambda_M(b_1)},\dots,\frac{\lambda_i (b_m)}{\lambda_M(b_m)},
\mu_{i1}, \dots, \mu_{in}\right)$,   where the $\mu_{ij}$ are as  defined in
\eqref{multiplicative}. 
\item\label{nu4}
 For each $j=(m+n+1),\dots, N$, there exists a real line through the origin in $\mathbb{C}$ separating
$1\in\mathbb{C}$ from the set $$\{\nu_{1j},\nu_{1j}+\nu_{2j},\dots, \nu_{1j}+\nu_{(M-1)j}\}.$$
\item\label{nu5} The numbers  \begin{align*}
\tilde{\nu}_1\colon &=-1-\sum\limits_{j=1}^N \nu_{1j},\\
\tilde{\nu}_2\colon &=1-\sum\limits_{j=1}^N \nu_{2j},\\
&\hspace{2.3mm}\vdots\\
\tilde{\nu}_{M-1}\colon &=1-\sum\limits_{j=1}^N \nu_{(M-1)j},\\
\end{align*} are all nonzero and 
there exists a real line through the origin in $\mathbb{C}$ separating
$1\in\mathbb{C}$ from the set $$\{\tilde{\nu}_{1},\tilde{\nu}_{1}+\tilde{\nu}_{2},\dots, \tilde{\nu}_{1}+\tilde{\nu}_{(M-1)}\}.$$

\end{enumerate}
\end{lem} 
\begin{rem} \label{remarknu}By \eqref{valores} of Lemma \ref{nuca}  we have $$\nu_{i(m+1)}=\mu_{i1},\; \nu_{i(m+2)}
=\mu_{i2}, \dots,\; \nu_{i(m+n)}=\mu_{in}.$$ Then, it follows
from \eqref{multiplicative0} that, for each $j=m+1,\dots,m+n$,  there exists a real line through the origin in $\mathbb{C}$ separating
$1\in\mathbb{C}$ from the set $$\{\nu_{1j},\nu_{1j}+\nu_{2j},\dots, \nu_{1j}+\nu_{(M-1)j}\}.$$ Therefore, 
\eqref{nu4} of Lemma \ref{nuca}  will actually holds true whenever $j \notin \{1, \dots, m\}$. 
\end{rem} 
\proof Since $N=(m+n) +(km+1)$,  it will be convenient to consider $F$ as a function of 
$$(u,v)\in \mathbb{C}^{m+n}\times \mathbb{C}^{km+1}.$$ 
For each $i=1,\dots,M-1$,
we start by finding vectors  satisfying
the properties \eqref{nu2} and \eqref{valores} of the lemma. Thus we are looking for solutions of the equation
$F(u,v)=0$. 
The  property established above about the $km\times  km$ submatrices of \eqref{matriz} 
implies that, in the equation \mbox{$F(u,v)=0,$} any set of $km$ variables chosen between the $km+1$ coordinates of $v$ can be expressed as linear functions of the $m+n+1$ remaining variables; so  these  remaining variables can be arbitrarily chosen.
 Since the  $m+n+1$ remaining variables includes the $m+n$ coordinates of $u$, we can find solutions of $F(u,v)=0$ with any prescribed value of $u$. Then there exist $$(\mathfrak{u}_1,\mathfrak{v}_1),\dots, (\mathfrak{u}_{M-1},\mathfrak{v}_{m-1})\in\mathbb{C}^{m+n}\times \mathbb{C}^{N-m-n}$$
 such that
 \begin{align}
  \label{vida2}&\hspace{0.0cm}F(\mathfrak{u}_1,\mathfrak{v}_1)=\dots =F(\mathfrak{u}_{M-1},\mathfrak{v}_{m-1})=0,\quad \textrm{and}\\
 \label{vida1} &\mathfrak{u}_i=\left(\frac{\lambda_i (b_1)}{\lambda_M(b_1)},\dots,\frac{\lambda_i (b_m)}{\lambda_M(b_m)},
\mu_{i1}, \dots, \mu_{in}\right),\quad i=1,\dots, M-1.
\end{align}
Let $G\colon \mathbb{C}^{km+1}\to \mathbb{C}^{km}$ be the linear function defined by $$G(v)=F(0,v).$$ 
We write  $$v=(v_1,\dots,v_{km+1})\in \mathbb{C}^{km+1},$$ 
$$\im v=(\im v_1,\dots,\im v_{km+1})\in\mathbb{R}^{km+1}$$ 
and consider the linear function $$\mathscr{S}(v)=v_1+v_2+\dots+v_{km+1}.$$
We will find $\mathfrak{v}\in \mathbb{C}^{km+1}$ such that
 \begin{align}\label{mata}
 G(\mathfrak{v})=0, \quad \im\mathscr{S}(\mathfrak{v})\neq 0,\quad
\im (\mathfrak{v})\in (\mathbb{R}^*)^{km+1}.
\end{align}
Since the matrix of $G$ is composed of the last $km+1$ columns of the matrix \eqref{matriz},
their $km\times  km$ submatrices have nonzero determinant. Then, in the equation $G(v)=0$, any set of $km$ 
coordinates between the $km+1$ coordinates of $v$  can be expressed as linear functions of the  remaining coordinate; so  this  remaining coordinate
of $v$ can be arbitrarily chosen. In particular, given any $j=1,\dots, (km+1)$, we can find a solution of $G(v)=0$ with
$\im v_j\neq 0$; that is, \begin{align}\label{kindle1}G^{-1}(0)\not\subset \{\im v_j =0\},\; j=1,\dots, (km+1).\end{align}
On the other hand, from \eqref{matriz1} we can see that the matrix of $G$ is given by the matrix $J_1$ with its first row deleted. Thus, the fact of $J_1$
being non-singular implies that $\ker (G) \not\subset \ker \mathscr{S}.$ Therefore, since $\ker (G)$ is a complex linear space, and since the maximal complex linear space in the real space $\{\im\mathscr{S}=0\}$ is  $\ker \mathscr{S}$, we deduce that $\ker(G)$ can not be contained in
$\{\im\mathscr{S}=0\}$, that is,
$$ G^{-1}(0)\not\subset \{\im\mathscr{S}=0\}.$$
 It follows from this and \eqref{kindle1} ---
 since $G^{-1}(0)$ is irreducible --- that
   $$G^{-1}(0)\not\subset \{\im\mathscr{S}=0\}\cup\bigcup\limits_{j=1}^{km+1} \{\im v_j =0\},$$ 
   and therefore we can choose $\mathfrak{v}\in \mathbb{C}^{km+1}$ satisfying \eqref{mata}.
Since
 \begin{align}\im (\mathfrak{v}_i+s\mathfrak{v})=\im \mathfrak{v}_i+s\im \mathfrak{v},\quad s\in\mathbb{R}
\end{align} and no  component of $\im \mathfrak{v}$ is zero, we can fix $s>0$ so big that 
\begin{itemize}
\item[($\star$)] for each $i=1,\dots,M-1$ and each $j=1,\dots, (km+1)$, the $j$th component of $\im (\mathfrak{v}_i+s\mathfrak{v})\in \mathbb{R}^{km+1}$ have the same sign as the $j^{\textrm{th}}$ component of $\im \mathfrak{v}$.
\end{itemize}

Now, for each $i=1,\dots, M-1$, define $$\nu_i=(\mathfrak{u}_i, \mathfrak{v}_i+s\mathfrak{v}).$$
Clearly, from \eqref{vida1} the property \eqref{valores} of the lemma holds. Property  \eqref{valores} of the lemma together with
($\star$)  imply the property \eqref{nu1} of the lemma. 
From \eqref{vida2} and \eqref{mata} we have
\begin{align}F(\mathfrak{u}_i,\mathfrak{v}_i+s\mathfrak{v})=F(\mathfrak{u}_i,\mathfrak{v}_i)+sG(\mathfrak{v})=0,
\end{align} so the property \eqref{nu2} of Lemma \ref{nuca} also holds. 
Now, let $j\in\{m+n+1,\dots,N\}$. From ($\star$) we deduce that the numbers 
$\im \nu_{1j}, \im \nu_{2j} \dots, \im \nu_{(M-1)j}$  have the same sign. 
This means that the numbers $$\nu_{1j},\nu_{2j}, \dots,\nu_{(M-1)j}$$ are simultaneously contained  in the upper half plane or in the lower half plane of $\mathbb{C}$. Then the same holds for the numbers 
$$\nu_{1j},\nu_{1j}+\nu_{2j},\dots,\nu_{1j}+\nu_{(M-1)j},$$ which in turns implies the  property \eqref{nu4} of the lemma. 
In the same way, the property \eqref{nu5} of the lemma will be guaranteed  if the numbers 
$$\im \tilde{\nu}_1, \im \tilde{\nu}_2 \dots, \im \tilde{\nu}_{M-1}$$ have the same sign. 
Given $i=1,\dots, M-1$,  we can express
$$\im \tilde{\nu}_i=-\im \sum\limits_{j=1}^N \nu_{ij}= c_i-s\im\mathscr{S}(\mathfrak{v}),$$
for some constant $c_i\in\mathbb{R}$. Therefore, by taking $s>0$ bigger if necessary, 
$\im \tilde{\nu}_i$ has the same sign as $-\im\mathscr{S}(\mathfrak{v})$ for all $i=1,\dots,M-1$, which finishes the proof.
\qed\\

From now on we fix  the vectors $\nu_1,\dots,\nu_{M-1}\in\mathbb{C}^N$ according to \mbox{Lemma \ref{nuca}.}
 Then the foliation $\mathcal{G}$ is already defined as stated by \eqref{con1} of \mbox{Proposition \ref{log}.}  Let us start
 the proof of the remaining assertions of  \mbox{Proposition \ref{log}.}
 
  Let $\tau\in T$. Then $\tau=b_{{l_0}}$ for some ${{l_0}}$ in $\{1,\dots, m\}$. By the choice of the $\nu_i$ we have  
 $$F (\nu_i)=0, \quad i=1,\dots,M-1,$$
that is
\begin{center}
$c_{il\alpha}=0$, \hspace{2mm}
for  $i=1,\dots,M-1$,  $\;l=1,\dots,m$, $\;\alpha=0,\dots,k-1$. 
\end{center}As we have seen above, these equalities guarantee
that $h_l$ is $k$-tangent to the identity for $l=1,\dots,m$. Thus, if we define  $$h_\tau\colon=h_{{l_0}},$$
 the assertion \eqref{con2a} of Proposition \ref{log} holds.  On the other hand, 
it follows from \eqref{linealizacion} and \eqref{valores} of Lemma \ref{nuca} that

\begin{align*}h_\tau^{*}(Y_{l_0})&=\frac{1}{s_M}\left(\sum\limits_{i=1}^{M-1}\nu_{il_0}s_i\frac{\partial}{\partial s_i}+s_M\frac{\partial}{\partial s_M}\right)\\
&=\frac{1}{s_M}\left(\sum\limits_{i=1}^{M-1}\frac{\lambda_i(\tau)}{\lambda_M(\tau)}s_i\frac{\partial}{\partial s_i}+s_M\frac{\partial}{\partial s_M}\right)\\
&=\frac{1}{\lambda_M(\tau)s_M}\left(\sum\limits_{i=1}^{M}{\lambda_i(\tau)}s_i\frac{\partial}{\partial s_i}\right),
\end{align*} 
which means that the foliation defined by $Y_{l_0}$ at $0\in\mathbb{C}^M$   is the pushforward  of the foliation generated by the linear system
\begin{equation}\label{linearsystemG3}\nonumber
\begin{aligned}
 s_1'&=&\lambda_1(\tau)s_1\\
 s_2'&=&\lambda_2(\tau)s_2\\
 &\hspace{0.2cm}\vdots&\\
s_M'&=&\lambda_M(\tau)s_M.
\end{aligned}
\end{equation} Thus, since $Y_{l_0}$ is nothing but the vector field $Y$ after the change of coordinates
$$(t_1,\dots,t_M)= (s_1,\dots,s_M)+(0,\dots,0,b_{l_0}),$$ the   assertion \eqref{con2c} of Proposition \ref{log} follows.

The assertion \eqref{con2.1} of Proposition \ref{log} is easily verified, so it needs no additional details. Thus we proceed with 
 the  proof of the remaining assertions. 
 
  It is easy to see that the line 
$$\mathbf{C}\colon=\{t\in\mathbb{C}^M\colon t_1=0,\dots,t_{M-1}=0\}$$ is invariant by $\mathcal{G}$, so 
$\mathbf{C}^*=\mathbf{C}\backslash\{p_1,\dots,p_N\}$ is a leaf of $\mathcal{G}$. We are interested in the holonomy
of this leaf, which --- due to the special form of the vector field $Y$ --- has a global nature and can be easily computed, as we show next.  
 Consider the hyperplane
 $$\Sigma =\{t\in\mathbb{C}^M\colon t_M=a\},$$  where $$a\in \mathbb{C}\backslash\{b_1,\dots,b_N\}.$$  In
 view of the natural identification  $\mathbf{C}\simeq\mathbb{C}$, a curve in
 $\mathbf{C}^*$ can be thought of as curve in $\mathbb{C}\backslash \{b_1,\dots,b_N\}$. Thus,
 we consider a  smooth curve  \begin{equation}\label{gam}\gamma\colon [0,1]\to \mathbb{C}\backslash\{b_1,\dots,b_N\},\quad \gamma(0)=a.\end{equation}
 Then, given 
 $z=(z_1,\dots,z_{M-1},a)\in\Sigma$, there exists a unique curve $$\gamma^z\colon [0,1]\to \mathbb{C}^M,\quad \gamma^z(0)=z$$ with last coordinate given by $\gamma$ and that is tangent to $\mathcal{G}$: in view of  \eqref{systemsol}, this curve is explicitly given by 
 \begin{equation}\label{solu1}\gamma^z (s)= \left(z_1\exp{\int\limits_{\gamma |_{[0,s]}}A_1(u)du },
  \dots, z_{M-1}\exp{\int\limits_{\gamma |_{[0,s]}}A_{M-1}(u)du },\gamma(s) \right). \end{equation}
 Thus, the holonomy associated to $\gamma$ is the map
  \begin{align}\label{solu2}
  \mathcal{H}_\gamma\colon z\mapsto \gamma^z(1),
  \end{align} which is a linear isomorphism between $\Sigma=\{t_M=a\}$ 
  and $\{t_M=\gamma(1)\}$. 
  This map $\mathcal{H}_\gamma$ can be computed if $\gamma$ is a closed curve, because it depends only on the integrals 
$$\int\limits_{\gamma}A_i(u)du, \quad i=1,\dots, M-1.$$ 
For each $j=1,\dots, N$, choose a smooth positively oriented  closed simple curve 
  \begin{align}\label{curvon}\gamma_j\colon [0,1]\to \mathbb{C}\backslash \{b_1,\dots,b_N\},\quad \gamma(0)=\gamma(1)=a,
  \end{align}
  whose interior domain intersects $ \{b_1,\dots,b_N\}$ exactly at $b_j$.  
  Then $$\int\limits_{\gamma_j}A_i(u)du=\int\limits_{\gamma_j} \sum\limits_{\mathfrak{l}=1}^N \frac{\nu_{i\mathfrak{l}}}{w-b_{l}}dw
  =\int\limits_{\gamma_j}  \frac{\nu_{ij}}{w-b_{j}}=2\pi \sqrt{-1}\nu_{ij},$$ 
  so that 
  \begin{align}\label{holodef} \mathcal{H}_{\gamma_j}(z)=\left(z_1e^{2\pi \sqrt{-1}\nu_{1j}},\dots, z_{M-1}e^{2\pi \sqrt{-1}\nu_{(M-1)j}}, a \right),\end{align}
  which can be  identified with the linear map 
  \begin{align}z\mapsto\left(e^{2\pi \sqrt{-1}\nu_{1j}}z_1,\dots, z_{M-1}e^{2\pi \sqrt{-1}\nu_{(M-1)j}} \right) \end{align} or even with its corresponding diagonal matrix.
  Then the holonomy group of $\mathbf{C}^*$  is generated by the diagonal matrices
  $\mathcal{H}_{\gamma_1},\dots, \mathcal{H}_{\gamma_N}$ in $\gl (M-1,\mathbb{C})$ and so,  in particular,  this group is  abelian. \\
 Before  going ahead  with the rest of the proof, we  establish  some notation and conventions. 
 \begin{defn}
 We know that the blow-up
  $\pi\colon\widehat{\mathbb{C}^M}\to\mathbb{C}^M$ is a 
 biholomorphism between  $\widehat{\mathbb{C}^M}\backslash\pi^{-1}(0)$  and
  $\mathbb{C}^M\backslash\{0\}$. Thus, if no confusion arise we can identify objects in 
  $\widehat{\mathbb{C}^M}\backslash\pi^{-1}(0)$
  with their corresponding images in  $\mathbb{C}^M\backslash\{0\}$. Then,  
   $w \in \widehat{\mathbb{C}^M}\backslash\pi^{-1}(0)$ and  $W\subset \widehat{\mathbb{C}^M}\backslash\pi^{-1}(0)$
   are identified with $\pi(w)$ and $\pi (W)$. The foliation $\mathcal{G}$ on $\widehat{\mathbb{C}^M}\backslash\pi^{-1}(0)$
     is identified with the foliation $\mathcal{F}$ on $\mathbb{C}^M\backslash\{0\}$. 
     Given a point $w \in \widehat{\mathbb{C}^M}\backslash\pi^{-1}(0)$, the leaf of $\mathcal{G}$ through $w$ is 
     identified with the leaf of $\mathcal{F}$ through $\pi(w)$, so this leaf is
  simply referred to as the leaf through $w$ with no reference to the foliation. A curve $\alpha\colon[0,1]\to\widehat{\mathbb{C}^M}\backslash\pi^{-1}(0)$ is identified with the curve 
   $\pi\circ\alpha$ in $\mathbb{C}^M\backslash\{0\}$;  furthermore, we define  the length of the curve $\alpha$ 
    as the Euclidean length of $\pi\circ\alpha$, and this number will be denoted by $\ell (\alpha)$. If the curve $\alpha$ is contained
    in a leaf, we will say that $\alpha$ is an integral curve. Finally, given $w \in \widehat{\mathbb{C}^M}\backslash\pi^{-1}(0)$, we denote by $\|w\|$ the Euclidean norm of $\pi(w)$, that is, $\|w\|\colon=|\pi(w)|$.
    \end{defn}

\begin{lem}\label{assertion1} If  $\gamma$ is a curve as in \eqref{gam}, there exists  ${K}_\gamma>0$ such that the following properties hold.
\begin{enumerate}
\item If $z\in\Sigma$, $z_1\neq 0$ and $\|z\|<1$, then  \begin{equation}\nonumber \ell (\gamma^z)\le K_\gamma\|z\|.
\end{equation}
\item 
The  number $K_\gamma$ depends continuously on $\gamma$ if we consider the $C^1$-topology in the space
  of curves $\gamma$.
  \end{enumerate}
 As a direct consequence of the first assertion above, there exists a constant  $K>0$ such that
\begin{equation}\label{lon1} z\in\Sigma\backslash\pi^{-1}(0) , \|z\|<1, \;j=1,\dots, N\implies \ell (\gamma_j^z)\le K\|z\|.
\end{equation}
  \end{lem}

\proof Write $\gamma^z(s)=(\gamma_1(s),\dots,\gamma_{M-1}(s),\gamma)$, $s\in [0,1]$. It follows from \eqref{solu1} that 
\begin{align}\label{que1}|\gamma_i(s)|\le K_1|z_i|, \quad i=1,\dots, M-1,
\end{align}where 
$$K_1\colon=\max\left\{\Bigg|\exp {\int\limits_{\gamma|_{[0,s]}}}A_i(u)du\Bigg|\colon s\in[0,1],i=1,\dots,M-1\right\}.$$
Also from \eqref{solu1} we have 
\begin{align}\nonumber|\gamma_i'(s)|=\Big|\gamma'(s)A_i(\gamma(s))\gamma_i(s)\Big|, \quad i=1,\dots, M-1,
\end{align}which together with \eqref{que1} leads to 
\begin{align}\label{que2}|\gamma_i'(s)|\le K_2K_1|z_i|, \quad i=1,\dots, M-1,
\end{align}where
$$K_2\colon=\max\Big\{\big|\gamma'(s)A_i(\gamma(s))\big|\colon s\in[0,1],\; i=1,\dots,M-1\Big\}.$$
Observe that 
\begin{align}
|(\pi\circ\gamma^z)'|&\le |\gamma_1'|+|(\gamma_1\gamma_2)'|+\dots +|(\gamma_1\gamma_{M-1})'|+
|(\gamma_1\gamma)')|.\label{que3}
\end{align}
Moreover, from \eqref{que1} and \eqref{que2}  we obtain
\begin{align}\nonumber |(\gamma_1\gamma_i)'|&\le |\gamma_1'||\gamma_i|+|\gamma_1||\gamma_i'| 
\\
&\le
(K_2K_1|z_1|)(K_1|z_i|)+(K_1|z_1|)(K_2K_1|z_i|),\nonumber 
\end{align} whence
\begin{align}
|(\gamma_1\gamma_i)'|& \le 2K_2K_1^2|z_1||z_i|, \quad i=1,\dots, M-1.\label{cual1}
\end{align}
On the other hand,
\begin{align}\nonumber|(\gamma_1\gamma)'|&\le |\gamma_1'||\gamma|+|\gamma_1||\gamma'|\le 
(K_2K_1|z_1|)|\gamma|+(K_1|z_1|)|\gamma'|,
\end{align}
so we have
\begin{align}
|(\gamma_1\gamma)'|&\le K_3|z_1|,\label{cual2}
\end{align}
where \begin{align}K_3\colon=\max\limits_{s\in[0,1]}\Big(K_2K_1|\gamma(s)|+K_1|\gamma'(s)|\Big).
\end{align}
By using \eqref{que2},  \eqref{cual1} and \eqref{cual2} in \eqref{que3}, we obtain 

\begin{align}\nonumber
|(\pi\circ\gamma^z)'|&\le K_2K_1|z_1|+2K_2K_1^2\sum\limits_{i=2}^{M-1}|z_1||z_i|+K_3|z_1|\\
&\le K_2K_1\|z\|+2K_2K_1^2\sum\limits_{i=2}^{M-1}\|z\|+K_3\|z\|\\
&\le \Big(K_2K_1+(M-2)2K_2K_1^2+K_3\Big)\|z\|,\nonumber \\
&\le K_\gamma \|z\|,
\end{align} where $$K_\gamma\colon=K_2K_1+(M-2)2K_2K_1^2+K_3. $$ 
Therefore,
\begin{align}\nonumber \ell (\gamma^z)=\int\limits_0^1|(\pi\circ\gamma^z)'|ds\le K_\gamma \|z\|.
\end{align}Finally, the last assertion of the lemma follows from the fact of $K_1$, $K_2$ and $K_3$ depending
continuously on $\gamma$. \qed

Between the curves  $\gamma_1,\dots,\gamma_N$  --- they are defined in \eqref{curvon} --- we distinguish the loops
$$\xi_1\colon =\gamma_{m+1},\; \xi_2\colon=\gamma_{m+2},\; \dots,\; \xi_n\colon =\gamma_{m+n}$$ and denote  their  holonomies  by
$$g_1\colon=\mathcal{H}_{\xi_1},\; g_2\colon=\mathcal{H}_{\xi_2},\;\dots, \; g_n\colon =\mathcal{H}_{\xi_n}.$$

\begin{lem}\label{densote} Let $G$ be the group generated by $g_1,\dots,g_n$ ---  this is  a subgroup
 of the holonomy group of the leaf $\mathbf{C}^*$.   
 \begin{enumerate}
 \item \label{esdensa} The orbit by $G$ of any $$z\in(\mathbb{C}^*)^{M-1}\times \{a\}\subset \Sigma$$  is dense in $\Sigma$.
 \item \label{contra1} There exists $\delta\in (0,1)$ such that 
\begin{equation}\nonumber
\|g_j(z)\|<\delta\|z\|,\quad z\in\Sigma\backslash\pi^{-1}(0),\; j=1,\dots, n.
 \end{equation}
 \end{enumerate}
\end{lem} 
\proof From \eqref{holodef} and  \eqref{valores} of Lemma \ref{nuca},
for each $j=1,\dots,n$ and $z\in\Sigma\backslash \pi^{-1}(0)$  we obtain
\begin{align} \label{kirma}g_j(z)=(\varsigma_1 z_1,\varsigma_2 z_2,\dots,\varsigma_{M-1}z_{M-1}, a),\end{align}
where 
$$\varsigma_1=e^{2\pi i\mu_{1j}},\;\varsigma_2 =e^{2\pi i\mu_{2j}},\dots, \varsigma_{M-1}=e^{2\pi i\mu_{(M-1)j}}.$$
So the first assertion of the lemma  follows directly  from \eqref{multiplicative1}.

On the other hand, from  \eqref{multiplicative0} we find $\tilde{\delta}>0$ such that 
$$\im (\mu_{1j}),\; \im (\mu_{1j}+\mu_{2j}), \dots,\; \im(\mu_{1j}+\mu_{(M-1)j})>\tilde{\delta}.$$
Then 
$$e^{2\pi i\im \mu_{1j}}, e^{2\pi i\im (\mu_{1j}+\mu_{2j})}, \dots,e^{2\pi i \im (\mu_{1j}+\mu_{(M-1)j})}
<e^{2\pi i\tilde{\delta}},$$
whence  
 $$|\varsigma_1|,|\varsigma_1\varsigma_2|,\dots,
|\varsigma_1\varsigma_{M-1}|<\delta\colon=e^{2\pi i\tilde{\delta}}.$$
Therefore, from \eqref{kirma} we have
\begin{align*}
\|g_j(z)\|=\sqrt{|\varsigma_1 z_1|^2+|\varsigma_1 z_1\varsigma_2 z_2|^2+\dots+|\varsigma_1 z_1\varsigma_{M-1} z_{M-1}|^2+|\varsigma_1 z_1 a|^2}\\
<\sqrt{\delta^2|z_1|^2+\delta^2| z_1 z_2|^2+\dots+\delta^2|z_1 z_{M-1}|^2+\delta^2|a|^2},
\end{align*} 
that is, $$\|g_j(z)\|<\delta\|z\|.$$ \qed

\noindent\emph{Proof of \eqref{con4} of Proposition \ref{log}.} 
It follows directly from the expression of the foliation $\mathcal{G}$ given in \eqref{con1} of Proposition 
\ref{log} that the hyperplanes $$\{t_2=0\},\dots,\{t_{M-1}=0\},\; \{ t_M=b_1\},\dots , \{t_M=b_N\}$$ are invariant by $\mathcal{G}$. Therefore
the hyperplanes   $$\{x_2=0\},\dots,\{x_{M-1}=0\}, \; \{x_M=b_1x_1\},\dots ,\{ x_M=b_Nx_1\}$$
are invariant by $\mathcal{F}$. The invariance of the hyperplane $\{x_1=0\}$ by $\mathcal{F}$ is equivalent to the invariance
by $\mathcal{G}$ of the hyperplane $\{t_M=\infty\}$ in the space $\pi^{-1}\big(\mathbb{C}^M\big)$. The local expression of the foliation
$\mathcal{G}$ near $\{t_M=\infty\}$ is given in the proof of Lemma \ref{pila4} at the end of the section. It will be clear from that expression
--- together with \eqref{nu5} of Lemma \ref{nuca} --- that $\{t_M=\infty\}$ is actually  invariant by $\mathcal{G}$, so we leave this verification to the reader --- which will prove \eqref{con4a} of Proposition \ref{log}. Given $\tau\in T$, the hypersurface $\pi^{-1}(S)$ near $p_\tau$ is given by the equation $$t_1\cdots t_{M-1}(t_M-\tau)=0.$$ 
On the other hand, it follows from the construction of $h_\tau$ --- see \eqref{defih} --- that the map $f_\tau$  respectively maps
the sets 
  $$\{s_1=0\}, \dots ,\{s_{M-1}=0\},\{s_M=0\}$$ into  the sets  $$\{t_1=0\}, \dots , \{t_{M-1}=0\}, \{t_M-\tau=0\}.$$
    This means that, in the local coordinates $(s_1,\dots, s_M)$, the hypersurface 
  $\pi^{-1}(S)$ is given by $\{s_1\cdots s_M=0\}$, which proves  \eqref{con4c} 
  of Proposition \ref{log}. 
  Recall that, given any $a'\in\mathbb{C}$ different
from $b_1,\dots, b_N$ and given any curve in $\mathbb{C}\backslash \{b_1,\dots,b_N\}$ connecting $a'$ with $a$, we obtain  an
associated
   holonomy map $$\mathfrak{h}\colon \{t_M=a'\}\to \{t_M=a\}, $$ which is a linear isomorphism.
  Thus, it follows from \eqref{esdensa} of Lemma \ref{densote} that the leaf through a point $w\in  \{t_M=a'\}$
  is everywhere dense provided 
  $$\mathfrak{h}(w)\in \Sigma^*=\{t_M=a\}\backslash \{t_1\dots t_{M-1}=0\}.$$ Thus,
  since the hyperplanes $\{t_1=0\}, \dots, \{t_{M-1}=0\}$ are invariant by $\mathcal{G}$, the leaf through $w$ is dense if
  $w$ is not contained in any of these hyperplanes. Therefore we conclude that any leaf outside the hyperplanes
  $$\{t_1=0\}, \dots, \{t_{M-1}=0\},\;\{t_M=b_1\}, \dots, \{t_{M}=b_N\}$$ is everywhere dense.,  so \eqref{con4b} 
  of Proposition \ref{log} is proved.\qed \\

 Let $\epsilon>0$ and let  $\Delta\subset\mathbb{C}^M$ be  a neighborhood of the origin --- as given in \eqref{con3} of Proposition \ref{log}.  Choose  $\varepsilon\in (0,\epsilon)$ such that 
 \begin{align}\label{gu1}x\in\mathbb{C}^M, \; |x|<3\varepsilon\implies x\in \Delta.
 \end{align}

\begin{lem}\label{assertion2}There exists $\rho\in(0,\varepsilon)$ with the following property:  if $z\in\Sigma\backslash\pi^{-1}(0)$,  $\|z\|<\rho$, $g\in G$ and $\|g(z)\|<\rho$, then there exists an integral curve  connecting $z$ with $g(z)$  of length smaller than  $\varepsilon/3 $.
\end{lem}

\proof Set $$\rho=\min\left\{\frac{(1-\delta)\varepsilon}{6Kn}, 1 \right\}$$
where $K$ is as given in Lemma \ref{assertion1}, $\delta$ as in Lemma \ref{densote}, and let $z\in\Sigma\backslash\pi^{-1}(0)$,  $||z||<\rho$.
 Suppose  that $g$ is of the form $g=g_i^{\circ l}$ for some $l\in\mathbb{N}$, $i\in\{1,\dots, n\}$. Without loss of generality we
 can assume $i=1$.
 Since $g_1=\mathcal{H}_{\xi_1}$, the curve $\xi_1^{z}$ connects $z$ with $g_1(z)$, the curve 
 $\xi_1^{g_1(z)}$ connects $g_1(z)$ with $g_1^{\circ 2}(z)$,  the curve $\xi_1^{g_1^{\circ 2}(z)}$ connects $g_1^{\circ 2}(z)$
  with $g_1^{\circ 3}(z)$, and so on.  Moreover, by \eqref{lon1} and \eqref{contra1} of Lemma \ref{densote} , we have 
 \begin{align}\label{lon2}\ell \Big(\xi_1^{g_1^{\circ j} (z)}\Big)\le K\big\|g_1^{\circ j} (z)\big\|
 \le K\delta^j\|z\|,\quad j\ge 0.
 \end{align} Then the integral curve 
 $$\alpha_1\colon=\xi_1^{z}* \xi_1^{g_1(z)}*\dots * \xi_1^{g_1^{\circ (l-1)}(z) },$$
 connects $z$ with $g_1^{\circ l}(z)$ and, from \eqref{lon2}, 
 \begin{align}\nonumber \ell (\alpha_1)&= \ell ( \xi_1^{z}) +\ell \left(\xi_1^{g_1(z)}\right) 
 +\dots + \ell \bigg(\xi_1^{g_1^{\circ {(l-1)}} (z)}\bigg)\\
 &\nonumber \le  K\|z\|+K\delta \|z\| +\dots +K\delta^{l-1} \|z\|, \\
  &\hspace{1.8cm}\ell (\alpha_1)\le \frac{K\|z\|}{1-\delta}.\nonumber
 \end{align}   Suppose now  that $$g=g_n^{\circ l_n}\circ g_{n-1}^{\circ l_{n-1}}\circ\dots\circ g_1^{\circ l_1},$$
where $l_1, \dots, l_n\in\mathbb{N}$. By the construction above, we find an integral curve  $\alpha_1$
connecting  $z$ with $g_1^{\circ l_1}(z)$,  such that 
\begin{gather*}\nonumber \ell (\alpha_1)\le\frac{K\|z\|}{1-\delta}<\frac{K\rho}{1-\delta},\\
\end{gather*}whence
\begin{gather*}\nonumber 
\ell (\alpha_1)<\frac{\varepsilon}{6n}.
\end{gather*}
Again by the construction above --- with $g_1^{\circ l_1}(z)$ instead of $z$ --- there is an integral curve  $\alpha_2$
connecting  $g_1^{\circ l_1}(z)$ with $ g_2^{\circ l_2} \left(g_1^{\circ l_1}(z)\right)$,   such that
 \begin{gather*}\ell (\alpha_2)< \frac{K||g_1^{\circ l_1}(z)||}{1-\delta}\le\frac{K||z||}{1-\delta}<\frac{K\rho}{1-\delta},
 \end{gather*} 
 whence
 \begin{gather*} \ell (\alpha_2)<\frac{\varepsilon}{6n}.
 \end{gather*} Then  $\alpha_1*\alpha_2$ is an integral curve that connects $z$ with $g_2^{\circ l_2}g_1^{\circ l_1}$ and
 satisfies 
 \begin{align*}\ell (\alpha_1*\alpha_2)=\ell (\alpha_1)+\ell (\alpha_2)<2\left(\frac{\varepsilon}{6n}\right).
 \end{align*} By iterating this argument, we find curves $\alpha_1, \dots, \alpha_n$, such that
 $$\overline{\alpha}\colon=\alpha_1* \dots * \alpha_n$$ is an integral curve,   connects $z$ with $g_n^{\circ l_n}\circ \dots \circ g_1^{\circ l_1} (z)=g(z)$ and satisfies
 \begin{gather*}\ell (\overline{\alpha})=\ell(\alpha_1)+\dots+\ell(\alpha_n)<n\left(\frac{\varepsilon}{6n}\right),\\
 \ell (\overline{\alpha})<\frac{\varepsilon}{6}.
 \end{gather*} If $g$ is of the type above, that is, if
 $$g=g_n^{\circ l_n}\circ g_{n-1}^{\circ l_{n-1}}\circ\dots\circ g_1^{\circ l_1},$$ where $l_1,\dots,l_n\in\mathbb{N}$, we will say that $g$ is contractive. Thus, we have proved that, if $g$ is contractive and $\|z\|<\rho$, then there exists an integral curve $\overline{\alpha}$ that connects $z$ with $g(z)$, such that 
$\ell (\overline{\alpha})<\frac{\varepsilon}{6}.$
 Now,  consider any $g\in G$ such that $||g(z)||<\rho$. Since $G$ is abelian, we can find $\mathfrak{g}_1,\mathfrak{g}_2\in G$ contractive such that 
 $g=\mathfrak{g}_2^{-1}\mathfrak{g}_1$. Since $\mathfrak{g}_1$ is contractive and $||z||<\rho$, there
  exists an integral curve $\overline{\alpha}_1$  connecting $z$ with $\mathfrak{g}_1(z)$, such that 
$\ell (\overline{\alpha}_1)<\frac{\varepsilon}{6}.
$
If we set $z'=g(z)$, since $\mathfrak{g}_2$ is contractive and $||z'||<\rho$, again we find an integral curve $\overline{\alpha}_2$ connecting $z'=g(z)$ with $\mathfrak{g}_2(z')=\mathfrak{g}_1(z)$, such that 
 $\ell (\overline{\alpha}_2)<\frac{\varepsilon}{6}$.
 Finally, the curve $$\alpha\colon= \overline{\alpha}_1*\overline{\alpha}_2^{-1}$$ is an integral curve connecting $z$ with $g(z)$,
 such that 
 $$\ell (\alpha)=\ell (\overline{\alpha}_1)+\ell (\overline{\alpha}_2)<\frac{\varepsilon}{3},$$ so Lemma 
 \ref{assertion2} is proved.\qed \\

 \noindent\emph{Proof of \eqref{con3a} of Proposition \ref{log}.}
  Set
 \begin{align*} \Sigma^*\colon =\{z\in \Sigma\colon  \|z\|<\rho, z_1\neq0,\dots,z_{M-1}\neq 0\},
 \end{align*}
 where $\rho$ is as given by Lemma \ref{assertion2}.  
Let ${\mathscr{U}}$ be the set of points that can be connected to a point in $ \Sigma^* $ by an integral curve of length smaller than $\varepsilon/3$. Since $ \Sigma^* $ is transverse to the foliation, it is easy to see that $\mathscr{U}$ is open.
Moreover, since $ \Sigma^* $ is disjoint of the exceptional divisor $\pi^{-1}(0)$,  so is $\mathscr{U}$. Then the set 
$\pi(\mathscr{U})$, which we identify with $\mathscr{U}$, will be the set mentioned in \eqref{con3} of \mbox{Proposition \ref{log}}.
 Fix  $p\in\mathscr{U}$. Then there exists an integral curve
 of length smaller than $\varepsilon/3$, connecting $p$ with a point $z_0\in\Sigma^*$. 
Set
 \begin{align*}\mathscr{E}\colon=\{z\in \Sigma^*\colon  z=g(z_0), g\in G\}.
 \end{align*} It follows from \eqref{esdensa} of Lemma \ref{densote} that $\mathscr{E}$ is dense in $\Sigma_{\rho}=\{z \in \Sigma:  \|z\|<\rho\}$.  By \mbox{Lemma \ref{assertion2}}, 
each $z\in\mathscr{E}$ is connected to $z_0$ by an integral curve of length  smaller than $\varepsilon/3$. 
Then, since $z_0$ in turn is connected to $p$ by an integral curve of length smaller than $\varepsilon/3$,  each
 $z\in \mathscr{E}$ is connected to $p$  by an integral curve of length smaller than $2\varepsilon/3$. Then,
  if  $E$ is defined as the set of points that can be connected to a point in $\mathscr{E}$ by an integral curve of length
 smaller than $\varepsilon/3$, we conclude  each point $q\in E$ is connected to $p$ by an integral curve $\beta_q$
 of length smaller than $\varepsilon\le\epsilon$. It is evident that  $E$ is contained   in the leaf through $p$: in fact, $q\in E$ is connected to $p$ by the integral curve $\beta_q$. Moreover, it follows directly from the definitions of $E$ and $\mathscr{U}$.
 that  $E\subset \mathscr{U}$.  Let us
 show that $E$ is dense in  $\mathscr{U}$.   Let $U\subset\mathscr{U}$ be an open set. Fix a point  $w \in U$. Then there exists an integral curve $\alpha$
 connecting $w$ with a point $z\in\Sigma^*$, such that $\ell(\alpha)<\varepsilon/3$.
 Since  $\Sigma^*$ is transverse to the foliation,  if $\Omega\subset U$ is a small enough neighborhood of $w$, there exists a submersion $h\colon \Omega\to\Sigma^*$ with $h(w)=z$, such that
 each point $w'\in\Omega$  can be connected to $h(w')\in \Sigma^*$ by an integral curve
  $\alpha'$ close enough 
 to $\alpha$ so that $\ell(\alpha')<\varepsilon/3$. Thus, since $\mathscr{E}$ is dense in $\Sigma^*$,  
 for a suitable choice of $w'\in\Omega$ we have  $h(w')\in\mathscr{E}$. Then  $w'\in\Omega\cap E$, so
${E}$ is dense in  $\mathscr{U}$. 
At this point, \eqref{con3a} of Proposition \ref{log} is almost proved:  it only remains to verify the following two facts.
\begin{enumerate}
\item $\mathscr{U}\subset\Delta$. Rigorously speaking this means that $\pi(\mathscr{U})\subset\Delta$.
\item For each $q\in E$, the curve  $\beta_q$ above is contained in $\Delta$.  Rigorously speaking: $\pi(\beta_q)$ is contained in $\Delta$. 
\end{enumerate}

Let $w$ be any point in $\mathscr{U}$. By the definition of $\mathscr{U}$ there exists an integral curve $\alpha$ connecting $w$ with a
    point $z\in\Sigma^*$, such that \mbox{$\ell (\alpha)<\varepsilon/3$}.  Thus, since $\ell (\alpha)$  is actually  the Euclidean length of
    $\pi(\alpha)$, we have  \begin{align}\nonumber|\pi(w) -\pi(z)|<\frac{\varepsilon}{3}.
  \end{align} Then
   \begin{align*}|\pi(w)|< |\pi(z)|+\frac{\varepsilon}{3}<\rho+\frac{\varepsilon}{3}<\varepsilon+\frac{\varepsilon}{3},
  \end{align*} so 
  \begin{align}\label{ff3}|\pi(w)|<\frac{4\varepsilon}{3}.
  \end{align}
 Thus, it follows from \eqref{gu1} that $\pi(w)\in\Delta$, therefore the first fact above is proved.

  Now, let $\zeta$ be any point in $\beta_q$. Then, since
   $\ell(\beta_q)<\varepsilon$, 
   \begin{align}\label{ff6}|\pi(\zeta)- \pi(p)|<{\varepsilon}.
  \end{align} Therefore,  since from inequality \eqref{ff3} we obtain
$|\pi(p)|<\frac{4\varepsilon}{3}$,  we find that
  \begin{align}\label{ff7}|\pi(\zeta)|<\frac{7\varepsilon}{3}.
 \end{align}
   Thus,  again by \eqref{gu1} we see that  $\pi(\zeta)\in\Delta$.  \qed \\

 We start with the proof of the two last assertions of Proposition \ref{log}. 

  \begin{lem}\label{pila1} Let $D\subset\mathbb{C}\backslash \{b_1,\dots,b_N\}$ be a closed disc. Then there exists 
  $\eta_D>0$ such that each point in
  $$\Omega_D\colon= \Big\{t\in\mathbb{C}^M\colon t_1\cdots t_{M-1}\neq 0, t_M\in D, \| t\|<\eta_D \Big\}$$
  is connected to a point in $\Sigma^*$ by an integral curve of length smaller that $\varepsilon/6$. In particular,
  $\Omega_D\subset\mathscr{U}.$
  \end{lem} 
  \proof
   For each $c\in D$,
  we can choose a smooth curve
 $$\gamma_c\colon [0,1]\to\mathbb{C}\backslash \{b_1,\dots,b_N\}, \quad \gamma_c(0)=a, \quad\gamma_c(1)=c,$$
 depending smoothly on the parameter $c\in D$. 
 From  Lemma \ref{assertion1} we have that there exists $K_{\gamma_c}>0$ depending continuously on $\gamma_c$, such that  
  \begin{align}\nonumber\ell (\gamma_c^z)< K_{\gamma_c}||z||, \quad z\in\Sigma^*. 
 \end{align}  Then, since $D$ is compact, we can find a positive number $\varrho<\rho$ such that  
 \begin{align}\label{copa1}\ell (\gamma_c^z)<\frac{\varepsilon}{6},\quad   z\in\Sigma^*,\; ||z||<\varrho,  \; c\in D. 
 \end{align} 
  For each $c\in D$, let $$\mathfrak{h}_c\colon \{t_M=c\}\to  \{t_M=a\}$$ be the holonomy associated to  $\gamma_c^{-1}$. Thus, given $w\in \{t_M=c\}$, we have $\mathfrak{h}_c (w)\in  \{t_M=a\}$ and the
the integral curve $\gamma_c^{\mathfrak{h}(w)}$ connects $\mathfrak{h}(w)$ with $w$.  Observe that, when
$w_1,\dots w_{M-1}\neq 0$ and the hypersurface $\{t_1\cdots t_{M-1}=0\}$ is invariant, we have
$\mathfrak{h}_c (w)\in \Sigma^*$. 
As we have seen before --- see \eqref{solu1} and \eqref{solu2} --- the holonomy $\mathfrak{h}_c$ is a linear map 
of the form
\begin{equation}\label{solu3}\mathfrak{h}_c(t_1,\dots,t_{M-1},c)= (\mathfrak{d}_1(c)t_1,
  \dots,\mathfrak{d}_{M-1}(c) t_{M-1}, a),\end{equation}
  where $$\mathfrak{d}_i(c)= \exp \int\limits_{\gamma_c^{-1}}A_i(u)du,\; i=1,\dots,M-1.$$
  The closure of the image of the hyperplane $ \{t_M=c\}$ by the blow-up map $\pi$ defines the linear subspace
   $$V_c\colon=\{(x_1,\dots,x_M)\in\mathbb{C}^M\colon x_M=cx_1\}.$$  In the same way, $ \{t_M=a\}$ defines
  the linear subspace 
   $$V_a\colon=\{(x_1,\dots,x_M)\in\mathbb{C}^M\colon x_M=ax_1\}.$$  It is easy to see from \eqref{solu3}
   that  $\pi\circ \mathfrak{h}_c\circ \pi^{-1}$ extends as a linear isomorphism 
   $$h_c\colon V_c\to V_a,$$
    which depends continuously on
   $c\in D$. 
 Then, since $D$ is compact, we can take  $\eta_D>0$  such that 
 \begin{align} \nonumber x\in V_c,\; |x|<\eta_D,\; c\in D \implies |h_c(x)|<\varrho.
 \end{align} Thus, applying this fact to $x=\pi (w)$ for any $w\in\Omega_D$ we obtain 
  \begin{align}  \|\mathfrak{h}_c (w)\|<\varrho,\quad w\in \Omega_D.
 \end{align} Moreover, since $w\in\Omega_D$ is outside the hypersurface $\{t_1\dots t_{M-1}=0\}$, which is invariant
 by the foliation, we deduce that $\mathfrak{h}_c (w)\in\Sigma^*$. Then  
 it follows from \eqref{copa1} that  the integral curve $\gamma_c^{\mathfrak{h}(w)}$, which  connects $w$ with $\mathfrak{h}(w)\in \Sigma^*$,
 has length smaller than $\varepsilon/6$. \qed \\

  For each $j=1,\dots,N$,  choose $r_j>0$ such that the disc 
$$\mathscr{D}_j\colon=\{u\in\mathbb{C}\colon |u-b_j|< r_j\}$$
intersects $\{b_1,\dots,b_N\}$ exactly at $b_j$. Moreover, let $R>0$  be such that  the disc 
\begin{align}\label{artesco}\mathscr{D}_R\colon =\{u\in\mathbb{C}\colon |u|< R\}\end{align} contains $\overline{\mathscr{D}_1}\cup\dots\cup\overline{\mathscr{D}_N}$. 
\begin{lem}\label{pila2} Set $$\mathscr{K}\colon= \overline{\mathscr{D}_R}- \mathscr{D}_1 \cup \dots \cup \mathscr{D}_N.$$
 There exists $\eta_\mathscr{K}>0$ such that 
each point in the set 
$$\Omega_\mathscr{K}\colon=\Big\{t\in\mathbb{C}^{M}\colon t_1\cdots t_{M-1}\neq 0, \; t_M\in \mathscr{K}, \;
 \|t\|<\eta_\mathscr{K} \Big\}$$ is connected to a point in $\Sigma^*$ by an integral curve of length smaller
 that $\varepsilon/6$. In particular, $\Omega_\mathscr{K}\subset \mathscr{U}$. 
 \proof  
 Since $\mathscr{K}$ is compact and contained in $\mathbb{C}\backslash \{b_1,\dots,b_N\}$, 
 we can cover $\mathscr{K}$ by finitely many compact discs $D_1,\dots, D_\mathfrak{n}$, $\mathfrak{n}\in\mathbb{N}$, each of them contained in $\mathbb{C}\backslash \{b_1,\dots,b_N\}$. From Lemma \ref{pila1} each point in the union of
 $\Omega_{D_1},\dots, \Omega_{D_\mathfrak{n}}$ is connected to a point in  
 $\Sigma^*$ by an integral curve of length smaller than $\varepsilon /6$.  Thus, since the choice
  $\eta_\mathscr{K}\colon =\min \{\eta_{D_1},\dots,\eta_{D_\mathfrak{n}}\}$ implies
  $$\Omega_\mathscr{K}\subset \Omega_{D_1}\cup \dots \cup \Omega_{D_\mathfrak{n}},$$
  the lemma follows. \qed \\

\end{lem}
 \begin{lem} \label{pila3}
Let  $l\in \{1,\dots,N\}$.  Suppose   there exists a real line through the origin in $\mathbb{C}$ separating
$1\in\mathbb{C}$ from the set $$\{\nu_{1l},\nu_{1l}+\nu_{2l},\dots, \nu_{1l}+\nu_{(M-1)l}\}.$$
Then there exists $\eta_l>0$ such that $\mathscr{U}$ contains the set
 $$\Omega_l\colon=\Big\{t\in\mathbb{C}^{M}\colon t_1\dots t_{M-1}\neq 0,\;  t_M\neq b_l, \; t_{M}\in \mathscr{D}_l,\; 
 \|t\|<\eta_l \Big\}.$$ In view of \eqref{nu4} of Lemma \ref{nuca} and Remark \ref{remarknu}, the conclusion of this lemma takes place if $b_l\notin T$. 
\end{lem}
\proof 
Recall that the vector field $Y$ defining $\mathcal{G}$
is given by
\begin{equation}\nonumber
Y=
\sum\limits_{i=1}^{M-1}A_i(t_M)t_i\frac{\partial}{\partial t_i}
+\frac{\partial}{\partial t_M},
\end{equation}
where \begin{equation}\nonumber
A_i(t_M)=
\sum\limits_{j=1}^N \frac{\nu_{ij}}{t_M-b_{j}}, \quad i=1,\dots M-1.
\end{equation}  Write
\begin{equation}\label{decomp1}
A_i(u)=\frac{\nu_{il}}{u-b_{l}}+ g_i(u),
\end{equation} 
where
\begin{equation}\nonumber
g_i(u)=
\sum\limits_{j=1,j\neq l}^N \frac{\nu_{ij}}{u-b_{j}}.
\end{equation}
Since $g_i$ is holomorphic on a neighborhood of  $\overline{\mathscr{D}_l}$ we can find  $K_1>1$ such that,
for each $i=1,\dots, M-1$, we have  
\begin{align}\label{cotak1}|g_i(u)|\le K_1,\quad u\in \overline{\mathscr{D}_l}, \quad \textrm{ and}
\end{align}
\begin{align}\label{percha}|A_i(u)|\le \psi(u)\colon = K_1+ \frac{K_1}{|u-b_l|},\; u\in \overline{\mathscr{D}_l}-\{b_l\}.
\end{align}
 It follows from  \eqref{nu4} of Lemma \ref{nuca} that there exists $\theta\in\mathbb{C}^*$ such that  
\begin{align}\label{tetacon} \begin{aligned}{}& \hspace{3.2cm} \re(\theta)>0,\quad \textrm{ and}\\
&\re \big(\theta \nu_{1l}\big),  \re \big(\theta[\nu_{1l}+\nu_{2l}]\big), \dots,  \re \big(\theta [\nu_{1l}+\nu_{(M-1)l}]\big) <0.\end{aligned}
\end{align} 
Set $\Gamma\colon=\Big \{1, \nu_{1l}, \big(\nu_{1l}+\nu_{2l}\big), \dots, \big(\nu_{1l}+\nu_{(M-1)l}\big)\Big\}$
and
\begin{align}\label{cotak2}
  K_2\colon =\max\limits_{\zeta\in \Gamma}\frac{2K_1 (r_l+1)|\theta|e^{\frac{2K_1r_l|\theta|}{\re \theta}}}{|\re(\zeta\theta)|},
   \end{align} 
 and take $\eta_l>0$ such that 
\begin{align}\label{crv}
\hspace{1cm}\eta_l(M+|b_l|+r_l)e^{\frac{2K_1r_l|\theta|}{\re \theta}}
<\eta_\mathscr{K} ,\;\textrm{ and }
\end{align}
\begin{align}\label{crv1}
\eta_l (M+|b_l|+r_l)K_2<\frac{\varepsilon}{6}.
\end{align} 
Now, fix a point $\sigma=(\sigma_1,\dots,\sigma_M)\in\Omega_l.$
Thus \begin{align}\label{borrador}
\|\sigma\|<\eta_l, \;\; 0<|\sigma_M-b_l|< r_l,\; \;\sigma_1\cdots \sigma_{M-1}\neq 0.
\end{align}
Set 
\begin{align}\label{tesi}\theta_\sigma\colon=\frac{\log \left(\frac{r_l}{|{\sigma_M-b_l}|}\right)\theta}{\re(\theta)}\end{align}
and consider the curve 
\begin{align}
\gamma(s)=b_l+(\sigma_{M}-b_l)e^{\theta_\sigma s}, \quad s\in[0,1].
\end{align}
It is easy to check that 
\begin{align}\label{control1}\gamma(0)={\sigma_M}, \quad  \gamma(1)\in{\partial \mathscr{D}_l}.\end{align}
As we have seen --- see \eqref{solu1} --- the curve 
\begin{equation}\gamma^\sigma (s)\colon= \left(\gamma_1(s),\dots, \gamma_{M-1}(s),\gamma(s) \right),\quad s\in[0,1],\end{equation}
where
 \begin{equation}\label{experto}\gamma_i(s)=\sigma_i\exp{\int\limits_{\gamma |_{[0,s]}}A_i(u)du },\quad  i=1,\dots,M-1, \end{equation}
  is an integral curve connecting $\sigma$ with 
  \begin{equation}\gamma^\sigma (1)=\left(\gamma_1(1),\dots, \gamma_{M-1}(1),\gamma(1) \right). \end{equation}
  Let us prove that $\gamma^\sigma(1)\in\Omega_\mathscr{K}$. Firstly, since 
  $\partial \mathscr{D}_l\subset \mathscr{K}$, it follows from \eqref{control1} that $\gamma(1)\in\mathscr{K}$. 
  Secondly,
   since the hyperplanes $$\{t_1=0\},\dots,\{t_{M-1}=0\}$$ are invariant by $\mathcal{G}$ and none of these hyperplanes contains $\sigma$, we 
conclude   $$\gamma^\sigma(1)\notin \{t_1\cdots t_{M-1}=0\}.$$
Thus it is enough to prove  that $\|\gamma^\sigma(1)\|<\eta_\mathscr{K}$.

 From \eqref{decomp1} we have
 
 \begin{align}\nonumber
\int\limits_{\gamma |_{[0,s]}}A_i(u)&=\int\limits_{\gamma |_{[0,s]}}\frac{\nu_{il}}{u-b_{l}}du+ 
\int\limits_{\gamma |_{[0,s]}}g_i(u)du\\
&=\int\limits_{0}^s\frac{\nu_{il}}{\gamma(\mathfrak{s})-b_{l}}\gamma'(\mathfrak{s})d\mathfrak{s}+ 
\int\limits_{0}^s g_i(\gamma(\mathfrak{s}))\gamma'(\mathfrak{s})d\mathfrak{s},
\end{align}
so that 
\begin{align}\label{asca}
\int\limits_{\gamma |_{[0,s]}}A_i(u)=(\nu_{il}\theta_\sigma)s+ 
\int\limits_{0}^s g_i(\gamma(\mathfrak{s}))\gamma'(\mathfrak{s})d\mathfrak{s}.
\end{align}
Let us show that the function $$B_i(s)\colon=\int\limits_{0}^s g_i(\gamma(\mathfrak{s}))\gamma'(\mathfrak{s})d\mathfrak{s}$$
has a bound independent of  $\sigma$. From \eqref{cotak1},
$$|B_i(s)|\le K_1\int\limits_{0}^s|\gamma'(\mathfrak{s})|d\mathfrak{s},$$ so it suffices to show that the length
 $\ell(\gamma)=\int\limits_{0}^1|\gamma'(\mathfrak{s})|d\mathfrak{s}$, has a bound independent of $\sigma$. 
 We have
\begin{align}\int\limits_{0}^1 |\gamma'(\mathfrak{s})|d\mathfrak{s} &= \int\limits_{0}^1 \left|\theta_\sigma{(\sigma_M-b_l)}
e^{\theta_\sigma\mathfrak{s}}\right|d\mathfrak{s}
= \int\limits_{0}^1 |\theta_\sigma {(\sigma_M-b_l)}|
e^{\re(\theta_\sigma)\mathfrak{s}}d\mathfrak{s},\\
&{}\hspace{-0.2cm} \int\limits_{0}^1 |\gamma'(\mathfrak{s})|d\mathfrak{s} =  \frac{|\theta_\sigma|| {\sigma_M-b_l}|
}{\re \theta_\sigma}\cdot (e^{\re(\theta_\sigma)}-1)
\end{align}
Thus, since from \eqref{tesi} $$|\theta_\sigma|=\frac{|\theta|}{|\re \theta|}\log\frac{r_l}{|{\sigma_M-b_l}|}\textrm{  and }
 \re \theta_\sigma=\log \frac{r_l}{|{\sigma_M-b_l}|},$$ we have 
\begin{align}\int\limits_{0}^1 |\gamma'(\mathfrak{s})|d\mathfrak{s} & =
\frac{|\theta|}{|\re \theta|}|\sigma_M-b_l|\left(\frac{r_l}{|\sigma_M-b_l|}-1\right)
\end{align}
so that 
\begin{align}\label{longama}\int\limits_{0}^1 |\gamma'(\mathfrak{s})|d\mathfrak{s}<\frac{r_l|\theta|}{|\re \theta|}. 
\end{align} Therefore 
\begin{align}|B_i(s)|\le \frac{K_1 r_l|\theta|}{|\re \theta|}, \quad  i=1,\dots, M-1
\end{align}
and from \eqref{asca} we obtain
\begin{align}\label{asca1}
\re\int\limits_{\gamma |_{[0,s]}}A_i(u)\le \re(\nu_{il}\theta_\sigma)s + \frac{K_1r_l|\theta|}{|\re \theta|}.
\end{align}
Then, since 
\begin{align}
 |\gamma_i(s)|=\left|\sigma_i\exp{\int\limits_{\gamma |_{[0,s]}}A_i(u)du } \right|
  =
 |\sigma_i|\exp{\re\int\limits_{\gamma |_{[0,s]}}A_i(u)du },
\end{align}
we obtain 
\begin{align}\label{tacho}
 |\gamma_i(s)|
 \le |\sigma_i|e^{\frac{K_1r_l|\theta|}{|\re \theta|}}e^{\re(\nu_{il}\theta_\sigma)s}.
\end{align}
On the other hand, 
from 
  \eqref{tetacon}, \eqref{tesi} and the inequality $|{\sigma_M-b_l}|<r_l$ we obtain that  
  \begin{align}\label{negativos}\re \big( \nu_{1l}\theta_\sigma\big),  \re \big([\nu_{1l}+\nu_{2l}]\theta_\sigma\big), \dots, 
   \re \big( [\nu_{1l}+\nu_{(M-1)l}]\theta_\sigma\big) <0.  \end{align}
  Then, from this and \eqref{tacho} we have
  \begin{gather}\label{beoplay}
 \begin{aligned}
 |\gamma_1|
 &\le |\sigma_1|e^{\frac{K_1r_l|\theta|}{|\re \theta|}}e^{\re(\nu_{1l}\theta_\sigma)s}\le |\sigma_1|e^{\frac{K_1r_l|\theta|}{|\re \theta|}},
 \\
|\gamma_1\gamma_2| &\le  |\sigma_1\sigma_2|e^{2\frac{K_1r_l|\theta|}{|\re \theta|}}
e^{\re\left([\nu_{1l}+\nu_{2l}]\theta_\sigma\right)s}
\le |\sigma_1\sigma_2|e^{2\frac{K_1r_l|\theta|}{|\re \theta|}},\\
&\hspace{2.5mm}\vdots \\
|\gamma_1\gamma_{M-1}| &\le  |\sigma_1\sigma_{M-1}|e^{2\frac{K_1r_l|\theta|}{|\re \theta|}}
e^{\re\left([\nu_{1l}+\nu_{(M-1)l}]\theta_\sigma\right)s}
\le |\sigma_1\sigma_{M-1}|e^{2\frac{K_1r_l|\theta|}{|\re \theta|}}.
\end{aligned}
\end{gather}
Therefore
\begin{align*}
\|\gamma^\sigma(1))\|&\le |\gamma_1(1)|+|\gamma_1(1)\gamma_2(1)|+\dots +|\gamma_1(1)\gamma_{M-1}(1)|+
|\gamma_1(1)\gamma (1)|
\\
&\le |\sigma_1|e^{\frac{K_1r_l|\theta|}{|\re \theta|}}+
|\sigma_1\sigma_2|e^{\frac{2K_1r_l|\theta|}{|\re \theta|}} +\dots 
\\
& \hspace{2.55cm} \dots +|\sigma_1\sigma_{M-1}|e^{\frac{2K_1r_l|\theta|}{|\re \theta|}} +
|\sigma_1|e^{\frac{K_1r_l|\theta|}{|\re \theta|}}(|b_l|+r_l)
\\
&\le \|\sigma\|\left(e^{\frac{K_1r_l|\theta|}{|\re \theta|}}+(M-2)e^{\frac{2K_1r_l|\theta|}{|\re \theta|}}+e^{\frac{K_1r_l|\theta|}{|\re \theta|}}(|b_l|+r_l)\right)
\\
&\le \|\sigma\|(M+|b_l| +r_l)e^{\frac{2K_1r_l|\theta|}{|\re \theta|}},
\end{align*} so that, by \eqref{borrador} and \eqref{crv}, 
$$\|\gamma^\sigma(1)\|<\eta_\mathscr{K},$$ which  shows that $\gamma^\sigma(1)\in\Omega_\mathscr{K}$. 
Since Lemma \ref{pila2} guarantees that $\gamma^\sigma(1)$ is connected to a point in $\Sigma^*$ 
by an integral curve of length smaller than $\varepsilon /6$, and $\gamma^\sigma$ connects $\sigma$ with 
$\gamma^\sigma(1)$, Lemma \ref{pila3} will be proved if we show that 
$\ell(\gamma^\sigma)<\varepsilon/6$. To see this we write
 \begin{align} \label{el1}\begin{aligned}
\ell (\gamma^\sigma)=\int\limits_{0}^{1}|(\pi\circ\gamma^\sigma)'(s)|ds  
\le \int\limits_{0}^{1}|\gamma_1'|ds+\int\limits_{0}^{1}|(\gamma_1\gamma_2)'|ds+\dots \\
\dots +
\int\limits_{0}^{1}|(\gamma_1\gamma_{M-1})'|ds+
\int\limits_{0}^{1}|(\gamma_1\gamma)')|ds.
\end{aligned}
\end{align}
Let us find bounds for each of the integrals above.
From \eqref{experto} and \eqref{percha}, we have 
\begin{gather}\label{lunes}
\begin{aligned}
 \int\limits_{0}^{1}|\gamma_1'|ds=\int\limits_{0}^{1}  |A_1(\gamma)\gamma'\gamma_1|ds\le
  \int\limits_{0}^{1}  |\psi(\gamma)\gamma'\gamma_1|ds. 
 \end{aligned}
\end{gather}
Observe that, since $\re\theta_\sigma=\log\frac{r_l}{|{\sigma_M-b_l}|}>0$,
\begin{align*}
|\psi (\gamma)\gamma'|&= \left(K_1  +\frac{K_1}{|\gamma-b_l|}\right)    |\gamma' |
\\
&=  \left(K_1+\frac{K_1}{ |{\sigma_M-b_l}|e^{\re \theta_\sigma s}}\right)   | \theta_\sigma| |{\sigma_M-b_l}|e^{\re \theta_\sigma s}
\\
&\le \Big(K_1|{\sigma_M-b_l}|e^{\re \theta_\sigma s}+K_1\Big) |\theta_\sigma|
\\
& \le \Big(K_1|{\sigma_M-b_l}|e^{\log\frac{r_l}{|{\sigma_M-b_l}|} }+K_1\Big) |\theta_\sigma|,
  \end{align*}
  so 
  \begin{align}\label{psi} |\psi (\gamma)\gamma'|  \le K_1(r_l+1)|\theta_\sigma|.
  \end{align}
Moreover, from \eqref{negativos} we have $\re(\nu_{1l}\theta_\sigma)<0$ and it follows from \eqref{tesi}   that 
$$\frac{|\theta_\sigma|}{|\re \nu_{1l}\theta_\sigma|}=\frac{|\theta|}{|\re \nu_{1l}\theta|}.
$$
Using these facts together with \eqref{psi} and  \eqref{tacho}
 in \eqref{lunes}, we obtain
  \begin{gather}\nonumber
\begin{aligned}
 \int\limits_{0}^{1}|\gamma_1'|ds\le
  \int\limits_{0}^{1}  |\psi(\gamma)\gamma'\gamma_1|ds\le K_1(r_l+1)|\theta_\sigma| \int\limits_{0}^{1} |\sigma_1|e^{\frac{K_1r_l|\theta|}{|\re \theta|}}e^{\re(\nu_{1l}\theta_\sigma)s}\\
 =\frac{ K_1(r_l+1)|\theta_\sigma||\sigma_1|e^{\frac{K_1r_l|\theta|}{|\re \theta|}}}{|\re(\nu_{1l}\theta_\sigma)|}\left({1-e^{\re(\nu_{1l}\theta_\sigma)}}\right) \\
 \le  
 \frac{ K_1(r_l+1)|\theta| |\sigma_1| e^{\frac{K_1r_l|\theta|}{|\re \theta|}}}{|\re(\nu_{1l}\theta)|},
 \end{aligned}
\end{gather}
  so that, 
  \begin{align}\label{martes} \int\limits_{0}^{1}|\gamma_1'|ds\le K_2 \|\sigma\|,
  \end{align} 
  were $K_2$ is as defined in \eqref{cotak2}.

Now, for each $i=2,\dots, M-1$, from \eqref{beoplay} we have 
\begin{align}
|\gamma_1\gamma_i| &\le  \|\sigma\|e^{2\frac{K_1r_l|\theta|}{|\re \theta|}}
e^{\re\left([\nu_{1l}+\nu_{2l}]\theta_\sigma\right)s}.
\end{align}
 This together with \eqref{experto} and \eqref{psi} leads to
 \begin{align} 
 \nonumber \int\limits_{0}^{1}|(\gamma_1\gamma_i)'|ds &= \int\limits_{0}^{1}|\gamma_1'\gamma_i+\gamma_i'\gamma_1|ds
 =
  \int\limits_{0}^{1}\Big|A_1(\gamma)\gamma'\gamma_1\gamma_i+A_i(\gamma)\gamma'\gamma_i\gamma_1\Big|ds
  \\
  \nonumber  & \le \int\limits_{0}^{1}2|\psi(\gamma)\gamma'\gamma_1\gamma_i|ds
  \\
\nonumber   &\le \int\limits_{0}^{1}2K_1(r_l+1)|\theta_\sigma|\|\sigma\| e^{2\frac{K_1r_l|\theta|}{|\re \theta|}}
e^{\re\left([\nu_{1l}+\nu_{2l}]\theta_\sigma\right)s}  ds 
\\ 
\label{tecla}
 &\le2 K_1(r_l+1)|\theta_\sigma|\|\sigma\| e^{2\frac{K_1r_l|\theta|}{|\re \theta|}}
\frac{e^{\re\left([\nu_{1l}+\nu_{2l}]\theta_\sigma\right)}-1 }{\re\left([\nu_{1l}+\nu_{2l}]\theta_\sigma\right)}.
 \end{align}
Again by \eqref{negativos} and \eqref{tesi},  we see that  $\re\left([\nu_{1l}+\nu_{2l}]\theta_\sigma\right)<0$ and 
$$\frac{|\theta_\sigma|}{|\re\left([\nu_{1l}+\nu_{2l}]\theta_\sigma\right)|}=\frac{|\theta|}{|\re\left([\nu_{1l}+\nu_{2l}]\theta\right)|}.
$$
Thus, from \eqref{tecla} we obtain

 \begin{align}
 \int\limits_{0}^{1}|(\gamma_1\gamma_i)'|ds
 & \le \frac{2K_1(r_l+1)|\theta_\sigma|\|\sigma\| e^{2\frac{K_1r_l|\theta|}{|\re \theta|}}}  {  |\re\left([\nu_{1l}+\nu_{2l}]\theta_\sigma\right)| }   
\left({1- e^{\re\left([\nu_{1l}+\nu_{2l}]\theta_\sigma\right)s} } \right)
\\
& \le \frac{2K_1(r_l+1)|\theta|\|\sigma\| e^{2\frac{K_1r_l|\theta|}{|\re \theta|}}}  {  |\re\left([\nu_{1l}+\nu_{2l}]\theta\right)| },   
 \end{align}
so that 
\begin{align}\label{miercoles}
 \int\limits_{0}^{1}|(\gamma_1\gamma_i)'|ds
 \le K_2\|\sigma\| ,\quad i=2,\dots, M-1.
 \end{align}

On the other hand,
\begin{gather}\nonumber
\begin{aligned}
\int|(\gamma_1\gamma)'|ds=\int |\gamma_1'\gamma+\gamma_1\gamma'|ds\le \int|\gamma_1'||\gamma|ds+\int|\gamma_1||\gamma'|ds.
\end{aligned}
\end{gather}

Then, from \eqref{beoplay},  \eqref{martes} and \eqref{longama}  we have 
\begin{gather}\nonumber
\begin{aligned}
\int|(\gamma_1\gamma)'|ds &\le  \int_0^1|\gamma_1'|(|b_l|+r_l) ds+\int_0^1 e^{\frac{K_1r_l|\theta|}{|\re \theta|}} \|\sigma\| |\gamma'|ds
\\
&  \le  K_2\|\sigma\|(|b_l|+r_l) +  e^{\frac{K_1r_l|\theta|}{|\re \theta|}}  \|\sigma\| \int_0^1 |\gamma'|ds
\\
  & \le  K_2\|\sigma\| (|b_l|+r_l) +  e^{\frac{K_1r_l|\theta|}{|\re \theta|}} \frac{r_l|\theta|}{|\re \theta|}\|\sigma\|.
\end{aligned}
\end{gather}
Thus, since  the inequality $K_1>1$ implies
$$e^{\frac{K_1r_l|\theta|}{|\re \theta|}} \frac{r_l|\theta|}{|\re \theta|}\le K_2,$$
we obtain 
\begin{align}\label{canasta3}
\int|(\gamma_1\gamma)'|ds
\le K_2(|b_l|+r_l+1) \|\sigma\|.
\end{align}
Using  \eqref{canasta3}, \eqref{miercoles}  and \eqref{martes}  in \eqref{el1}, we obtain  
 \begin{align*}
\ell (\gamma^\sigma) &\le \int|\gamma_1'|ds+\int|(\gamma_1\gamma_2)'|ds+\dots +\int|(\gamma_1\gamma_{M-1})'|ds+
\int|(\gamma_1\gamma)')|ds\\
&\le K_2\|\sigma\|+\dots +K_2\|\sigma\| +K_2(|b_l|+r_l+1) \|\sigma\|
\\
 &\le K_2(M+|b_l|+r_l)\|\sigma\|
\end{align*} and it follows from \eqref{borrador} and  \eqref{crv1} that
\begin{align*}
\ell (\gamma^\sigma)<\frac{\varepsilon}{6}.
\end{align*}
Lemma \ref{pila3} is proved. \qed \\

 Now, we want to establish a result analogous to Lemma \ref{pila3}, but near  $t_M=\infty$. To do so, 
 we consider coordinates $(s_1,\dots, s_M)$ in $\pi^{-1}(\mathbb{C}^M)$  such that  
$$t=\phi(s)\colon =(s_1s_M,\frac{s_2}{s_M},\dots, \frac{s_{M-1}}{s_M},\frac{1}{s_M}).$$
Note that 
$$\pi(\phi(s))=(s_1s_M,s_1{s_2},\dots, s_1s_{M-1},s_1).$$
Thus, if we define \begin{align}
 \label{futon1}\tilde{\pi}\colon (s_1,\dots, s_m)\mapsto (\tilde{x}_1,\dots,\tilde{x}_M)= (s_1,s_1s_2,\dots,s_1s_M),
 \end{align} we can see that 
 $$|\tilde{\pi}(s)|=|\pi(\phi(s))|.$$ Therefore it will be consistent to define 
 $$\|s\|\colon=|(s_1,s_1s_2,\dots, s_1 s_M)|.$$

  Let $R$ be  as before --- see \eqref{artesco} ---  and consider the disc 
 
$$\tilde{\mathscr{D}}\colon =\{s_M\in \mathbb{C}\colon |s_M|\le R\}.$$ 
 
 \begin{lem}\label{pila4}There exists $\tilde{\eta}>0$ such that $\mathscr{U}$ contains  the set 
$$\tilde{\Omega}\colon=\Big\{s\in\mathbb{C}^M\colon s_1\cdots s_{M-1}s_M\neq 0,\; s_M\in \tilde{\mathscr{D}}, \;
\|s\|<\tilde{\eta}\Big\}.$$ 
\end{lem}
\begin{rem}\label{rempila} Note that in the coordinates $(t_1,\dots,t_M)$ the set $\tilde{\Omega}$ is given by 
$$\tilde{\Omega}=\Big\{t\in\mathbb{C}^M\colon t_1\cdots t_{M-1}\neq 0,\; |t_M|>R, \;
\|t\|<\tilde{\eta}\Big\}.$$ 
\end{rem}
 
\proof 
We can check that $\pi\circ \phi \circ \tilde{\pi}^{-1} $  is given by
\begin{align} 
\label{futon2}(\tilde{x}_1, \dots, \tilde{x}_M)\mapsto (x_1,\dots,x_M)=(\tilde{x}_M,\tilde{ x}_2,\dots,\tilde{x}_{M-1}, \tilde{x}_1),\end{align}
which is an  Euclidean isometry. This means that, given a curve $\alpha$  in the coordinates $(s_1,\dots, s_M)$,
the Euclidean length of $\tilde{\pi}(\alpha)$ coincides with the Euclidean length of $\pi(\phi(\alpha))$, that is,
we can define $\ell (\alpha)$ as the length of  $\tilde{\pi}(\alpha)$ and we will have $\ell (\alpha)=\ell(\phi(\alpha))$. 
Thus, we are able to work in the coordinates $(s_1,\dots, s_M)$  following the same steps of the proof of Lemma \ref{pila3}.
Let us do that. 
In the coordinates $(s_1,\dots, s_M)$, the foliation $\mathcal{G}$ is generated by the rational vector field
$$\tilde{Y}\colon=-\frac{1}{s_M^2}\phi^*(Y),$$ which --- by a straightforward computation ---  is expressed in the form
\begin{equation}\nonumber
\tilde{Y}=
\sum\limits_{i=1}^{M-1}\tilde{A}_i(s_M)s_i\frac{\partial}{\partial s_i}
+\frac{\partial}{\partial s_M},
\end{equation}
where 

\begin{align*}
\tilde{A}_i(s_M)&= \left(\frac{\tilde{\nu}_i}{s_M}+\tilde{g}_i(s_M)  \right),\quad i=1,\dots, M-1,\\
\tilde{\nu}_1& =-1-\sum\limits_{j=1}^N \nu_{ij},\\
\tilde{\nu}_2& =1-\sum\limits_{j=1}^N \nu_{2j},\\
&\hspace{2.3mm}\vdots\\
\tilde{\nu}_{M-1}&=1-\sum\limits_{j=1}^N \nu_{(M-1)j},\\
\end{align*}
and $\tilde{g}_1,\dots,\tilde{g}_{M-1}$ are holomorphic on a neighborhood $\tilde{\mathscr{D}}$. 
Now,  we choose  $\tilde{K}_1>1$ such that,
for each $i=1,\dots, M-1$, we have  
\begin{align}\nonumber\hspace{0.5cm}|\tilde{g}_i(u)|\le \tilde{K}_1,\quad  u\in {\tilde{\mathscr{D}}}, \;\textrm{ and}
\end{align}
\begin{align}\nonumber \hspace{0.5cm}|\tilde{A}_i(u)|\le \tilde{\psi}(u)\colon = \tilde{K}_1+ \frac{\tilde{K}_1}{|u|},\quad  u\in {\tilde{\mathscr{D}}}.
\end{align}
From  \eqref{nu5} of Lemma \ref{nuca}, we can find $\tilde{\theta}\in\mathbb{C}$ such that  
\begin{align}\nonumber \begin{aligned}{}& \hspace{3.2cm} \re(\tilde{\theta})>0,\quad \textrm{ and}\\
&\re\big({\tilde{\theta}}\tilde{\nu}_1\big),  \re \big({\tilde{\theta}}[\tilde{\nu}_{1}+\tilde{\nu}_{2}]\big), \dots,  
 \re \big({\tilde{\theta}}[\tilde{\nu}_{1}+\tilde{\nu}_{M-1}]\big)<0.\end{aligned}
\end{align} 
Set $$\tilde{\Gamma}\colon= \{1, \tilde{\nu}_{1}, \big(\tilde{\nu}_{1}+\tilde{\nu}_{2}\big), \dots, \big(\tilde{\nu}_{1}+\tilde{\nu}_{(M-1)}\big)\}$$
and
\begin{align}\nonumber
  \tilde{K}_2\colon =\max\limits_{\zeta\in \tilde{\Gamma}}\frac{2\tilde{K}_1 (R+1)|\tilde{\theta}|e^{\frac{2\tilde{K}_1 R|\tilde{\theta}|}{\re \tilde{\theta}}}}{|\re(\zeta\tilde{\theta})|},
   \end{align} 
 and take $\tilde{\eta}>0$ such that 
\begin{align}\nonumber
\hspace{1cm}\tilde{\eta}(M+R)e^{\frac{2\tilde{K}_1R|\tilde{{\theta}}|}{\re \tilde{\theta}}}
<\eta_\mathscr{K} ,\;\textrm{ and }
\end{align}
\begin{align}\nonumber
\tilde{\eta} (M+R)\tilde{K}_2<\frac{\epsilon}{6}.
\end{align} 
From this point forward, the proof follows exactly as in Lemma \ref{pila3}, so we leave  the details to the reader.  \qed\\

\noindent\emph{Proof of \eqref{con3b} of Proposition \ref{log}}. 
Let $$\zeta\colon =[\mathfrak{x}_1:\dots : \mathfrak{x}_M]$$ be a point in the exceptional divisor  $\pi^{-1}(0)$  outside the tangent cone of $$\prod\limits_{\tau\in T}(x_M-\tau x_1)=0.$$ Suppose first that $\mathfrak{x}_1\neq 0$. Then the point $\zeta$ is given in the coordinates $(t_1,\dots, t_M)$
as
$$\zeta=\left(0, \frac{\mathfrak{x}_2}{\mathfrak{x}_1},\dots, \frac{\mathfrak{x}_M}{\mathfrak{x}_1}\right).$$
If $\frac{\mathfrak{x}_M}{\mathfrak{x}_1}\notin \{ b_1,\dots ,b_N\}$, 
we take a closed disc  $$D\subset \mathbb{C}\backslash\{b_1,\dots,b_N\}$$ centered at $ \frac{\mathfrak{x}_M}{\mathfrak{x}_1}$ and so, by Lemma
\ref{pila1},  the corresponding set $\Omega_D$ is contained in $\mathscr{U}$. Recall that we are identifying the sets $\mathscr{U}$ and $\pi^{-1}(\mathscr{U})$. Since in the statement of Proposition \ref{log} the notation $\mathscr{U}$ refers to the set before the blow-up, rigorously speaking
Lemma \ref{pila1} shows that $\Omega_D\subset \pi^{-1}(\mathscr{U})$.   Therefore, the set 
$\Omega_D \cup \pi^{-1}({S})$, which is a neighborhood of $\zeta$, is contained in  $\pi^{-1}(\mathscr{U}\cup S)$. 
If $\frac{\mathfrak{x}_M}{\mathfrak{x}_1}=b_l$ for some $l=1,\dots,N$ with $b_l\notin T$,  by Lemma \ref{pila3} we have 
\mbox{$\Omega_l\subset \pi^{-1}(\mathscr{U})$.} Therefore the set 
$\Omega_l \cup \pi^{-1}({S})$, which is a neighborhood of $\zeta$, is contained in  $\pi^{-1}(\mathscr{U}\cup S)$.

Suppose now that $\mathfrak{x}_1=0$. In this case $\mathfrak{x}_M\neq 0$, otherwise $\zeta$ belongs to the tangent cone above. 
Then, by \eqref{futon1} and \eqref{futon2} the point $\zeta$ is given in the coordinates $(s_1,\dots,s_M)$ as
$$\zeta =(0, \frac{\mathfrak{x}_2}{\mathfrak{x}_M},\dots, \frac{\mathfrak{x}_{M-1}}{\mathfrak{x}_M}, 
\frac{\mathfrak{x}_1}{\mathfrak{x}_M}).$$
Thus, by Lemma
\ref{pila4} we have $\tilde{\Omega}\subset \pi^{-1}(\mathscr{U})$. Therefore the set 
$\tilde{\Omega} \cup \pi^{-1}({S})$, which is a neighborhood of $\zeta$, is contained in  $\pi^{-1}(\mathscr{U}\cup S)$. \qed\\

\noindent\emph{Proof of \eqref{con3c} of Proposition \ref{log}.} 
Since $\tau\in T$ we have $\tau=b_l$ for some $l=1,\dots,m$. Thus, it follows from Lemma \ref{pila3} that 
$\Omega_l\subset \pi^{-1}(\mathscr{U})$. Therefore the set 
$\Omega_l \cup \pi^{-1}({S})$, which is contained in  $\pi^{-1}(\mathscr{U}\cup S)$, is a neighborhood of any point in
$\{t_1=0,t_M=\tau\}$. If the property about $\tau$ takes place for all $\tau\in T$, then the conclusion of
Lemma \ref{pila3} holds for each  $l=1,\dots,N$. This together with Lemma \ref{pila2} and Lemma \ref{pila4} leads to the inclusion
$$\Omega_1\cup\dots\cup \Omega_N\cup\Omega_\mathscr{K}\cup\tilde{\Omega}\subset \pi^{-1}(\mathscr{U}).$$ 
But if we take $$\eta\colon=\min\{ \eta_1,\dots,\eta_N, \eta_\mathscr{K},\tilde{\eta}\},$$
it follows from the definitions of the sets $\Omega_1,\dots, \Omega_N,\Omega_\mathscr{K},\tilde{\Omega}$ and Remark \ref{rempila} that 
$$\Omega_1\cup\dots\cup \Omega_N\cup\Omega_\mathscr{K}\cup\tilde{\Omega}$$ contains
the set 
 $$\Omega\colon=\Big\{t\in\mathbb{C}^{M}\colon t_1,\dots,t_{M-1}\neq 0, \; t_M\neq b_l, \;l=1,\dots,N, \; 
 \|t\|<\eta \Big\},$$
whence $\Omega\subset  \pi^{-1}(\mathscr{U})$. Therefore the set $\Omega\cup\pi^{-1}(S)$, which is a neighborhood of the exceptional divisor, is contained in $\pi^{-1}(\mathscr{U}\cup S)$. The proof of Proposition 
\ref{log} is finished. \qed

\end{document}